\documentclass[12pt]{article}

\usepackage[T2A,T1]{fontenc}
\usepackage[utf8]{inputenc}
\usepackage[russian,english]{babel}

\usepackage{epstopdf}
\usepackage{amsmath}
\usepackage{amsthm}
\usepackage{amssymb}
\usepackage{amsfonts}
\usepackage{mathrsfs,bbm}
\usepackage[all,2cell]{xy}
\usepackage{color}
\usepackage{pdfcolmk}
\usepackage{comment}

\numberwithin{equation}{subsection}

\theoremstyle{plain}
\newtheorem{theorem}[subsubsection]{Theorem}
\newtheorem{proposition}[subsubsection]{Proposition}
\newtheorem{lemma}[subsubsection]{Lemma}
\newtheorem{corollary}[subsubsection]{Corollary}

\theoremstyle{definition}
\newtheorem{definition}[subsubsection]{Definition}

\theoremstyle{remark}
\newtheorem{example}[subsubsection]{Example}
\newtheorem{remark}[subsubsection]{Remark}

\makeatletter
\newcommand*{\relrelbarsep}{.386ex}
\newcommand*{\relrelbar}{%
  \mathrel{%
    \mathpalette\@relrelbar\relrelbarsep
  }%
}
\newcommand*{\@relrelbar}[2]{%
  \raise#2\hbox to 0pt{$\m@th#1\relbar$\hss}%
  \lower#2\hbox{$\m@th#1\relbar$}%
}
\providecommand*{\rightrightarrowsfill@}{%
  \arrowfill@\relrelbar\relrelbar\rightrightarrows
}
\providecommand*{\leftleftarrowsfill@}{%
  \arrowfill@\leftleftarrows\relrelbar\relrelbar
}
\providecommand*{\xrightrightarrows}[2][]{%
  \ext@arrow 0359\rightrightarrowsfill@{#1}{#2}%
}
\providecommand*{\xleftleftarrows}[2][]{%
  \ext@arrow 3095\leftleftarrowsfill@{#1}{#2}%
}
\makeatother

\renewcommand{\1}{\mathbbm{1}}

\DeclareFontFamily{OT1}{wncyi}{}
\DeclareFontShape{OT1}{wncyi}{m}{it}{
<5> <6> <7> <8> <9> gen * wncyi
<10> <10.95> <12> <14.4> <17.28> <20.74> <24.88> wncyi10
}{}
\DeclareSymbolFont{cyrletters}{OT1}{wncyi}{m}{it}
\DeclareSymbolFontAlphabet{\cyrmath}{cyrletters}
\DeclareMathSymbol{\rE}{\cyrmath}{cyrletters}{003}
\DeclareMathSymbol{\rD}{\cyrmath}{cyrletters}{068}
\DeclareMathSymbol{\rG}{\cyrmath}{cyrletters}{017}
\DeclareMathSymbol{\rI}{\cyrmath}{cyrletters}{073}
\DeclareMathSymbol{\rL}{\cyrmath}{cyrletters}{076}
\DeclareMathSymbol{\rZ}{\cyrmath}{cyrletters}{090}

\renewcommand{\phi}{\varphi}

\newcommand{\Oc}{\mathcal{O}}

\newcommand{\Dc}{\mathcal{D}}
\newcommand{\Bc}{\mathcal{B}}
\newcommand{\Mc}{\mathcal{M}}

\newcommand{\Vc}{\mathcal{V}}

\newcommand{\Mrm}{\mathrm{M}}
\newcommand{\Drm}{\mathrm{D}}

\renewcommand{\O}{\mathrm{O}}

\newcommand{\B}{\mathbb{B}}

\newcommand{\Mfk}{\mathfrak{M}}
\newcommand{\Sc}{\mathcal{S}}
\newcommand{\Kc}{\mathcal{K}}

\newcommand{\Ec}{\mathcal{E}}

\newcommand{\cof}{\mathrm{cof}}

\newcommand{\Obf}{\mathbf{O}}
\newcommand{\Cbf}{\mathbf{C}}
\newcommand{\Rbf}{\mathbf{R}}
\newcommand{\Kbf}{\mathbf{K}}
\newcommand{\KMbf}{\mathbf{KM}}
\newcommand{\HCbf}{\mathbf{HC}}

\newcommand{\Lbf}{\mathbf{L}}
\newcommand{\Zbf}{\mathbf{Z}}
\newcommand{\Abf}{\mathbf{A}}

\newcommand{\Qbf}{\mathbf{Q}}

\newcommand{\Hbf}{\mathbf{H}}

\newcommand{\F}{\mathbb{F}}

\newcommand{\Rc}{\mathcal{R}}
\newcommand{\Gc}{\mathcal{G}}
\newcommand{\Xc}{\mathcal{X}}
\newcommand{\Cc}{\mathcal{C}}

\newcommand{\Uc}{\mathcal{U}}

\newcommand{\Ac}{\mathcal{A}}

\newcommand{\Db}{\mathbb{D}}
\newcommand{\A}{\mathbb{A}}
\newcommand{\Lb}{\mathbb{L}}
\renewcommand{\C}{\mathbb{C}}
\newcommand{\R}{\mathbb{R}}

\newcommand{\Q}{{\mathbb{Q}}}

\newcommand{\Z}{\mathbb{Z}}

\newcommand{\N}{\mathbb{N}}

\newcommand{\G}{\mathbb{G}}
\newcommand{\Sb}{\mathbb{S}}

\newcommand{\Pb}{\mathbb{P}}
\newcommand{\Mb}{\mathbb{M}}

\newcommand{\Tc}{\mathcal{T}}

\newcommand{\Sp}{\textsc{Sp}}

\newcommand{\QCoh}{\textsc{QCoh}}
\newcommand{\DQCoh}{\textsc{DQCoh}}

\newcommand{\Spec}{\mathrm{Spec}}
\newcommand{\Core}{\mathrm{Core}}
\newcommand{\Gap}{\mathrm{Gap}}
\newcommand{\Arr}{\textsc{Arr}}

\newcommand{\HH}{\mathrm{HH}}

\newcommand{\KM}{\mathrm{KM}}

\newcommand{\Rings}{\textsc{Rings}}
\newcommand{\CatRig}{\textsc{CatRig}}
\newcommand{\CatAnRig}{\textsc{CatAnRig}}

\newcommand{\coker}{\mathrm{coker}}

\newcommand{\Mod}{\textsc{Mod}}
\newcommand{\Shv}{\textsc{Shv}}
\newcommand{\Perf}{\textsc{Perf}}
\newcommand{\PreShv}{\textsc{PreShv}}
\newcommand{\AnSm}{\textsc{AnSm}}
\newcommand{\ProjAn}{\textsc{ProjAn}}

\newcommand{\An}{\textsc{An}}
\newcommand{\Anrm}{\mathrm{An}}

\newcommand{\SH}{\textsc{SH}}
\newcommand{\RigSH}{\mathrm{RigSH}}

\newcommand{\ind}{\textrm{ind-}}
\newcommand{\pro}{\textrm{pro-}}

\newcommand{\Rat}{\mathrm{Rat}}

\newcommand{\Hom}{\mathrm{Hom}}
\newcommand{\Stab}{\mathrm{Stab}}
\newcommand{\Map}{\mathrm{Map}}

\newcommand{\Ber}{\mathrm{Ber}}

\newcommand{\uHom}{\underline{\mathrm{Hom}}}

\newcommand{\Var}{\textsc{Var}}

\newcommand{\Lotimes}{\overset{\Lb}{\otimes}}

\renewcommand{\1}{\mathbbm{1}}

\newcommand{\Mon}{\textsc{Mon}}

\DeclareMathOperator*{\colim}{\mathrm{colim}}
\DeclareMathOperator*{\holim}{\mathrm{holim}}
\DeclareMathOperator*{\hocolim}{\mathrm{hocolim}}

\newcommand{\Func}{\mathrm{Func}}

\renewcommand{\Proj}{\mathrm{Proj}}

\newcommand{\Gr}{\mathrm{Gr}}
\newcommand{\ib}{\mathrm{ib}}

\newcommand{\Aff}{\textsc{Aff}}

\newcommand{\Grpd}{\textsc{Grpd}}
\newcommand{\Sch}{\textsc{Sch}}
\newcommand{\RatAlg}{\textsc{RatAlg}}
\newcommand{\DAlg}{\textsc{DAlg}}
\newcommand{\DAff}{\textsc{DAff}}
\newcommand{\DAn}{\textsc{DAn}}
\newcommand{\Algrm}{\textrm{Alg}}
\newcommand{\Ringrm}{\textrm{Ring}}
\newcommand{\Alg}{\textsc{Alg}}

\newcommand{\Sets}{\textsc{Sets}}

\newcommand{\SNRings}{\textsc{SNRings}}
\newcommand{\SNMod}{\textsc{SNMod}}
\newcommand{\BanRings}{\textsc{BanRings}}

\newcommand{\BanMod}{\textsc{BanMod}}

\newcommand{\SSets}{\textsc{SSets}}

\newcommand{\End}{\mathrm{End}}

\newcommand{\Cat}{\textsc{Cat}}

\newcommand{\id}{\mathrm{id}}

\renewcommand{\ch}{\mathrm{ch}}
\newcommand{\Td}{\mathrm{Td}}
\newcommand{\gr}{\mathrm{gr}}

\newcommand{\Tr}{\mathrm{Tr}}

\newcommand{\Gal}{\mathrm{Gal}}

\newcommand{\DA}{\mathrm{DA}}

\newcommand{\KV}{\mathrm{KV}}

\newcommand{\BGLbf}{\mathbf{BGL}}

\newcommand{\HC}{\mathrm{HC}}

\newcommand{\GL}{\mathrm{GL}}

\newcommand{\Sh}{\mathrm{Sh}}
\newcommand{\DR}{\mathrm{DR}}
\renewcommand{\H}{\textsc{H}}

\title{Overconvergent global analytic geometry}

\author{Fr\'ed\'eric Paugam}

\begin{document}


\maketitle

\begin{flushright}
``The rain has stopped, the clouds have drifted away,\\
\;and the weather is clear again.''\\
-- Ry\={o}kan, {\it One robe, one bowl}
\end{flushright}

\begin{abstract}
We define a notion of global analytic space with overconvergent structure sheaf.
This gives an analog on a general base Banach ring of Gro{\ss}e-Kl\"onne's
overconvergent $p$-adic spaces and of Bambozzi's generalized affinoid varieties over $\R$.
This also gives an affinoid version of Berkovich's and Poineau's global
analytic spaces. This affinoid approach allows the introduction
of a notion of strict global analytic space, that has some relations with the ideas of Arakelov
geometry, since the base extension along the identity morphism on $\Z$ (from
the archimedean norm to the trivial norm) sends a strict global analytic space to a usual
scheme over $\Z$, that we interpret here as a strict analytic space over $\Z$ equipped
with its trivial norm. One may also interpret some particular analytification functors as mere
base extensions.
We use our categories to define overconvergent motives and an overconvergent
stable homotopy theory of global analytic spaces.
These have natural Betti, de Rham and pro-\'etale realizations.
Motivated by problems in global Hodge theory and integrality questions in the theory of special
values of arithmetic $L$-functions, we also define derived overconvergent global analytic spaces
and their (derived) de Rham cohomology. Finally,
we use Toen and Vezzosi's derived geometric methods to define a natural (integral)
Chern character on analytic Waldhausen's $K$-theory with values in analytic cyclic homology.
The compatibility of our constructions with Banach base extensions gives
new perspectives both on global analytic spaces and on the various realizations
of the corresponding motives.
\end{abstract}

\newpage
\tableofcontents
\newpage

\section{Introduction}
Berkovich defined analytic spaces over a general Banach ring $(R,|\cdot|_R)$ in \cite{Berkovich1},
and Poineau studied them in a refined way in \cite{Poineau1} and \cite{Poineau2}.
Their local models are given by coherent sheaves of ideals in the ring of analytic functions
on an open subset of the analytic affine space $\A^n_{(R,|\cdot|_R)}$. This affine space
is given by the space of multiplicative seminorms
$$|\cdot|:R[X_1,\dots,X_n]\to \R_+$$
on the polynomial ring whose restriction to $R$ is bounded by $|\cdot|_R$, and
equipped with its natural sheaf of analytic functions, defined as local uniform limits
of rational functions without poles.

On a $p$-adic field, it is convenient to work with a refined notion of analytic space,
that was also defined by Berkovich in \cite{Berkovich-etale-cohomology},
using building blocs similar to those of Tate's rigid analytic geometry \cite{Tate2}:
affinoid algebras.
These are given by quotients of rings of power series $\Q_p\{\rho^{-1} T\}$
that converge (in a $p$-adic sense) on a given polydisc of arbitrary positive real
radius $\rho$ (Tate's theory reduces to $\rho=1$, but this is not so well adapted
to trivially valued fields, or to more general Banach rings, like the Banach ring
$\Zbf=(\Z,|\cdot|_\infty)$ of integers with its archimedean absolute value).

The natural question that we answer in this paper is to give a setting for global analytic
geometry that allows to change the base Banach ring along a bounded morphism
$\phi:(A,|\cdot|_A)\to (B,|\cdot|_B)$, and also to work with a notion of
affinoid algebra over a general base Banach ring, in a way that is
compatible to Berkovich's original constructions when the base Banach ring
is $\Qbf_p=(\Q_p,|\cdot|_p)$.

If we want to get back Berkovich's analytic spaces over $\Qbf_p$, we have to
work with uniform base Banach rings, given by Banach rings whose seminorm
is power-multiplicative. Indeed, in this setting, we will have
$$\Qbf_p\{\rho^{-1}T\}\cong \Zbf\{\rho^{-1}T\}\otimes_\Zbf\Qbf_p,$$
where $\otimes$ denotes the coproduct of uniform Banach rings, and we do
the base change along the bounded morphism $\Zbf=(\Z,|\cdot|_\infty)\to (\Q_p,|\cdot|_p)=\Qbf_p$.

It was already quite clear in Berkovich's original work \cite{Berkovich1} that, if the
base Banach ring is $\Cbf=(\C,|\cdot|_\infty)$, the naive definition of affinoid spaces
over $\Cbf$ would lead to important difficulties: the uniform ring of convergent power series on
the unit disc identifies with the ring of continuous functions on the disc that
are holomorphic in its interior. A simple quotient of such a ring will give
the ring of all continuous functions on the unit circle, that is not something
we would like to call an affinoid algebra over $\Cbf$. 
As shown by Bambozzi in his thesis \cite{Bambozzi}, rings of overconvergent power
series on closed polydiscs over an archimedean field are better behaved than usual rings
of convergent power series, and allow the definition of a complete archimedean analog 
of Berkovich's $p$-adic analytic spaces. They will also allow us, in the $p$-adic setting,
to circumvent in a very natural way all the boundary-related difficulties that appear in
Berkovich's theory.

It is known since the work of Monsky and Washnitzer \cite{Monsky-Washnitzer} that
the de Rham cohomology of non-proper $p$-adic analytic spaces does not work well with convergent
power series, because the Poincar\'e Lemma fails in this setting. Its correct formulation
must be done in terms of rings of overconvergent power series (Gro{\ss}e-Kl\"onne subsequently
developed the overconvergent analogs of rigid analytic varieties in \cite{Grosse-Kloenne},
and studied their de Rham cohomology in \cite{Grosse-Kloenne2} and \cite{Grosse-Kloenne3}).

The ring $R\{\rho^{-1}T\}^\dagger$ of overconvergent power series of radius $\rho$ is
a formal filtered colimit of the uniform Banach rings $R\{\nu^{-1}T\}$
of convergent power series of radii $\nu>\rho$. The natural setting
to study such formal filtered colimits is the setting of uniform ind-Banach rings.
We will thus define overconvergent power series rings as particular kinds of
uniform ind-Banach rings.

However, a uniform Banach ring is necessarily reduced. This prevents us from
using nilpotent elements in uniform ind-Banach rings to formulate
differential calculus algebraically, as one does in scheme theory and in complex
analytic geometry.

We thus need to define a category of overconvergent analytic rings over a given
Banach ring $(R,|\cdot|_R)$ that contains the rings of overconvergent power series,
but that also allows us to use nilpotent elements. The construction of this
``completion'' of the category of overconvergent power series rings is done in
a way similar to the one used by Lawvere in synthetic differential geometry
\cite{Lawvere-categorical-dynamics}, by Dubuc and Zilber in synthetic analytic
geometry \cite{Dubuc-Zilber} and by Lurie in derived geometry in \cite{Lurie-DAG-V},
using ``functors of functions''.
The advantage of this categorical approach to analytic rings is that it generalizes
directly to the derived setting, and allows a natural (i.e., functorial in the $\infty$-categorical
sense) definition of (derived) de Rham cohomology for non-smooth spaces.
Remark that there is also in \cite{Ben-Bassat-Kremnitzer} an approach to
non-archimedean analytic geometry using a geometry relative to the
symmetric monoidal category of Banach spaces.

Once a convenient category of overconvergent rings is defined, it is easy to
use the ``functor of point'' approach to define a natural notion of overconvergent
analytic space over a given Banach ring.

The great interest of this affinoid approach to global analytic geometry is that
it allows the definition of strict global analytic spaces that is stable by base
extension along a bounded morphism of Banach ring. For example,
strict dagger analytic spaces over over $(\Z,|\cdot|_\infty)$ have a
base extension to $(\Z,|\cdot|_0)$ given by usual schemes. Such strict global
analytic models for schemes over $\Z$ may be considered as
``Arakelov type'' models. This will be discussed further in Section
\ref{Dagger-analytic-Arakelov}. 

Using this new setting for global analytic geometry, we will define in Section
\ref{global-analytic-motives} various cohomological invariants, as \'etale,
analytic motivic cohomology, and global analytic $K$-theory. We will explain the relation
of these invariants to the ones that were already developed before in the
theory of schemes, that is given in our theory by a mere base extension
along a bounded morphism of Banach rings, in the case of schemes
that admit ``Arakelov type'' models in our sense.

Finally, motivated by applications to global comparison isomorphisms between
\'etale cohomology and de Rham cohomology, and by integrality questions in
the theory of special values of $L$-functions, we will define global analytic
derived analytic spaces, their derived de Rham cohomology, and their
cyclic Chern character.

{\flushleft{\it\bf Acknowledgments: }}

During the preparation of this work, the author was supported by the university
Pierre and Marie Curie and the ANR project ``Espaces de Berkovich globaux''.
I thank J. Ayoub, F. Bambozzi, O. Ben-Bassat, B. Bhatt, B. Conrad, D.-C. Cisinski, F. Deglise,
B. Drew, A. Ducros, M. Flach, E. Gro{\ss}e-Kl\"onne, F. Ivorra, M. Karoubi, K. Kedlaya,
T. Lemanissier, R. Liu, F. Loeser, B. Morin, J. Poineau, M. Porta, J. Riou, M. Robalo, G. Tamme,
M. Temkin, P. Schapira, J. Scholbach, P. Scholze and A. Vezzani for useful discussions.
Special thanks are due to \foreignlanguage{russian}{мудрец}.
 
\section{Seminormed algebraic objects}
\subsection{$\R_+$-graded sets and seminorms}
\label{R-plus-graded-sets}
Let $\R_+:=\R_{\geq 0}$ be the set of positive real numbers.
We will now define various categories of $\R_+$-graded sets that
give the natural setting for the theory of seminormed algebraic objects.
\begin{definition}
An \emph{$\R_+$-graded set} is a pair $(X,|\cdot|_X)$ composed of
a set $X$ and a map $|\cdot|_X:X\to \R_+$. A \emph{graded
map} (resp. \emph{contracting map}, resp. \emph{bounded map}) of
$\R_+$-graded set is a map $f:X\to Y$ such that
$$
\begin{array}{c}
|f(x)|_Y=|x|_X\\
\textrm{(resp. }|f(x)|_Y\leq |x|_X \textrm{ for all $x\in X$,}\\
\textrm{resp. there exists $C>0$ such that }|f(x)|_Y\leq C\cdot |x|_X\textrm{ for all $x\in X$).}
\end{array}
$$
The category of $\R_+$-graded sets with graded (resp. contracting, resp. bounded) maps
is denoted $\R_+^\Sets$ (resp. $\R_{+\leq 1}^\Sets$, resp. $\R_{+\leq}^\Sets$).
\end{definition}

\begin{lemma}
The category $\R_+^\Sets$ has arbitrary limits and colimits and has internal homomorphisms
for the product monoidal structure.
The category $\R_{+\leq 1}^\Sets$ has arbitrary colimits,
and in particular, the degree zero set $\{0_0\}$ as terminal object.
The category $\R_{+\leq}^\Sets$ has finite colimits and in particular
$\{0_0\}$ as a terminal object. More generally, $\R_{+\leq}^\Sets$ has
uniformly bounded colimits, meaning that if $X:I\to \R_{+\leq}^\Sets$ is a uniformly bounded
diagram (i.e., a diagram such that there exists $C$ such that for all $\phi:i\to j$ in $I$,
$|X(\phi)(x)|_{X_j}\leq C\cdot|x|_{X_i}$ for all $x\in X_i$) and $f:X\to Z$ is a uniformly
bounded cocone (meaning that there exists $D>0$ such that for every $i\in I$,
$f(i):X_i\to Z$ is $D$-bounded), then $\colim_{i\in I}f_i:\colim_{i\in I} X_i\to Z$ exists.
The categories $\R_{+\leq 1}^\Sets$ and $\R_{+\leq}^\Sets$ have finite limits
and internal homorphisms for the product monoidal structure.
\end{lemma}
\begin{proof}
The category $\R_+^\Sets=\uHom_{\Cat}(\R_{+,disc},\Sets)$,
being the functor category of $\Sets$-valued functors on
the discrete category with set of objects $\R_+$, has arbitrary limits and colimits.
The set $\R_+$, graded by the identity, is its final object.
The disjoint union $\coprod_i X_i$ of  the underlying sets of a family $(X_i,|\cdot|_i)$
of $\R_+$-graded sets, equipped with the grading $\coprod |\cdot|_{X_i}$ is always
a graded and a contracting coproduct. If the family is finite, it is a bounded coproduct.
If we have an $\N$-indexed family $\{X_n\}$ of non-empty sets of degree $1$, we may
define a family of bounded maps to $Y=\R_+$ by sending $X_n$ to $n$. There is no
bounded map that extends this family to the coproduct set. From this counter-example,
we see that the obstruction to having a colimit (i.e., a coproduct) for a cocone
$f:X\to Z$, with $X:I\to \R_{+\leq}^\Sets$
a discrete diagram, disappears if the cocone is uniformly bounded.
If $f,g:(X,|\cdot|_X)\to (Y,|\cdot|_Y)$ are two parallel bounded maps, and
$\pi:Y\to Z=\coker(f,g)$ is the coequalizer of the underlying sets, one may equip it
with the grading (the infimum is that of a constant if both maps are graded)
$$|z|_Z=\inf_{y\in \pi^{-1}(z)}|y|_Y.$$
This gives a coequalizer in the categories $\R_{+\leq}^\Sets$ and
$\R_{+\leq 1}^\Sets$. This proves all the desired results about colimits.
The product $\prod_{i=1}^n X_i$ of the underlying sets of a finite family $(X_i,|\cdot|_i)_{i=1,\dots,n}$
of $\R_+$-graded sets, equipped with the grading
$$|(x_1,\dots,x_n)|:=\max_i|x_i|$$
is a product in the categories $\R_{+\leq 1}^\Sets$ and $\R_{+\leq}^\Sets$.
Given a parallel pair $f,g:(X,|\cdot|_X)\to (Y,|\cdot|_Y)$ of bounded (resp. contracting) maps,
the kernel of the pair $(f,g)$ of set maps equipped with the grading induced
by that of $X$ is a kernel for the pair $(f,g)$ in the category $\R_{+\leq}^\Sets$
(resp. $\R_{+\leq 1}^\Sets$).
Indeed, if $h:(Z,|\cdot|_Z)\to (X,|\cdot|_X)$ is a bounded (resp. contracting) map such that
$f\circ h=g\circ h$, then $h$ factors set theoretically in a unique way through the kernel,
and this factorization is bounded (resp. contracting).
The internal homomorphisms between two objects $X$ and $Y$
of $\R_{+\leq}^\Sets$ (resp. $\R_{+\leq1}^\Sets$)
are given by the set $\uHom(X,Y)$ of bounded (resp. contracting) maps with the grading given by
$$|f|_{\uHom(X,Y)}:=\inf\left\{C>0,\;|f(x)|_Y\leq C\cdot |x|_X\textrm{ for all $x\in X$}\right\}.$$
\end{proof}

The set $\R_+$ has a natural family of commutative monoid structures indexed by $p\in ]0,+\infty]$
given by $(+_p,0)$, where
$$r+_p s:=\sqrt[p]{r^p+s^p}$$
for $p<+\infty$ and
$$r+_\infty s:=\max(r,s).$$
Remark that we have $x+_p s\leq r+_{p'}s$ if $p'\leq p$.
One also has the multiplicative monoid structure $(\times,1)$,
given by $r\times s:=r.s$.
We now define the associated symmetric monoidal structures on $\R_+$-graded sets.
\begin{definition}
The \emph{multiplicative (resp. $p$-additive, resp. maximum) tensor product} of
two $\R_+$-graded sets $(X,|\cdot|_X)$ and $(Y,|\cdot|_Y)$ is the product $X\times Y$ equipped
with the grading
$$
\begin{array}{c}
|(x,y)|_{X\otimes_m Y}:=|x|_X\cdot|y|_Y\\
\textrm{(resp. }|(x,y)|_{X\otimes_p Y}:=|x|_X+_p |y|_Y,\\
\textrm{resp. }|(x,y)|_{X\otimes_\infty Y}:=\max(|x|_X,|y|_Y)).
\end{array}
$$
\end{definition}
The three tensor products give monoidal structures on the categories $\R_+^\Sets$,
$\R_{+\leq}^\Sets$ and $\R_{+\leq 1}^\Sets$. The unit object of the multiplicative
monoidal structures is the one element set $\{1_1\}$. The unit object of the
$p$-additive and maximum monoidal structure is the one element set $\{0_0\}$.

Remark that the maximal tensor product $\otimes_\infty$ is equal to the
product $\times$ in the categories $\R_{+\leq}^\Sets$ and
$\R_{+\leq1}^\Sets$. 

\begin{definition}
\begin{enumerate}
\item A \emph{weakly seminormed additive monoid (resp. abelian group)}
is a commutative monoid (resp. an abelian group) object in
the symmetric monoidal category $(\R_{+\leq}^\Sets,\otimes_\infty,\{0_0\})$. More concretely,
this is a monoid (resp. an abelian group) $(M,+,0)$ equipped with a grading $|\cdot|_M$ with
$|0|_M=0$, and such that there exists $C>0$ with
$$|m+n|\leq C\cdot \max(|m|,|n|) \textrm{ (resp. $|m-n|\leq C\cdot \max(|m|,|n|)$).}$$
\item A \emph{seminormed monoid (resp. abelian group)} is a commutative monoid
(resp. an abelian group) object in the symmetric
monoidal category $(\R_{+\leq 1}^\Sets,\otimes_1,\{0_0\})$. More concretely, this
is a monoid (resp. an abelian group) $(M,+,0)$ equipped with a grading $|\cdot|_M$ such that
$|0|_M=0$ and
$$|m+n|\leq |m|+|n| \textrm{ (resp. $|m-n|\leq |m|+|n|$).}$$
\item A \emph{weakly seminormed multiplicative monoid} is a commutative monoid object
in the symmetric monoidal category $(\R_{+\leq}^\Sets,\otimes_m,\{1_1\})$. More concretely,
this is a monoid $(M,\times,1)$ equipped with a grading $|\cdot|_M$ such that there exists $C>0$
with
$$|m\cdot n|_M\leq C\cdot|m|_M\cdot|n|_M.$$
\item A \emph{seminormed multiplicative monoid} is a commutative monoid object
in the symmetric monoidal category $(\R_{+\leq 1}^\Sets,\otimes_m,\{1_1\})$. More concretely,
this is a monoid $(M,\times,1)$ equipped with a grading $|\cdot|_M$ such that $|1|_M\leq 1$
and
$$|m\cdot n|_M\leq |m|_M\cdot|n|_M.$$
\end{enumerate}
\end{definition}

If $(\Cc,\otimes)$ is one of the above symmetric monoidal categories, we will
denote $\Mon(\Cc,\otimes)$ the category of (commutative) monoids in $\Cc$.
There is a natural forgetful functor
$$\Mon(\Cc,\otimes)\to \Cc.$$

Remark that if $(M,|\cdot|_M)$ is a weakly seminormed abelian group
(resp. a weakly seminormed multiplicative monoid)
then $(M,|\cdot|_M^t)$ is also a weakly seminormed abelian group
(resp. multiplicative monoid) for every $t>0$.

\begin{definition}
A \emph{weakly seminormed ring (resp. seminormed ring)} is a tuple
$(R,|\cdot|_R,\times,+)$ composed of a weakly seminormed (resp. seminormed)
abelian group $(R,|\cdot|_R,+)$ and a weakly seminormed (resp. seminormed)
multiplicative monoid $(R,|\cdot|_R,\times)$ such that $(R,+,\times)$ is a ring
in the classical sense. More concretely, an $\R_+$-graded set $(R,|\cdot|_R)$
equipped with a ring structure $(R,\times,+)$ such that $|0|_R=0$ is a weakly seminormed ring
if there exists $C>0$ and $D>0$ such that
$$|a-b|_R\leq C\cdot\max(|a|_R,|b|_R)\;\textrm{ and }\;|a\cdot b|_R\leq D\cdot |a|_R\cdot|b|_R.$$
It is a seminormed ring if we can choose $C=2$, $D=1$, we have $|1|_R\leq 1$,
and we further have
$$|a-b|_R\leq |a|_R+|b|_R.$$
\end{definition}

By definition, a weakly seminormed ring has an underlying $\R_+$-graded set
in $\R_{+\leq}^\Sets$, whereas a seminormed ring has an underlying $\R_+$-graded
set in $\R_{+\leq 1}^\Sets$.

A seminormed ring is called complete (or a Banach ring) if it is complete for
the topology induced by its seminorm.
We will denote $\SNRings$ (resp. $\BanRings$) the category of seminormed
(resp. complete seminormed) rings.

\subsection{Seminormed modules}
\begin{definition}
Let $(R,|\cdot|_R)$ be a seminormed ring. A \emph{seminormed module} over $(R,|\cdot|_R)$ is
a module over $(R,|\cdot|_R)$ in  $(\R_{+\leq 1}^\Sets,\otimes_1,\otimes_m)$. More
concretely, this is a seminormed
abelian group $(M,+,|\cdot|_M)$ together with a multiplication map
$\cdot:R\times M\to M$ that makes $M$ an $R$-module in the usual sense
and such that
$$|a\cdot m|_M\leq |a|_A\cdot |m|_M.$$
We will denote $\SNMod(R,|\cdot|_R)$ the category of modules over
$(R,|\cdot|_R)$
\end{definition}

Let $(R,|\cdot|_R,+,\times)$ be a seminormed ring,
and let $X$ be an object in $\R_{+\leq1}^\Sets$.
One may equip the set
$$R^{(X)}:=\Hom_{\Sets-fs}(X,R)$$
of all finitely supported maps from $X$ to $R$ with the $\ell^1$-grading given by
$$\left\|\sum_X a_x\{x\}\right\|_1:=\textstyle\sum |a_x|_R\cdot |x|_X.$$
This gives a seminormed $R$-module structure $(R^{(X)},\|\cdot\|_1)$
called the free seminormed module on $R$.
If $f:X\to Y$ is a morphism in $\R_{+\leq 1}^\Sets$, and
$$f_*:R^{(X)}\to R^{(Y)}$$
is the associated module map, given by
$$(f_*a)_y:=\sum_{x\in X,\;f(x)=y} a_x,$$
then we have
$$
\|f_* a\|_1:=\left\|\sum_{y\in Y}\left(\sum_{f(x)=y} a_x\right)\{y\}\right\|_1=
\sum_{y\in Y}\left|\sum_{f(x)=y} a_x\right|_R\cdot |y|_Y\leq 
\sum_{y\in Y}\sum_{f(x)=y} (|a_x|_R\cdot |f(x)|_Y)
$$
so that
$$
\|f_*a\|_1\leq \sum_{y\in Y}\sum_{f(x)=y} (|a_x|_R\cdot |x|_X)= \sum_{x\in X}|a_x|_R\cdot |x|_X=
\|a\|_1,
$$
which implies that $f_*$ is a morphism in $\R_{+\leq 1}^\Sets$, so that
$X\mapsto R^{(X)}$ gives a functor
$$R^{(\cdot)}:\R_{+\leq 1}^\Sets\longrightarrow \SNMod(R,|\cdot|_R).$$
The free seminormed module $R^{(X)}$ on an $\R_+$-graded set $X$ has the following
universal property: if $f:X\to M$ is a contracting morphism from $X$ to a seminormed
$R$-module, there exists a unique extension of $f$ to a morphism of seminormed $R$-modules
$$[f]:R^{(X)}\to M.$$
The extension is given by $[f](a)=\sum_x a_x\cdot f(x)$. It is a contracting
map since
$$
|f(a)|_M=|\sum_x a_xf(x)|_M\leq
\sum_x |a_x|_R\cdot |f(x)|_M\leq \sum_x|a_x|_R\cdot|x|_X=
\|a\|_1.
$$
Composition with the forgetful functor $\SNMod(R,|\cdot|_R)\to \R_{+\leq 1}^\Sets$
thus gives an endofunctor
$$\Sigma^m_R:\R_{+\leq 1}^\Sets\to \R_{+\leq 1}^\Sets$$
which is monadic (using the usual composition of linear combinations).
One may recover $R$ from $\Sigma^m_R$ by setting $R=\Sigma_R(\{1_1\})$.
This gives an embedding $R\mapsto \Sigma^m_R$ of seminormed rings into
monads in $\R_{+\leq 1}^\Sets$, similar to the embedding of usual rings
in monads in $\Sets$, used by Durov \cite{Durov-2007} in his theory of
generalized rings.

If $R$ is a Banach ring, a seminormed module over $R$ is called a Banach module
if the underlying abelian group is complete for the topology induced by its group
seminorm. We denote $\BanMod(R,|\cdot|_R)$ the category of Banach modules over $R$.
There is a natural completion functor
$$\hat{-}:\SNMod(R,|\cdot|_R)\to \BanMod(R,|\cdot|_R),$$
and the composition of the free seminormed module functor
$R^{(\cdot)}$ with $\hat{-}$ gives an endofunctor
$$\widehat{\Sigma}^m_R:\R_{+\leq 1}^\Sets\to \R_{+\leq 1}^\Sets$$
that is monadic. One may still recover $R$ from $\widehat{\Sigma}^m_R$ by
setting $R=\widehat{\Sigma}^m_R(\{1_1\})$.

One may wonder what happens to these monadic embeddings if we work in the
category $\R_{+\leq}^\Sets$ of $\R_+$-graded sets with bounded maps, and with
a weakly seminormed ring $(R,|\cdot|_R,+,\times)$.
We will denote $C$ the norm of the map $(x,y)\mapsto x+y$
and $D$ the norm of the multiplication map $(x,y)\mapsto x\cdot y$ on $R$.
In this setting, the free $R$-module $R^{(X)}$ on an $\R_+$-graded set $X$,
defined as the set of finitely supported maps $a:X\to R$, gives a functor
$$R^{(\cdot)}:\R_{+\leq}^\Sets\to \Mod(R)$$
with values in (non-graded) $R$-modules.
One may define the $\ell^\infty$ grading on $R^{(X)}$ by setting
$$\left|\sum_{x\in X} a_x\{x\}\right|_\infty:=\max_{x\in X}(|a_x|_R\cdot |x|_X).$$
The problem is that for $f:X\to Y$ a bounded map, the associated map
$$f_*:R^{(X)}\to R^{(Y)}$$
given by
$$f_*a_y:=\sum_{x\in X,\;f(x)=y} a_x$$
is not anymore bounded in general. Similarly, the composition of linear combinations
$$\mu:R^{\left(R^{(X)}\right)}\to R^{(X)}$$
is often not bounded. This shows that $X\mapsto R^{(X)}$ induces only a monad
$\Sigma_R^m:\R_{+\nleq}^\Sets\to \R_{+\nleq}^\Sets$ in the category $\R_{+\nleq}^\Sets$ of
$\R_+$-graded sets with unbounded (i.e., arbitrary) maps, not in the
category $\R_{+\leq}^\Sets$.
If $f:R\to S$ is a bounded morphism of weakly seminormed rings with norm $C_f$,
and $X$ is an $\R_+$-graded set, we get a natural map
$$[f]:R^{(X)}\to S^{(X)}$$
that is $\R_+$-bounded for the $\ell^\infty$-gradings.
Indeed, we have
$$
|[f](a)|_\infty=|\sum_x f(a_x)\{x\}|_\infty:=\max_x(|f(a_x)|_S\cdot |x|_X)\leq
C_f\cdot \max_x(|a_x|_R\cdot |x|_X)=C_f\cdot |a|_\infty.
$$
We thus have a natural embedding $R\mapsto \Sigma_R^m$
of the category of weakly seminormed rings
into the category of monads on $\R_{+\nleq}^\Sets$ with morphisms given by
monad morphisms $f:\Sigma\to \Sigma'$ such that for all $\R_+$-graded
set, $f_X:\Sigma(X)\to \Sigma'(X)$ is bounded.

Even if $\Sigma_R^m$ is not a monad in $\R_{+\leq}^\Sets$, it has a complete $\R_+$-filtration
by a natural family of monads on $\R_{+\leq}^\Sets$, defined in the following way:
for $\rho\geq 0$, let $R^{(X)}_{\leq \rho}\subset R^{(X)}$ be the subset
of finitely supported maps $a:X\to R$ such that
for every partition $\coprod_i Z_i=Z$ of a subset $Z$ of $X$, we have
$$\left|\sum_{z\in Z} a_z\right|_R\leq \rho\cdot\max_i\left(\left|\sum_{z\in Z_i}a_{z}\right|_R\right),$$
together with the grading induced by the $\ell^\infty$ grading on $R^{(X)}$.
If $\nu\leq \rho$, we will have
$$R^{(X)}_{\leq \nu}\subset R^{(X)}_{\leq \rho}.$$
Moreover, the free module $R^{(X)}$ may be described as the union
$$R^{(X)}=\cup_{\rho\to \infty}R^{(X)}_{\leq \rho}.$$
If $t>0$ and $\rho>0$, we will have
$$(R,|\cdot|_R)^{(X)}_{\leq \rho}=(R,|\cdot|_R^t)^{(X)}_{\leq \rho^t}.$$
Remark that by definition, the natural embedding $R^{(X)}_{\leq \rho} \hookrightarrow R^{(X)}$
is a bijection if $\rho\geq C$ and $X$ has a cardinal in $\{0,1,2\}$.
Remark also that $R^{(X)}_{\leq \rho}$ is not an $R$-module in general.
For example, if $(R,|\cdot|_R)=(\Z,|\cdot|_\infty)$, $\rho=C=2$,
and $X=\{1,2,3\}$, we get that $\Z^{(X)}_{\leq 2}\subset \Z^3$ is the set of triples
$(n_1,n_2,n_3)$ of integers such that
$$|n_1|_\infty+|n_2|_\infty+|n_3|_\infty\leq 2\cdot\max_i(|n_i|_\infty),$$
and this is not stable by addition, because $(1,1,2)$ and $(0,0,-1)$ are in it, but not
$(1,1,1)$. We will see that in spite of these defects, the $\R_{+\nleq}^\Sets$-monad structure
on $X\mapsto R^{(X)}$ given by the composition of linear combinations
extends to an $\R_{+\leq}^\Sets$ monad structure on each $R_{\leq \rho}^{(X)}$.
If $f:(X,|\cdot|_X)\to (Y,|\cdot|_Y)$ is a bounded map of norm $\|f\|=C_f$,
then the map
$$f_*:R^{(X)}_{\leq\rho} \to R^{(Y)}_{\leq\rho}$$
given by
$$(f_*a)_y:=\sum_{x\in X,\;f(x)=y} a_x$$
is well defined. Indeed, a partition $\coprod_i Z_i'$ of a subset $Z'$ of $Y$ gives
a partition of the subset $Z=f^{-1}(Z')$ of $X$ by $Z_i=f^{-1}(Z'_i)$, so that
$$
\left|\sum_{z'\in Z'} (f_*a)_{z'}\right|_R:=\left|\sum_{z\in Z'}\sum_{f(z)=z'} a_z\right|_R=
\left|\sum_{a\in Z} a_{z}\right|_R
\leq \rho\cdot\max_i\left(\left|\sum_{z\in Z_i}a_{z}\right|_R\right)=
\rho\cdot\max_i\left(\left|\sum_{z'\in Z_i'}f_*a_{z'}\right|_R\right).
$$
It is also bounded since
$$
\|f_* a\|_\infty:=\left\|\sum_{y\in Y}\left(\sum_{f(x)=y} a_x\right)\{y\}\right\|_\infty=
\max_{y\in Y}\left(\left|\sum_{f(x)=y} a_x\right|_R\cdot |y|_Y\right)\leq 
\rho\cdot\max_{\underset{f(x)=y}{y\in Y}} (|a_x|_R\cdot |f(x)|_Y)
$$
so that
$$
\|f_*a\|_\infty\leq
\rho\cdot C_f\cdot \max_{\underset{f(x)=y}{y\in Y}} (|a_x|_R\cdot |x|_X)=
\rho\cdot C_f\cdot \max_{x\in X}(|a_x|_R\cdot |x|_X)=
\rho\cdot C_f\cdot \|a\|_\infty,
$$
which implies that $f_*$ is a morphism in $\R_{+\leq}^\Sets$ of norm smaller than $\rho\cdot C_f$.
The endofunctor
$$R_{\leq \rho,}^{(\cdot)}:\R_{+\leq}^\Sets\to \R_{+\leq}^\Sets$$
is also monadic since the composition map for linear combinations
$$\mu:R^{\left(R^{(X)}_{\leq\rho}\right)}_{\leq\rho}\to R^{(X)}_{\leq\rho}$$
is bounded.
Indeed, if $a\in R^{\left(R^{(X)}_{\leq\rho}\right)}_{\leq\rho}$ is given by
$$
a=\displaystyle\sum_{b\in R^{(X)}_{\leq\rho}}
\textstyle a_{\sum_x b_x\{x\}} \left\{\sum_x b_x\{x\}\right\},
$$
then we have
$$
\begin{array}{ccl}
|a|_\infty & = &
\displaystyle\max_{b\in R^{(X)}_{\leq\rho}}
\textstyle\left(|a_{\sum_x b(x)\{x\}}|_R\cdot|\sum b_x\{x\}|_\infty\right)\\
& = &
\displaystyle\max_{b\in R^{(X)}_{\leq\rho}}
\textstyle\left(|a_{\sum_x b(x)\{x\}}|_R\cdot \max_{x}(|b_x|_R\cdot|x|_X)\right)\\
& = &
\displaystyle\max_{\underset{x\in X}{b\in R^{(X)}_{\leq\rho}}}
\textstyle\left(|a_{\sum_x b_x\{x\}}|_R\cdot |b_x|_R\cdot|x|_X\right)
\end{array}
$$
and also
$$
\begin{array}{ccl}
|\mu(a)|_\infty
& = &
\left|\displaystyle\sum_{x\in X}
\left(\sum_{\underset{b_x\neq 0}{b\in R^{(X)}_{\leq\rho}}}
a_{\sum_x b_x\{x\}}\cdot b_x\right)\{x\}\right|_\infty\\
& = &
\displaystyle\max_{x\in X}
\left(\left|\sum_{\underset{b_x\neq 0}{b\in R^{(X)}_{\leq\rho}}}
a_{\sum_x b_x\{x\}}\cdot b_x\right|_R\cdot|x|_X\right)\\
& \leq &\rho\cdot D\cdot
\displaystyle\max_{x\in X}
\left(\max_{\underset{b_x\neq 0}{b\in R^{(X)}_{\leq\rho}}}
\left(|a_{\sum_x b_x\{x\}}|_R\cdot |b_x|_R\right)\cdot|x|_X\right)\\
& = &\rho\cdot D\cdot 
\displaystyle\max_{\underset{x\in X}{b\in R^{(X)}_{\leq\rho}}}
\textstyle\left(|a_{\sum_x b_x\{x\}}|_R\cdot |b_x|_R\cdot|x|_X\right)
\end{array}
$$
so that
$$|\mu(a)|_\infty\leq \rho\cdot D\cdot |a|_\infty.$$
One may recover $(R,+,\times,|\cdot|_R)$ as $R=R_{\leq C}^{(\{1_1\})}$
(where $C$ denotes, as before, the norm of the addition map $+$ on $R$).
The addition and multiplication maps can be found back
by composition of the monadic multiplication
$$\mu:R_{\leq C}^{\left(R_{\leq C}^{(\{1_1\})}\right)}\to R_{\leq C}^{(\{1_1\})}$$
with the (bounded) embeddings
$$
\begin{array}{cccc}
[+]:	& R^2	& \to 		& R_{\leq C}^{\left(R_{\leq C}^{(\{1_1\})}\right)}\\
	& (a,b)	& \mapsto		& 1\cdot\{a\}+1\cdot\{b\}
\end{array}
$$
and
$$
\begin{array}{cccc}
[\times]:	& R^2	& \to 		& R_{\leq C}^{\left(R_{\leq C}^{(\{1_1\})}\right)}\\
		& (a,b)	& \mapsto		& a\cdot\{b\}
\end{array}
$$
Recall however that a bounded morphism $f:R\to S$ of weakly seminormed
rings only induces a morphism $R^{(\cdot)}\to S^{(\cdot)}$ of monads in $\R_{+\nleq}^\Sets$,
and not any morphism $R^{(\cdot)}_{\leq \rho}\to S^{(\cdot)}_{\leq \rho}$ in general.
This makes the theory of weakly seminormed rings algebraically much more complicated than
the theory of seminormed rings, and explains why we will mostly work with seminormed
structures and the category $\R_{+\leq 1}^\Sets$ from now on.
\begin{remark}
\label{Durov-ring-of-integers}
We may think of
$$R^\circ_\rho:X\mapsto R_{\leq \rho}^{(X)}$$
as a kind of ring of $\rho$-integers in $R$, analogous to Durov's archimedean ring of
integers $\Z_\infty\subset \R$ from \cite{Durov-2007}. It is possible to define an ideal
$$R^{\circ\circ}_\rho:X\mapsto R_{<\rho}^{(X)}$$
in this monad, with quotient given denoted by $\tilde{R}_\rho:=R^{\circ}_\rho/R^{\circ\circ}_\rho$.
The $\R_+^*$-graded monad
$$\Gr(R,|\cdot|_R):=\oplus_{r\in |R|_R\backslash 0} \tilde{R}_\rho$$
then gives a natural archimedean analog of Temkin's reduction from \cite{Temkin-local-properties-II},
Section 3, that may be useful to describe the archimedean points of the dagger analytic topoi,
if we start from a usual seminormed ring $R$.
\end{remark}

\subsection{Seminormed polynomials and convergent power series}
Let $(X,|\cdot|_X)$ be an object of $\R_{+\leq 1}^\Sets$ and $(\N,+,0)$ be
the additive monoid of non-negative integers. Let
$$(X)^\N:=\Hom_{\Sets-fs}(X,\N)$$
be the monoid of monomials on $X$, given by the set of finitely supported maps from
$X$ to $\N$, seen as a multiplicative monoid whose generic element is of the
form $X^\alpha$ for $\alpha:X\to \N$.
The multiplicative grading on $(X)^\N$ is given by
$$\|X^\alpha\|_m:=\prod_{x\in X,\;\alpha(x)\neq 0}|x|_X^{\alpha(x)},$$
with the convention that the empty product is equal to $1$.
This is a multiplicative monoid map $\|\cdot\|_m:(X)^\N\to \R_+$.
Indeed, if $\alpha,\beta:X\to \N$ are two finitely supported maps,
then we have
$$
\begin{array}{ccl}
\|X^\alpha\cdot X^\beta\|_m	& :=	& \|X^{\alpha+\beta}\|_m\\
						& :=	& \prod_{x\in X,(\alpha+\beta)(x)\neq 0}|x|_X^{(\alpha+\beta)(x)}\\
			& =	& \displaystyle
				   \left(\prod_{\underset{\beta(x)=0}{\underset{\alpha(x)\neq 0}{x\in X}}}
				   |x|_X^{\alpha(x)}\right)\cdot
				   \left(\prod_{\underset{\beta(x)\neq 0}{\underset{\alpha(x)\neq 0}{x\in X}}}
				   |x|_X^{(\alpha+\beta)(x)}\right)\cdot
				   \left(\prod_{\underset{\alpha(x)=0}{\underset{\beta(x)\neq 0}{x\in X}}}
				   |x|_X^{\beta(x)}\right)
\end{array}
$$
and also
$$
\begin{array}{ccl}
\|X^\alpha\|_m\cdot \|X^\beta\|_m	& :=	&
				   \left(\prod_{x\in X,\alpha(x)\neq 0} |x|_X^{\alpha(x)}\right)\cdot
				   \left(\prod_{x\in X,\beta(x)\neq 0}|x|_X^{\beta(x)}\right)\\
			& =	& \displaystyle
				   \left(\prod_{\underset{\beta(x)=0}{\underset{\alpha(x)\neq 0}{x\in X}}}
				   |x|_X^{\alpha(x)}\right)\cdot
				   \left(\prod_{\underset{\beta(x)\neq 0}{\underset{\alpha(x)\neq 0}{x\in X}}}
				   |x|_X^{\alpha(x)}\right)\cdot
				   \left(\prod_{\underset{\alpha(x)\neq 0}{\underset{\beta(x)\neq 0}{x\in X}}}
				   |x|_X^{\beta(x)}\right)\cdot
				   \left(\prod_{\underset{\alpha(x)=0}{\underset{\beta(x)\neq 0}{x\in X}}}
				   |x|_X^{\beta(x)}\right)
\end{array}
$$
so that
$$\|X^\alpha\cdot X^\beta\|_m=\|X^\alpha\|_m\cdot \|X^\beta\|_m.$$
Now suppose given a contracting map $f:X\to Y$ and $\alpha\in (X)^\N$.
As before, we define
$$f_*\alpha(y)=\sum_{x\in f^{-1}(y)}\alpha(x).$$
Suppose given a fixed $y\in Y$ such that $f_*\alpha(y):=\sum_{x\in f^{-1}(y)}\alpha(x)\neq 0$.
Then
$$
|y|_Y^{f_*\alpha(y)}=|y|_Y^{\sum_{x\in f^{-1}(y)} \alpha(x)}=
\prod_{x\in f^{-1}(y),\alpha(x)\neq 0} |f(x)|_Y^{\alpha(x)}.
$$
This implies that
$$
\|f_*\alpha\|_m:=\prod_{y\in Y,f_*\alpha(y)\neq 0}|y|_Y^{f_*\alpha(y)}=
\prod_{x\in f^{-1}(Y),\alpha(x)\neq 0}|f(x)|_Y^{\alpha(x)},
$$
and since $f:X\to Y$ is contracting, we get
$$
\|f_*\alpha\|_m\leq
\prod_{x\in f^{-1}(Y),\alpha(x)\neq 0}|x|_X^{\alpha(x)}=\prod_{x\in X,\alpha(x)\neq 0}|x|_X^{\alpha(x)}=
\|\alpha\|_m.
$$
This shows that $X\mapsto (X)^{\N}$ gives a functor
$$(\cdot)^\N:\R_{+\leq 1}^\Sets\longrightarrow \Mon(\R_{+\leq 1}^\Sets,\otimes_m).$$
This free multiplicatively seminormed monoid has the following universal property:
if \mbox{$f:X\to N$} is a contracting map from an $\R_+$-graded set to a (commutative)
multiplicatively seminormed monoid, there is a unique extension
$$[f]:(X)^\N\to N$$
of $f$ to a morphism of multiplicatively seminormed monoids, given by
$$[f](\alpha):=\prod_{x\in X,\alpha(x)\neq 0}f(x)^{\alpha(x)}.$$

Now we may combine the above two constructions: if $(X,|\cdot|_X)$ is an $\R_+$-graded set
and $(R,|\cdot|_R,+,\times)$ is a seminormed ring, we may define the seminormed polynomial
ring $R[X]$ by setting $R[X]:=R^{((X)^\N)}$, together with the $\ell^1$-seminorm $\|\cdot\|_1$
associated to the multiplicative seminorm $\|\cdot\|_m$ on $(X)^\N$. By construction, this is a
seminormed $R$-module
together with a multiplication given by the usual polynomial multiplication
$$
(\sum a_\alpha X^\alpha)\cdot (\sum b_\alpha X^\alpha):=
\sum_\gamma \left(\sum_{\alpha+\beta=\gamma} a_\alpha\cdot b_\beta\right)X^\gamma.
$$
We also have
$$
\left|(\sum a_\alpha X^\alpha)\cdot (\sum b_\alpha X^\alpha)\right|\leq
\sum_\gamma \left(\sum_{\alpha+\beta=\gamma} |a_\alpha|_R\cdot |b_\beta|_R\right)
|X^\alpha|.|X^\gamma|=\left|\sum a_\alpha X^\alpha\right|\cdot \left|\sum b_\alpha X^\alpha\right|.
$$
Combined with our previous results, this shows that we have defined a functor
$$R[\cdot]_1:\R_{+\leq1}^\Sets\to \SNRings.$$
The polynomial $R$-algebra has the following universal property: if $X$ is an $\R_+$-graded
set, $R\to S$ is a (contracting) morphism of seminormed rings, and $f:X\to S$ is a contracting map,
then $f$ extends uniquely to a morphism of seminormed $R$-algebras
$$[f]:R[X]_1\to S.$$
Composition with the forgetful functor $\SNRings\to \R_{+\leq 1}^\Sets$ gives an endofunctor
$$\Sigma^r_R:\R_{+\leq 1}^\Sets\to \R_{+\leq 1}^\Sets$$
which is monadic, since it is constructed as the composition of two monadic functors.
One may recover $R$ from $\Sigma^r_R$ by setting $R=\Sigma^r_R(\{\emptyset\})$.

There is a natural completion functor
$$\hat{-}:\SNRings\to \BanRings$$
that sends a seminormed ring to its completion.
If $R$ is a Banach ring, we will denote by
$$R\langle \cdot\rangle:\R_{+\leq1}^\Sets\to \BanRings$$
the composition of the completion functor with $R[\cdot]_1$, that
sends $X$ to $R\langle X\rangle:=\widehat{R[X]}$.
It still induces a monadic functor on $\R_{+\leq}^\Sets$.

\begin{definition}
A weakly seminormed ring $(A,|\cdot|_A)$ is called \emph{uniform} if its seminorm
is \emph{power-multiplicative}, meaning that
$$|a^n|_A=|a|_A^n\textrm{ for all $a\in A$ and $n\in \N$}.$$
We denote $\SNRings_{u}$ (resp. $\BanRings_u$) the category of uniform
(resp. complete uniform) seminormed rings.
\end{definition}

We will now show that, if we work with the category of uniform weakly seminormed rings,
which are the basic building blocs of global analytic geometry, we may restrict to
the category of uniform seminormed rings without loosing any information.
\begin{lemma}
The natural functor
$$\SNRings_u\to \SNRings_{w,u}$$
from uniform seminormed rings to uniform weakly seminormed rings is
fully faithful and every weakly seminormed ring is of the form
$(R,|\cdot|^t_R)$ for some $t>0$ and $(R,|\cdot|_R)$ a seminormed ring.
\end{lemma}
\begin{proof}
Let $(R,|\cdot|_R,+,\times)$ be a uniform weakly seminormed ring. Then by \cite{E-Artin1},
Theorem 3, there exists $t>0$
such that $|\cdot|_R^t$ fulfills the triangle inequality
$$|a+b|_R^t\leq |a|_R^t+|b|_R^t.$$
Moreover, if we have $|a\cdot b|_R\leq C|a|_R\cdot|b|_R$, then taking $n$-th powers
of the arguments, $n$-th roots of the terms in the inequality, and passing to the
limit $n\to \infty$, we get $|a\cdot b|_R\leq |a|_R\cdot |b|_R$. This shows
the last statement.
Similarly,  if $f:R\to S$ is a bounded morphism of uniform weakly seminormed rings,
the above ``$n$-th roots of unity argument'' shows that $f$ is contracting.
This shows that the inclusion functor is full. It is clearly faithful, which finishes the
proof that it is an equivalence of categories.
\end{proof}

There is a natural ``uniformization'' functor
$$U:\SNRings\to \SNRings_{u}$$
that sends $(R,|\cdot|_R)$ to $R$ equipped with the seminorm given on $a\in A$ by
$$|a|_{R,u}:=\lim_{n\to +\infty} \sqrt[n]{|a^n|}.$$

We will denote by
$$R\{\cdot\}:\R_{+\leq1}^\Sets\to \BanRings_{u}$$
the composition of $U$ with $R\langle\cdot\rangle$.
This still induces a monadic functor on $\R_{+\leq1}^\Sets$.

If $X=\rho=(\rho_1,\dots,\rho_n)\in \R_+^n$, we denote $\rho^{-1}T$ the set of variables
$\{\rho^{-1}T_1,\dots,\rho_n^{-1}T_n\}$ equipped with the $\R_+$-grading given
by $|\rho_i^{-1}T_i|=\rho_i$. We will denote
$$R\langle \rho^{-1}T\rangle:=R\langle X\rangle\;\textrm{ and }\;R\{\rho^{-1}T\}:=R\{X\}.$$

\begin{proposition}
If $(R,|\cdot|_R)$ is a seminormed ring, the category $\Mod(R,|\cdot|_R)$ has arbitrary colimits,
finite limits and internal homomorphisms. It also has a natural symmetric monoidal structure
called the projective tensor product.
Seminormed rings (resp. uniform seminormed rings) have finite colimits and finite limits.
\end{proposition}
\begin{proof}
The fact that seminormed modules have arbitrary colimits, finite limits and internal homorphisms
follows from the fact that the category $\R_{+\leq 1}^\Sets$ has these properties.
For example, if $(M,|\cdot|_M)$ is a seminormed abelian group and $R\subset M$ be a submodule,
the map $M/R\to \R_+$ defined by
$$|x|=\inf_{m\in \pi^{-1}(x)}|m|_M$$
is a group seminorm on $M/R$ called the residue seminorm.
If $R$ is a seminormed ring and $I\subset R$ is an ideal, then the residue seminorm
on $R/I$ is a ring seminorm that makes $A/I$ the quotient seminormed ring.
Similarly, if $f,g:(R,|\cdot|_R)\to (S,|\cdot|_S)$ are two parallel morphisms of seminormed rings,
then the coequalizer $\coker(f,g)$ is given by the ring $S/I$, where $I$ is the ideal generated by
elements of the form $f(x)-g(x)$ for $x\in R$, equipped with the residue seminorm.
Let $(R,|\cdot|_R)$ be a seminormed ring and $(M,|\cdot|_M)$ and $(N,|\cdot|_N)$ be
two seminormed modules over $(R,|\cdot|_R)$.
The seminorm on $M\otimes_R N$ given by the residue
seminorm of the $\ell^1$ module seminorm
$$
\begin{array}{cccc}
\|\cdot\|_1:	& R^{(M\times N)}		& \to 	 &\R_+\\
			& \sum a_{(m,n)}(m,n)	& \mapsto & \sum |a_{(m,n)}|_R\cdot|m|_M\cdot |n|_N
\end{array}
$$
along the quotient map $R^{(M\times N)}\to M\otimes_A N$ is called the projective
tensor product seminorm and denoted $\|\cdot\|_p:M\otimes_A N\to \R_+$.
This gives a symmetric monoidal structure $\otimes_{R,1}$ on $\Mod(R)$.
The initial seminormed ring (resp. uniform seminormed ring) is given by the ring
of integers $\Z$ together with its archimedean seminorm $|\cdot|_\infty$.
Indeed, $\Z$ is the initial ring, and the maximal ring seminorm on $\Z$
must fulfill $|1|=1$ and $|1+1|=|1|+|1|=2$, so that all ring seminorms on $\Z$
are bounded by this one.
If $(R_i,|\cdot|_{i})_{i\in I}$ is a finite family of seminormed rings,
the projective tensor product seminorm on $\otimes_{(\Z,|\cdot|_\infty)} (R_i,|\cdot|_i)$
is a ring seminorm that gives the coproduct of the family.
We already showed that seminormed rings have coequalizers.
The same is true for uniform seminormed rings using the uniformization functor $U$.
This shows the statement about finite colimits of seminormed rings.
Now the non-empty finite limits of the underlying $\R_+$-graded sets in
$\R_{+\leq 1}^\Sets$ give limits in seminormed rings and uniform seminormed rings.
\end{proof}

\section{Dagger algebras}
We will start this section by explaining the motivations underlying the introduction
of the category of dagger algebras.

An important drawback of Banach rings like $R\langle \rho^{-1}T\rangle$ or $R\{\rho^{-1}T\}$
is that they are often non-noetherian (for example, if $\rho=1$ and
$R=\Cbf$, one gets the ring of continuous functions on the complex unit disc
that are analytic on its interior).
Moreover, Banach rings of convergent power series don't have a nice differential calculus,
at least over $\Qbf_p$, since the Poincar\'e lemma fails for them.
These two facts are our main motivations for the introduction of rings of
overconvergent power series.

A power series $f\in R[[T]]$ in one variable over a Banach ring $R$ is overconvergent
on a disc of radius $\rho$ if it converges on a disc of radius $\nu>\rho$.
We would like to define the ring of overconvergent power series as the
filtered colimit of the rings $R\langle \nu^{-1}T\rangle$. However, this
filtered colimit in the category of Banach rings is simply $R\langle\rho^{-1} T\rangle$.
So we need to work in a category that contains Banach rings and that allows us to keep
track of the overconvergence properties. This will be the category of ind-Banach rings,
that is already discussed (in the complex situation) in Bambozzi's thesis
\cite{Bambozzi}, Section 3.3. We will thus define the ring $R\langle \rho^{-1}T\rangle^\dagger$
of $\ell^1$-overconvergent power series on a given polydisc of radius $\rho$
as the uniform ind-Banach ring given by the filtered colimit of the Banach rings
$R\langle \nu^{-1}T\rangle$ for $\nu>\rho$. One may then define $\ell^1$-dagger
affinoid algebras as ind-Banach rings that are isomorphic to algebras of the form
$R\langle \rho^{-1}T\rangle^\dagger/I$ for $I$ a finitely presented ideal.

Since we want to get back $p$-adic analytic spaces in the sense of Berkovich
when we work over the Banach ring $\Qbf_p=(\Q_p,|\cdot|_p)$, we will not work with general Banach
rings, but with uniform Banach rings. If we denote $\otimes$ the coproduct
in the category $\BanRings_u$ of uniform Banach rings, we get for example
for $\Zbf=(\Z,|\cdot|_\infty)$, natural isomorphisms
$$\Qbf_p\{\rho^{-1}T\}\cong \Zbf\{\rho^{-1}T\}\otimes_\Zbf\Qbf_p$$
and
$$\Cbf\{T\}\cong \Zbf\{\rho^{-1}T\}\otimes_\Zbf\Cbf,$$
where we use the two contracting morphisms $\Zbf=(\Z,|\cdot|_\infty)\to (\C,|\cdot|_\C)=\Cbf$
and $\Zbf=(\Z,|\cdot|_\infty)\to (\Q_p,|\cdot|_p)=\Qbf_p$ to extend the scalars.
The coproduct of uniform Banach rings thus exactly gives, by scalar extension,
the rings of analytic functions on a disc that we would like to use both in the
$p$-adic and in the complex case.

We will thus define the ring $R\{\rho^{-1}T\}^\dagger$ of $\ell^\infty$-overconvergent power series
on a given polydisc of radius $\rho$ as the uniform ind-Banach ring given by the filtered
colimit of the uniform Banach rings $R\{\nu^{-1}T\}$ for $\nu>\rho$.

An important drawback of the category of uniform Banach rings is that they
are reduced, so that they exclude the use of nilpotent elements, that has proved (since Weil)
to be so useful to do differential calculus in a geometric setting.
This will lead us to the definition of the category of $\ell^\infty$-dagger algebras
over a given Banach ring, that is inspired by the work of Dubuc-Zilber \cite{Dubuc-Zilber} on
analytic models for synthetic analytic geometry, of Lurie \cite{Lurie-DAG-V} on
general geometries and of Porta on complex derived geometry \cite{Porta-these}.
A dagger rational domain in an overconvergent polydisc $D^\dagger(0,\rho)$ (given by
the Berkovich spectrum of the ind-Banach ring $R\{\rho^{-1}T\}^\dagger$, to be defined
in this section) is a subspace
defined by finitely many inequalities of the form
$$
D(f_0,\rho_0|f_1,\rho_1,\dots,f_n,\rho_n)=\{x\in D^\dagger(0,\rho),\;\rho_0|f_i(x)|\leq \rho_i|f_0(x)|\},
$$
for $(f_i)$ generating the underlying ring of $R\{\rho^{-1}T\}^\dagger$ as an ideal and $\rho_i>0$.
One may define the ring of overconvergent functions on $D$ as the quotient
$$R\{(\rho_i/\rho_0)^{-1}T_i\}^\dagger/(f_0T_1-f_1,\dots,f_0T_n-f_n)$$
in the category of uniform ind-Banach rings, and morphisms of dagger rational
domains will be defined as morphisms of the corresponding uniform ind-Banach algebras
over $R$.

A dagger algebra $A$ over a Banach ring $R$ will be a ``functor of functions''
$$A:\Rat^\dagger_{(R,|\cdot|_R)}\to \Sets$$
on the category $\Rat_{(R,|\cdot|_R)}^\dagger$ of dagger rational domains in finitely
generated overconvergent power series rings $R\{\rho^{-1}T\}^\dagger$ for varying $\rho$
and $T$, that commutes with finite products and sends pullback diagrams of the form
$$
\xymatrix{
D_1\ar@{^{(}->}[r]\ar[d] 	& D_2\ar[d]\\
D_3\ar@{^{(}->}[r]^i		& D_4}
$$
with $i$ an embedding of a sub-rational domain, to pullbacks.
There is a natural fully faithful embedding
$$\RatAlg_{(R,|\cdot|_R)}^\dagger:=\Rat_{(R,|\cdot|_R)}^{\dagger,op}\to \Alg^\dagger_{(R,|\cdot|_R)}$$
of the category of rational domain algebras to the category of dagger algebras.

\subsection{Overconvergent power series}
We will first use the indization of the category of Banach and uniform Banach rings to
get a convenient category of overconvergent algebras over a general
Banach ring. We will use \cite{Kashiwara-Schapira-categories-and-sheaves}, Chapter 6,
as a general reference on indization of categories.

There is a natural Yoneda embedding
$$
\begin{array}{ccc}
\BanRings		& \to		 &\BanRings^\vee=\Hom(\BanRings^{op},\Sets)\\
(A,|\cdot|_A)	& \mapsto & (B,|\cdot|_B)\mapsto
					     \Hom_{\BanRings}\left((B,|\cdot|_B),(A,|\cdot|_A)\right)
\end{array}
$$
and a natural Yoneda embedding
$$
\begin{array}{ccc}
\BanRings_u		& \to		 &\BanRings_u^\vee=\Hom(\BanRings_u^{op},\Sets)\\
(A,|\cdot|_A)		& \mapsto & (B,|\cdot|_B)\mapsto
					     \Hom_{\BanRings_u}\left((B,|\cdot|_B),(A,|\cdot|_A)\right).
\end{array}
$$

\begin{definition}
The categories $\ind\BanRings$ (resp. $\ind\BanRings_u$) of \emph{ind-Banach (resp. uniform
ind-Banach) rings} is the full subcategory of $\BanRings^\vee$ (resp. $\BanRings_u^\vee$)
whose objects are isomorphic to filtered colimits of Banach rings (resp. uniform Banach rings).
\end{definition}

Recall that the natural uniformization functor
$$\BanRings\to \BanRings_u$$
commutes with finite colimits, but that this is not the case of the natural embedding
$$\BanRings_u\to \BanRings.$$
For example, the uniform tensor product is not isomorphic to the projective tensor product.

Recall some basic facts about ind-objects in a given category $\Cc$.
If $\alpha:I\to \Cc$ and $\beta:J\to \Cc$ are two filtrant systems in $\Cc$,
we may compute the morphisms between the corresponding ind-objects by setting
$$
\Hom_{\ind\Cc}(\colim_i \alpha(i),\colim_j \beta(j)):=
\lim_i\colim_j\Hom_{\Cc}(\alpha(i),\beta(j)).
$$
If $f:A\to B$ is a morphism between two ind-objects in $\Cc$,
there exists a filtrant category $I$, two functors $\alpha:I\to \Cc$ and $\beta:I\to\Cc$
and a morphism of functors $\phi:\alpha\to \beta$ such that $\colim_I\, \phi=f$ in $\ind\Cc$.
More generally, if $f,g:A\to B$ is a pair of parallel morphisms between two ind-objects
in $\Cc$, there exists a filtrant category $I$, two functors $\alpha:I\to \Cc$ and
$\beta:I\to \Cc$ and two morphisms of functors $\phi,\psi:\alpha\to \beta$
such that $f=\colim_I\,\phi$ and $g=\colim_I\, \psi$. This implies that the cokernel in $\ind\Cc$
of $f$ and $g$ may be computed as the colimit of the objectwise cokernels
$$\coker\left(f,g:A\to B\right)=\colim_I\left(\coker\left(\phi(i),\psi(i):\alpha(i)\to \beta(i)\right)\right).$$
We have more generally the following result (see \cite{Kashiwara-Schapira-categories-and-sheaves},
Proposition 6.1.18):
\begin{proposition}
\label{indization-commute-colimits}
Let $\Cc$ be a category that admits cokernels (resp. finite coproducts, resp. finite colimits).
Then $\ind\Cc$ admits cokernels (resp. small coproducts, resp. small colimits)
and the natural embedding $\Cc\to \ind\Cc$ commutes with cokernels (resp. finite
coproducts, resp. finite colimits).
\end{proposition}

Since $\BanRings$ and $\BanRings_u$ have finite colimits,
$\ind\BanRings$ and $\ind\BanRings_u$ have small colimits and
the natural embeddings
$$\BanRings\to \ind\BanRings\;\textrm{ and }\;\BanRings_u\to \ind\BanRings_u$$
send finite colimits to finite colimits,
so that the coproduct of a diagram of (uniform) Banach rings seen as a (uniform) ind-Banach
ring is simply given by their (uniform) projective tensor product, and the cokernel of a pair
of morphisms of (uniform) Banach rings is given by the quotient (uniform) Banach ring.
More generally, the colimit of a finite diagram of (uniform) ind-Banach rings may be computed
by using filtered colimits of (uniform) projective tensor products and (uniform) quotients of the 
component (uniform) Banach rings of the given diagram.

In particular, the Banach ring $\Zbf=(\Z,|\cdot|_\infty)$ is both the initial ind-Banach and
uniform ind-Banach ring.

\begin{example}
Here is an interesting example of a non-trivial ind-Banach algebra.
For every finite \'etale extension $K$ of $\Q$, we denote
$\Obf_K=(O_K,\|\cdot\|_\infty)$ its ring of integers equipped with the norm
$$\|f\|_\infty:=\max_{\sigma:K\to \C}|\sigma(f)|_\infty.$$
Then the ring of integers $\bar{\Z}$ of $\bar{\Q}$ may be equipped with
an ind-Banach ring structure $\bar{\Zbf}$ given by the fact that it is the union of
all $O_K$ for $K$ finite \'etale over $\Q$. There is a natural action
of $\Gal(\bar{\Q}/\Q)$ on $\bar{\Zbf}$.
\end{example}

The forgetful functor $\Ringrm:\BanRings\to \Rings$ giving the underlying ring of a Banach ring
extends to two ``underlying ring'' functors
$$\Ringrm:\ind\BanRings\to \Rings\;\textrm{ and }\;\Ringrm:\ind\BanRings_u\to \Rings,$$
which is given by taking the filtered colimit of the underlying rings of a given diagram
of (uniform) Banach rings.

\begin{definition}
The \emph{Berkovich spectrum} $\Mc(A)$ of an ind-Banach ring $A$ is the set of equivalence
classes of morphisms $\chi:A\to (K,|\cdot|)$ where $(K,|\cdot|)$ is a multiplicatively
normed Banach field. If $A$ is a Banach ring, we equip $\Mc(A)$ with the coarsest topology
that makes all evaluation maps
$$
\begin{array}{ccccc}
|a(\cdot)|:	& \Mc(A)		& \to 	& \R_+\\
		& |\cdot(x)|	& \mapsto	& |a(x)|
\end{array}
$$
for $a\in A$ continuous.
If $A$ is an ind-Banach ring, we equip $\Mc(A)$ with the projective limit topology
$$\Mc(A)=\lim\Mc(A_i)$$
for a description $A=\colim A_i$ of $A$ as a colimit of Banach rings.
If $x$ is a point of $\Mc(A)$, the minimal Banach field $(\Kc(x),|\cdot|_x)$ in the corresponding
equivalence class is called the \emph{residue field at $x$}.
\end{definition}

The Berkovich spectrum is clearly functorial.

We now define overconvergent power series over a given Banach ring.
\begin{definition}
Let $(X,|\cdot|_X)$ be an $\R_+$-graded set. A grading $|\cdot|_1$ on $X$ is called
an \emph{over-grading} if
$$|x|_X<|x|_1\textrm{ for all $x\in X$.}$$
If $(R,|\cdot|_R)$ is a (uniform) seminormed ring, an over-grading of $|\cdot|_R$ that is
also a (uniform) seminorm will be called a (uniform) over-seminorm.
Let $R$ be a Banach ring and $X$ be an $\R_+$-graded set.
The ind-Banach ring of \emph{$\ell^1$-overconvergent power series} on $R$ is defined as
the (formal) filtered colimit, i.e., ind-Banach ring given by
$$
R\langle X\rangle^\dagger:=
\underset{|\cdot|_1:X\to \R_+\textrm{ over-seminorms}}{\colim} R\langle(X,|\cdot|_1)\rangle.
$$
If $R$ is unifom, the uniform ind-Banach ring of \emph{overconvergent power series} on $R$
is defined as the (formal) filtered colimit, i.e., ind-uniform Banach ring given by
$$R\{X\}^\dagger:=\underset{|\cdot|_1:X\to \R_+\textrm{ over-seminorms}}{\colim} R\{(X,|\cdot|_1)\}.$$
\end{definition}
In particular, if $\rho=(\rho_1,\dots,\rho_n)\in \R_{>0}^n$ is a polyradius, the
ind-Banach ring $R\langle\rho_1^{-1}T_1,\dots,\rho_n^{-1}T_n\rangle^\dagger$ of
$\ell^1$-overconvergent power series of radius $\rho$ is given by the colimit
$$
R\langle\rho_1^{-1}T_1,\dots,\rho_n^{-1}T_n\rangle^\dagger:=
\colim_{\nu>\rho} k\langle\nu_1^{-1}T_1,\dots,\nu_n^{-1}T_n\rangle,
$$
and the uniform ind-Banach ring $R\{\rho_1^{-1}T_1,\dots,\rho_n^{-1}T_n\}^\dagger$ of
overconvergent power series of radius $\rho$ is given by the colimit
$$
R\{\rho_1^{-1}T_1,\dots,\rho_n^{-1}T_n\}^\dagger:=
\colim_{\nu>\rho} k\{\nu_1^{-1}T_1,\dots,\nu_n^{-1}T_n\}.
$$

We will now recall basic properties of uniform Banach rings of overconvergent
power series. Similar properties also hold for their $\ell^1$ versions.
Recall that the Banach ring $R\{\rho^{-1}T\}$ has the following universal property:
for every bounded morphism of uniform Banach rings $R\to R'$ and every element $f\in R'$ such that
$|f|_{R'}\leq \rho$, there exists a unique commutative diagram:
$$
\xymatrix{R\{\rho^{-1}T\}\ar[r] & R'\\
R\ar[u]\ar[ur]}
$$
The uniform ind-Banach ring $R\{\rho^{-1}T\}^\dagger$ has the following universal property:
there is an isomorphism natural in $R'=\colim R_i'$:
$$
\Hom_R(R\{\rho^{-1}T\}^\dagger,\colim_i R_i')\cong
\Hom(R,R')\times \lim_{\rho'>\rho}\colim_i \{f\in R_i',\;|f|_{R_i'}\leq \rho'\}.
$$
There is a natural coproduct on the category of overconvergent power series algebras which
is given by
$$
R\{\rho^{-1} T\}^\dagger\otimes_R R\{\nu^{-1} S\}^\dagger\cong R\{\rho^{-1} T,\nu^{-1} S\}^\dagger
$$
More generally, this formula also works for quotients by finitely generated ideals
$$
R\{\rho^{-1} T\}^\dagger/I\otimes_R R\{\nu^{-1} S\}^\dagger/J\cong
R\{\rho^{-1} T,\nu^{-1} S\}^\dagger/(I,J)
$$
because the indization functor commutes with finite colimits.

The formation of overconvergent power series also commutes with extensions of the
base Banach ring, meaning that for $f:R\to S$ a bounded morphism, we have
$$S\{X\}^\dagger\cong R\{X\}^\dagger\otimes_R S.$$
This allows us to define overconvergent power series over general ind-Banach rings.
\begin{definition}
Let $R$ be an ind-Banach (resp. a uniform ind-Banach) ring and $(X,|\cdot|_X)$ be
an $\R_+$-graded set.
Let $\Zbf=(\Z,|\cdot|_\infty)\to R$ be the canonical morphism.
The ind-Banach (resp. uniform ind-Banach) ring of overconvergent power series
on $R$ with variables in $X$ is defined by
$$
\begin{array}{c}
R\langle X\rangle^\dagger:=\Zbf\langle X\rangle^\dagger\otimes_\Zbf R\\
\textrm{ (resp. } R\{X\}^\dagger:=\Zbf\{X\}^\dagger\otimes_\Zbf R\textrm{).}
\end{array}
$$
\end{definition}

\subsection{Rational domains and dagger algebras}
In this subsection, we will denote $R$ a fixed ind-Banach ring.
We want to describe rational domains in the overconvergent polydisc
$D^\dagger_R(0,\nu):=\Mc(R\{\nu^{-1}T\}^\dagger)$
defined by finitely many inequalities of the form
$$D(f_0,\rho_0|\rho_1,f_1,\dots,\rho_n,f_n)=\{x\in D^\dagger(0,\nu),\;\rho_0|f_i(x)|\leq \rho_i|f_0(x)|\},$$
for $(f_i)$ generating $\Ringrm(R\{\nu^{-1}T\}^\dagger)$ as an ideal and $\rho_i>0$.
This will be done by defining directly the associated ind-Banach algebras.

\begin{definition}
Let $A$ be an ind-Banach ring and
let $(f_0,\dots,f_n)\in \Ringrm(A)^{n+1}$ be a finite family of
elements of the underlying ring of $A$ that generate $\Ringrm(A)$ as an ideal, and
$\rho=(\rho_0,\dots,\rho_n)\in \R_{>0}^{n+1}$.
\begin{enumerate}
\item The \emph{rational domain (resp. uniform rational domain) algebra} $A\langle\rho_0,f_0|\rho_1,f_1,\dots,\rho_n,f_n\rangle$
(resp. $A\{\rho_0,f_0|\rho_1,f_1,\dots,\rho_n,f_n\}$)
is the Banach (resp. uniform Banach) algebra over $A$ given by the quotient
$$
\begin{array}{c}
A\langle(\rho_1/\rho_0)^{-1}T_1,\dots,(\rho_n/\rho_0)^{-1}T_n\rangle/(f_0T_1-f_1,\dots,f_0T_n-f_n)\\
\left(\textrm{resp. }
A\{(\rho_1/\rho_0)^{-1}T_1,\dots,(\rho_n/\rho_0)^{-1}T_n\}/(f_0T_1-f_1,\dots,f_0T_n-f_n)
\right).
\end{array}
$$
\item The \emph{overconvergent rational domain (resp. uniform overconvergent rational domain)
algebra} $A\langle\rho_0,f_0|\rho_1,f_1,\dots,\rho_n,f_n\rangle^\dagger$
(resp. $A\{\rho_0,f_0|\rho_1,f_1,\dots,\rho_n,f_n\}^\dagger$)
is the ind-Banach (resp. uniform ind-Banach) algebra over $A$ given by the quotient
$$
\begin{array}{c}
A\langle(\rho_1/\rho_0)^{-1}T_1,\dots,(\rho_n/\rho_0)^{-1}T_n\rangle^\dagger/(f_0T_1-f_1,\dots,f_0T_n-f_n)\\
\left(\textrm{resp. }
A\{(\rho_1/\rho_0)^{-1}T_1,\dots,(\rho_n/\rho_0)^{-1}T_n\}^\dagger/(f_0T_1-f_1,\dots,f_0T_n-f_n)
\right).
\end{array}
$$\item The \emph{associated rational domain (resp. dagger rational domain)} is the morphism
$$
\begin{array}{c}
D(\rho_0,f_0|\rho_1,f_1,\dots,\rho_n,f_n):=
\Mc(A\{\rho_0,f_0|\rho_1,f_1,\dots,\rho_n,f_n\})
\longrightarrow
\Mc(A)\\
\left(\textrm{resp. }
D^\dagger(\rho_0,f_0|\rho_1,f_1,\dots,\rho_n,f_n):=
\Mc(A\{\rho_0,f_0|\rho_1,f_1,\dots,\rho_n,f_n\}^\dagger)
\longrightarrow
\Mc(A)
\right).
\end{array}
$$
\item The family of morphisms
$$
\begin{array}{c}
\left\{D(f_i,\rho_i|f_0,\rho_0\dots,\widehat{f_i,\rho_i},\dots,f_n,\rho_n)\longrightarrow
\Mc(A)\right\}_i\\
\left(\textrm{resp. }
\left\{D^\dagger(f_i,\rho_i|f_0,\rho_0\dots,\widehat{f_i,\rho_i},\dots,f_n,\rho_n)\longrightarrow
\Mc(A)\right\}_i
\right)
\end{array}
$$
is called the \emph{standard covering} associated to the family $\{(f_i,\rho_i)\}_{i=0,\dots,n}$.
\item A dagger rational domain algebra (resp. rational domain covering) is called \emph{strict} if
the polyradius $\rho=(\rho_0,\dots,\rho_n)$ has all components equal to $1$.
\end{enumerate}
\end{definition}

It is clear from the definition and properties of the coproduct of ind-Banach and
uniform ind-Banach rings that the pushout of a rational domain (resp. uniform rational domain)
algebra along an arbitrary morphism of ind-Banach (resp. uniform ind-Banach)
rings is again a rational domain (resp. uniform rational domain) algebra,
and that a rational domain (resp. uniform rational domain) algebra
on a rational domain (resp. uniform rational domain) algebra over $R$ is a
rational domain (resp. uniform rational domain) algebra over $R$.
The same results are also true in the overconvergent setting.

We will denote $\RatAlg_R^{an}$ (resp. $\RatAlg^\dagger_R$)
the category of uniform ind-Banach algebras that are isomorphic to uniform rational domain algebras
over power series algebras $R\{\nu^{-1}T\}$ (resp. uniform overconvergent rational domain
algebras over overconvergent power series algebras $R\{\nu^{-1}T\}^\dagger$) for various
multiradii $\nu$.
We will denote $\RatAlg^{an,s}_R$ (resp. $\RatAlg^{\dagger,s}_R$)
the subcategory of strict rational domain algebras (resp. strict dagger rational domain
algebras) over power series (resp. overconvergent power series) algebras
of polyradii $(1,\dots,1)$. We will also denote $\RatAlg_R^{\dagger^1}$ the category
of ind-Banach algebras that are isomorphic to dagger rational domain algebras over
overconvergent power series algebras $R\langle\nu^{-1}T\rangle^\dagger$ for
various multiradii $\nu$ and $\RatAlg^{\dagger^1,s}_R$ the subcategory of
strict algebras.
To treat all these categories in a unified formalism, we introduce the following
notation.
\begin{definition}
A \emph{type of analytic spaces} is an element $t$ of the set
$$\{an,\{an,s\},\dagger,\{\dagger,s\},\{\dagger^1\},\{\dagger^1,s\}\}.$$
The category of rational domain $t$-algebras is denoted $\RatAlg^t_R$.
If $t$ is a type of analytic spaces and $A$ is an ind-Banach ring,
we denote $A(\rho^{-1}T)^t$ the algebra of $t$-convergent power series over $A$
and $A(\rho_0,f_0|\rho_1,f_1,\dots,\rho_n,f_n)^t$ the standard $t$-rational
domain algebra.
\end{definition}
For example, we will have
$$
\begin{array}{rcl}
A(T)^{an} & := & A\{T\},\\
A(\rho^{-1}T)^{\dagger} & :=	& A\{\rho^{-1}T\}^\dagger,\\
A(\rho_0,f_0|\rho_1,f_1,\dots,\rho_n,f_n)^{\dagger^1} & := & A\langle \rho_0,f_0|\rho_1,f_1,\dots,\rho_n,f_n\rangle^\dagger,
\end{array}
$$
and the notation $A(\rho^{-1}T)^{an,s}$ means that $\rho=1$ and that we work with $A\{T\}$.

\begin{proposition}
\label{uniform-completion-rational}
The uniformization and Banach completion functors give a
natural diagram
$$
\xymatrix{
\RatAlg_R^{\dagger^1,s}\ar[d]\ar[r]^u 	&
\RatAlg_R^{\dagger,s}\ar[d]\ar[r]		& \RatAlg_R^{an,s}\ar[d]\\
\RatAlg_R^{\dagger^1}\ar[r]^u			&
\RatAlg_R^\dagger\ar[r]				& \RatAlg_R^{an}}
$$
whose arrows commute with finite coproducts and pushouts along
rational domain embeddings and whose vertical arrows are fully faithful.
If $R$ is a uniform ind-Banach ring, the uniformization arrows $u$ are equivalences.
\end{proposition}
\begin{proof}
The only non-formal point is the isomorphism statement.
It follows from Poineau's article \cite{Poineau2}, Corollary 1.8 (see also Bambozzi and
Ben-Bassat \cite{Bambozzi-Ben-Bassat}, Theorem 8.2 in the case $R=\Zbf$). The basic
idea is to show that for $\rho<\nu<\mu$, there is a restriction morphism
$$R\langle\mu^{-1}T\rangle\to R\{\nu^{-1}T\},$$
which will induce the inverse morphism of the canonical morphism
$$R\{\rho^{-1}T\}^\dagger\to R\langle\rho^{-1}T\rangle^\dagger.$$
\end{proof}

\begin{example}
\label{examples-rational-domain-algebras}
Let $(R,|\cdot|)$ be a uniform Banach ring.
\begin{enumerate}
\item If $R$ is an integral ring equipped with the trivial seminorm $|\cdot|_0$,
the strictly convergent or overconvergent power series are given by polynomials, and
we will see in Proposition \ref{strict-rational-algebraic} that strict
rational domain algebras are given by localizations.
\item One may look at $\Z_p$ and $\Q_p$ as (non-strict) rational domain algebras
over the base Banach ring $R=\Zbf_0:=(\Z,|\cdot|_0)$, by using the formulas
$$
\Z_p=\Oc(\{2\cdot|p|\leq |1|\}):=\Zbf_0\{2T\}/(T-p)
$$
and
$$
\Q_p=\Oc(\{2\cdot |p|\leq |1|,\;|1|\leq 3\cdot |p|\}):=\Zbf_0\{2T,(1/3)S\}/(T-p,pS-1).
$$
This will play an essential role in the analytic derived de Rham cohomology approach to
$p$-adic period rings, that we will discuss in Section \ref{derived-dagger-analytic-geometry}.
\item Over $R=\Qbf_p=(\Q_p,|\cdot|_p)$, the above notions of rational domains give back
the usual rational domains of Berkovich's geometry
\cite{Berkovich1},
and the strict rational domains of Tate's rigid analytic geometry \cite{Tate2}.
Their overconvergent analogs were already used by Gro{\ss}e-Kl\"onne \cite{Grosse-Kloenne}.
\item Over $R=\Cbf=(\C,|\cdot|_\infty)$, we find back essentially the same strict overconvergent
affinoid rational domains as those used by Bambozzi in his thesis \cite{Bambozzi}.
\end{enumerate}
\end{example}

\begin{example}
\label{examples-global}
Let $\Zbf:=(\Z,|\cdot|_\infty)$ be the Banach ring of integers with its archimedean absolute value.
\begin{enumerate}
\item The strict convergent (resp. overconvergent) power series are given by polynomials,
equipped with the sup norm (resp. over-seminorms of the sup norm) on the unit polydisc.
Non strict rational domain algebras over them include
ind-Banach rings like $\Z_p$ and $\Q_p$ (isomorphic to the ind-Banach rings
$\Z_p$ and $\Q_p$ described in Example \ref{examples-rational-domain-algebras}),
but also the ind-Banach ring
$$\R=\Oc(\{2\cdot|1|\leq |2|\}):=\Zbf\{2T\}/(2T-1).$$
\item Strict rational domains over the polynomial ring $\Zbf[X]$ (equipped with its sup norm
on the global unit disc $D(0,1)_\Z$), will be for example given by
$$\{|nX|\leq |1|\},\textrm{ and more generally }\{|qsX-ps|\leq |r|\},$$
with $(n,p,q,r,s)\in \N^5$ such that $qsX-ps$ and $r$ are relatively prime in $\Z[X]$.
The archimedean fiber of these two examples give the disc $D(0,1/n)$ and
(when $s$ and $q$ are non-zero) the discs $D(p/q,r/s)$ with arbitrary rational center
and arbitrary rational radius, intersected with the unit disc. Remark that this intersection
may be empty.
\item The natural map $\Zbf\{X_0,1/X_0\}^\dagger\to \Zbf\{X_1,1/X_1\}^\dagger$ given by
$X_0\mapsto 1/X_1$ gives a well defined isomorphism between the rational domains
$\{1\leq |X_0|\leq 1\}$ and $\{1\leq |X_1|\leq 1\}$, with underlying algebras of functions
the Zariski domain algebras $\Z[X_0,1/X_0]$ and $\Z[X_1,1/X_1]$ on the polynomial
algebras $\Z[X_0]$ and $\Z[X_1]$. We will see in Proposition \ref{projective-Arakelov} that
this allows us to put on the algebraic  projective space $\Proj(\Z[X_1,X_1])$ the structure of
a strict dagger analytic space over the Banach ring $\Zbf$, giving a kind of additional
``Arakelov structure'' on it.
\item The natural ``multiplicative'' comultiplication $\Delta_m$ on $\Z[X_0,1/X_0]$ naturally
extends to a comultiplication $\Delta_m^{\dagger}$ on $\Zbf\{X_0,1/X_0\}^\dagger$.
We will call the corresponding analytic group over $\Zbf$ simply $U(1)$. Over
the non-archimedean base $\Zbf_0=(\Z,|\cdot|_0)$, one has $U(1)=\G_m$, but
over $\R$, we get $U(1)_\R$ and over $\Q_p$, we also get $U(1)_{\Q_p}$.
This may be quite disturbing to an algebraic geometer that the natural base
extension (i.e., analytification) of the algebraic multiplicative group
$\G_m$ to $K=\R$ or $\Q_p$ is not $\G_{m,K}^{an}$, but $U(1)_{K}$.
\end{enumerate}
\end{example}

\begin{proposition}
\label{rational-pre-geometries}
The opposite category to $\RatAlg_R^t$ for
$t\in \{an,\{an,s\},\dagger,\{\dagger,s\},\dagger^1,\{\dagger^1,s\}\}$
a type of analytic spaces, equipped with the admissible subcategories of rational domain
embeddings and the Grothendieck topology generated by rational domain coverings
is a pre-geometry in the sense of Lurie \cite{Lurie-DAG-V}, Definition 3.1.1.
\end{proposition}
\begin{proof}
To simplify notations, we will write the proof for the case $t=an$, and it applies
directly to other cases by replacing analytic power series $R\{\rho^{-1}T\}$ by
$t$-convergent power series $R(\rho^{-1}T)^t$.
We already know that there are finite coproducts and pushout diagrams along
rational domain embeddings.
Let us show that if a triangle
$$
\xymatrix{&Y\ar[dr]^g&\\
X\ar[ur]^f\ar[rr]^h&&Z}
$$
of rational domains is such that $g$ and $h$ are rational domain embeddings, then $f$
is also a rational domain embedding. Indeed, suppose we have a commutative
diagram (multi-index notation)
$$
\xymatrix{&R\{\rho^{-1}T\}/(sT-r)\ar[dr]^f&\\
R\ar[ur]^g\ar[rr]^h&&R\{\nu^{-1}S\}/(uS-v)}
$$
Then we get a natural morphism
$$
\left(R\{\rho^{-1}T\}/(sT-r)\right)\{\nu^{-1}S\}/(uS-v)\longrightarrow R\{\nu^{-1}S\}/(uS-v)
$$
that induces a natural morphism
$$
R\{\rho^{-1}T\}/(sT-r)\{\nu^{-1}S\}/(uS-v,T-f(T))\longrightarrow R\{\nu^{-1}S\}/(uS-v).
$$
This last map is an isomorphism, so that $f$ is indeed a rational domain embedding.
This argument extends to the overconvergent setting.
Let us now show that every retract of a rational domain embedding is a rational domain embedding.
Suppose given a retraction diagram
$$
\xymatrix{
A\ar[r]^i\ar[d]\ar@/^1.5pc/[rr]^{\id}	& R\ar[r]^r\ar[d]				& A\ar[d]\\
B\ar[r]^(0.22)i\ar@/_1.5pc/[rr]_{\id}		& R\{\rho^{-1}T\}/(gT-f)\ar[r]^(0.75)r	& B}
$$
By construction, there is a canonical retraction
$$
\xymatrix{
B\ar[r]_(0.22){i_1}\ar@/^1.8pc/[rr]^{\id}		& A\{\rho^{-1}T\}/(r(g)T-r(f))\ar[r]_(0.75){r_1}	& B}
$$
where the map $i_1$ is given by the composition
$$
B\overset{i}{\longrightarrow}R\{\rho^{-1}T\}/(gT-f)\overset{r}{\longrightarrow} A\{\rho^{-1}T\}/(r(g)T-r(f))
$$
and the map $r_1$ is given by the universal property of the quotient.
It remains to show that $i_1\circ r_1=\id$ to get that $i_1$ and $r_1$ are inverse isomorphisms,
which will finish the proof.
\end{proof}

For $t\in \{an,\{an,s\},\dagger,\{\dagger,s\},\dagger^1,\{\dagger^1,s\}\}$ a type of analytic spaces,
the ``underlying ring'' functor
$$\Ringrm:\ind\BanRings\to \Rings$$
restricts naturally to an ``underlying $R$-algebra'' functor
$$\Algrm:\RatAlg_R^t\to \Alg_R.$$
If $f:R\to S$ is a morphism of ind-Banach algebras, there are natural
base extension functor
$$\textrm{-}\otimes_R S:\RatAlg^t_R\to \RatAlg^t_S.$$
For $u\in \{an,\dagger\}$, there is also a natural fully faithful
embedding
$$\RatAlg^{u,s}_R\to \RatAlg^u_R.$$

We will now introduce the categories of analytic and dagger algebras, that will allow
the use of nilpotent elements, which is not possible with uniform ind-Banach rings.
We carefully inform the reader that the category $\Sets$ used in this definition
is a category of small enough sets.
\begin{definition}
For $t\in \{an,\dagger,\dagger^1\}$ a non-strict type of analytic spaces,
a \emph{$t$-algebra (resp. a very strict $t$-algebra)} will be a functor
$$
A:(\RatAlg^t_R)^{op}\to \Sets\;\textrm{ (resp. }A:(\RatAlg^{t,s}_R)^{op}\to \Sets\textrm{)}
$$
that sends finite coproducts of algebras to products, and pushout diagrams of the form
$$
\xymatrix{
A_1\ar[r]^i\ar[d] 	& A_2\ar[d]\\
A_3\ar[r]		& A_4}
$$
where $i:A_1\to A_2$ corresponds to the embedding of a rational sub-domain,
to pullbacks.
We denote $\Alg^t_R$ the category of $t$-algebras and $\Alg^{t,vs}_R$
the category of very strict $t$-algebras. Objects of
$\Alg^{an}_R$ (resp. $\Alg^{\dagger}_R$, resp. $\Alg^{\dagger^1}_R$)
will be called \emph{analytic (resp. uniform dagger, resp. dagger) algebras}
over $R$.
\end{definition}

We have thus introduced three new (very strict) types of analytic
spaces $\{an,vs\}$, $\{\dagger,vs\}$ and $\{\dagger^1,vs\}$.

For $t\in \{an,\dagger,\dagger^1\}$ a non-strict type of analytic spaces,
there is are natural Yoneda embeddings
$$\RatAlg_R^t\to \Alg^t_R\textrm{ and }\RatAlg_R^{t,s}\to \Alg^{t,vs}_R$$
that commutes by definition with finite coproducts and pushout diagrams of the form
$$
\xymatrix{
A_1\ar[r]^i\ar[d] 	& A_2\ar[d]\\
A_3\ar[r]			& A_4}
$$
where $i$ corresponds to an embedding of a rational sub-domain.

\begin{proposition}
Let $t\in \{an,\{an,vs\},\dagger,\{\dagger,vs\},\dagger^1,\{\dagger^1,vs\}\}$ be
a type of analytic spaces.
If $f:R\to S$ is a morphism of ind-Banach algebras,
there is a natural base extension functor
$$\textrm{-}\otimes_R S:\Alg^t_R\to \Alg^t_S.$$
\end{proposition}
\begin{proof}
The base extension functor is given by the following construction:
every $A$ in $\Alg^t_R$ may be presented
as a colimit of representable functors
$$
A=\colim_i \Hom_{\RatAlg^t_R}(-,A_i)=
\colim_i\Hom_{\ind\BanRings_{R\backslash}}(-,A_i),
$$
with $A_i\in \RatAlg^t_R$.
One then defines $A\otimes_R S$ by
$$A\otimes_R S:=\colim_i \Hom_{\ind\BanRings_{S\backslash}}(-,A_i\otimes_R S),$$
where the colimit is taken in the category $\Alg^t_S$.
\end{proof}
By construction,
the diagram
$$
\xymatrix{
\RatAlg^t_R\ar[rr]^{-\otimes_R S}\ar[d]	&& \RatAlg^t_S\ar[d]\\
\Alg^t_R\ar[rr]^{-\otimes_R S}			&& \Alg^t_S}
$$
commutes.

The following lemma will imply that strict analytic geometry on a
trivially seminormed ring is essentially equivalent to scheme theory.
\begin{proposition}
\label{strict-rational-algebraic}
If $R$ is an integral ring with its trivial seminorm, then the categories
of very strict dagger and of very strict analytic algebras are equivalent to the category of
usual $R$-algebras, and rational domains are given by localizations.
\end{proposition}
\begin{proof}
This follows from the fact that rational domain algebras correspond to localization of
polynomial algebras, and that a functor that sends coproducts to products and
localization pushouts to pullbacks on them is equivalent to a functor on polynomial algebras
that sends coproduct to products,
which is well known (since Lawvere \cite{Lawvere-functorial-semantics}) to be
the same thing as an arbitrary algebra.
So let us prove that rational domain algebras in polynomial rings over a trivially
normed integral ring correspond to their localizations. First remark that the
spectral seminorm on the unit polydisc for the polynomial algebra on $R$
is given by the trivial seminorm. Indeed, it is the uniform seminorm associated
to the canonical Banach norm, given by
$$\left|\sum a_\alpha X^\alpha\right|:=\sum |a_\alpha|_0,$$
which gives
$$\left|\sum a_\alpha X^\alpha\right|_0=\max\left(|a_\alpha|_0\right).$$
If $A=R[X_1,\dots,X_n]$ is a polynomial algebra and $f_0,\dots,f_n$
are elements in $A$ such that there exist $a_i\in A$ with
$1=\sum_{i\in I} a_i f_i$ then for every point in $\Mc(A)=\Mc(A,|\cdot|_0)$,
we have $|a_i(x)|\leq |a_i|_0\leq 1$ and
$$1=|1|\leq \max_{i\in I}|f_i(x)|\leq \max_{i\in I}|f_i|_0\leq 1$$
so that
$$\{x,\;|f_i(x)|\leq |f_0(x)|\textrm{ for all }i>0\}=\{x,\;1\leq |f_0(x)|\}.$$
If we look at the corresponding rational domain algebra, this gives the quotient
$$R[X_1,\dots,X_n]\{f_0,1|1,1\}=R[X_1,\dots,X_n,T]/(f_0T-1),$$
equipped with the quotient seminorm of the trivial norm on the polynomial ring,
that is also the trivial norm in this case, so that we get
$$\{x,\;1\leq |f_0(x)|\}=\{x,\;f_0(x)\neq 0\}.$$
so that
$$A\{f_0,1|f_1,1,\dots,f_n,1\}=A\{f_0,1|1,1\}=A[1/f_0].$$
The same argument applies also to the overconvergent setting, since
the Berkovich spectrum does not change.
\end{proof}

For $t\in \{an,\dagger,\dagger^1\}$ a non-strict type of analytic spaces,
the underlying $R$-algebra functor $\Algrm:\RatAlg_R^t\to \Alg_R$ extends naturally
to a functor $\Algrm:\Alg^t_R\to \Alg_R$ given by
$$
\Algrm(A):=\colim_{\rho\to\infty}A(R(\rho^{-1}T)^t)
$$
with $T$ a variable.

\begin{definition}
If $A$ is an ind-Banach $R$-algebra and $t\in \{an,\dagger,\dagger^1\}$
is a type of analytic spaces,
the \emph{$t$-completion of $A$} is the $t$-algebra defined by setting its values on
a rational domain $t$-algebra $B$ in $\RatAlg^t_R$ to be
$$
A^t(B):=\Hom_{\ind\BanRings_{R\backslash}}(B,A).
$$
\end{definition}

By definition, the $t$-completion $A^t$ of $A$ naturally commutes with all pushout
diagrams, and in particular with those needed to be a $t$-algebra over $R$.
This construction thus gives a $t$-completion functor
$$
\begin{array}{cccc}
(\cdot)^t: 	& \ind\BanRings_{R}	& \to 	 & \Alg^t_R\\
		& A				& \mapsto & A^t
\end{array}
$$

\begin{example}
We may apply the dagger completion to the ind-Banach ring $\bar{\Zbf}$ over $(\Z,|\cdot|_\infty)$,
to get a dagger algebra $\bar{\Zbf}^\dagger$ equipped with a natural $\Gal(\bar{\Q}/\Q)$-action.
\end{example}

We may apply the $t$-completion to morphisms $R\to (K,|\cdot|)$ to a
multiplicatively normed Banach field $(K,|\cdot|)$, to define the Berkovich spectrum
of a $t$-algebra.
\begin{definition}
Let $A$ be a $t$-algebra over $R$ for $t\in \{an,\dagger,\dagger^1\}$ a non-strict type
of analytic spaces.
The \emph{berkovich spectrum of $A$} is the set $\Mc(A)$ of equivalence
classes of morphisms
$$\chi:A\to (K,|\cdot|)^t$$
for the family of morphisms $R\to (K,|\cdot|)$ from $R$ to a multiplicatively seminormed field,
together with the coarsest topology that makes the maps
$$\Mc(A)\to \R_+$$
given by $x\mapsto |a(x)|$ for $a\in \Algrm(A)$ continuous.
\end{definition}

The following proposition shows that using spectra of dagger algebras does not give
more general Berkovich spectra. As we said before, the interest is to allow the introduction of 
nilpotent elements, which will be useful for the development of differential calculus.
\begin{proposition}
Let $A$ be a $t$-algebra over $R$ for $t\in \{an,\dagger,\dagger^1\}$ a non-strict type
of analytic spaces.
There exists an ind-Banach ring (that is uniform if $A$ is uniform)
$\ib(A)$ that is initial with the property that there exists a morphism
$$A\to \ib(A)^t.$$
The natural morphism
$$\Mc(\ib(A))\to \Mc(A)$$
is an isomorphism.
\end{proposition}
\begin{proof}
By abstract nonsense, the $t$-algebra $A$ is a canonical colimit of (representable)
rational domain algebras $A_i$. Setting $\ib(A)$ to be the colimit of these rational domain algebras
$A_i$ in the category of uniform ind-Banach rings, we find that $\ib(A)$ has the desired universal
property. The map $\Mc(\ib(A))\to \Mc(A)$ is a bijection by the universal
property of $\ib(A)$ and it is clearly bicontinuous.
\end{proof}

\begin{definition}
We will say that $A$ is a \emph{uniform} $t$-algebra if the canonical morphism
$$A\to \ib(A)^t$$
is an isomorphism.
\end{definition}

\begin{proposition}
\label{uniform-completion-algebra}
The diagram described in Proposition \ref{uniform-completion-rational} induces
the following diagram of categories of algebras:
$$
\xymatrix{
\Alg_R^{\dagger^1,vs}			 	&
\Alg_R^{\dagger,vs}\ar[l]_u			& \Alg_R^{an,vs}\ar[l]\\
\Alg_R^{\dagger^1}\ar[u]				&
\Alg_R^\dagger\ar[l]_u\ar[u]			& \Alg_R^{an}\ar[u]\ar[l]}
$$
The arrows $u$ are equivalences if $R$ is uniform.
\end{proposition}
\begin{proof}
This follows from the fact that all the functors of loc. cit. commute with coproducts
and pushouts along rational domain immersions.
\end{proof}

\begin{remark}
\label{convergent-analytic-algebras}
Let $(R,|\cdot|_R)$ be a uniform Banach ring.
Proposition \ref{uniform-completion-algebra} allows us to relate
dagger algebras to usual analytic algebras, that are usually used on
a non-archimedean Banach field, e.g., by Berkovich \cite{Berkovich1} and
Tate \cite{Tate2}.
The problem with the approach to analytic geometry using convergent power
series and analytic algebras is that convergent
power series on polydiscs are not well behaved in the archimedean setting (e.g., not
Noetherian), and that it may lead to looking at the unit circle $S^1$ with its algebra of
continuous complex valued functions as an analytic space (see \cite{Berkovich1}, Remark 1.5.5).
As explained by Bambozzi in his thesis, the overconvergent setting solves this problem
with $S^1$ (see \cite{Bambozzi}, Example 6.4.38).
It is also well known that $p$-adic de Rham cohomology behaves better in the overconvergent
setting \cite{Grosse-Kloenne}.
\end{remark}

\subsection{Affinoid algebras}
We will now introduce the basic building blocks in our approach to analytic geometries,
that will be finitely generated, in some sense. A first natural idea is to use the 
($0$-truncated) envelopping geometries of the rational pre-geometries described in
Proposition \ref{rational-pre-geometries}, in the sense of Lurie \cite{Lurie-DAG-V}, Definition 3.4.9.
In the strict case, this will give a very restrictive notion of affinoid algebra over $\Zbf$,
that one may think of as a notion of affinoid algebra over the analytic field
with one element $\F_{\{\pm 1\}}$ (its algebraic version if described by Durov in \cite{Durov-2007}).
Since general projective varieties over $\Z$ may
not have a model over this deeper base, we will also introduce another notion of
strict affinoid algebra, that is better suited to our purpose of studying
projective schemes as strict global analytic spaces.
\begin{definition}
\label{affinoid-R-algebra}
Let $t\in \{an,\dagger,\dagger^1\}$ be a non-strict type of analytic spaces.
A $t$-algebra $A$ over $R$ that is (isomorphic to) a coequalizer
$$
\xymatrix{
C\ar@<0.7ex>[r]^f\ar@<-0.4ex>[r]_g & B\ar[r]	& A}
$$
of two morphisms in $\RatAlg^t_R$ (i.e., that is finitely presented)
is called a \emph{$t$-affinoid $R$-algebras}. We denote
$\Aff^t_R$ the category of $t$-affinoid $R$-algebras.
A $t$-affinoid $R$-algebra is called a \emph{strict $t$-affinoid $R$-algebra}
if $B$ above can be chosen to be a strict $t$-algebra. It is called a \emph{very
strict $t$-affinoid $R$-algebra} if the above diagram comes from a diagram of
strict rational $t$-algebras. We will denote $\Aff^{t,s}_R$ the category
of strict affinoid $t$-algebras and $\Aff^{t,vs}_R$ the category of
very strict $t$-affinoid $R$-algebras.
\end{definition}

By definition, for $t\in \{an,\dagger,\dagger^1\}$, there is a natural sequence of fully faithful functors
$$\RatAlg^t_R\to \Aff^t_R\to \Alg^t_R.$$
We also have a natural sequence of fully faithful functors
$$\RatAlg^{t,s}\to \Aff^{t,vs}_R\to \Aff^{t,s}_R\to \Aff^t_R\to \Alg^t_R$$
and a fully faithful functor
$$\Aff^{t,vs}_R\to \Alg^{t,vs}_R.$$

For all $t\in \{an,\dagger,\dagger^1\}$ the forgetful functor on $\Alg^t_R$ restricts to
a natural forgetful functor
$$
\Algrm: \Aff^t_R \to  \Alg_R
$$
that induces forgetful functors
$$\Algrm:\Aff^{t,s}_R\to \Alg_R\;\textrm{ and }\;\Algrm:\Aff^{t,vs}_R\to \Alg_R.$$

\begin{definition}
Let $t\in \{an,\dagger,\dagger^1\}$ be a non-strict type of analytic spaces.
The category of \emph{strict $t$-algebras} is the category
$$\Alg^{t,s}_R:=\ind\Aff^{t,s}_R.$$
\end{definition}

We now show that the difference between strict and very strict algebras vanishes
if we work on non-trivially valued fields. Of course, on Banach rings such as
$\Zbf=(\Z,|\cdot|_\infty)$, they give two distinct notions.
\begin{proposition}
Let $(R,|\cdot|_R)$ be a non-trivially seminormed field and $t\in \{an,\dagger,\dagger^1\}$
be a non-strict type of analytic spaces.
Then the natural functors
$$\Aff^{t,vs}_R\to \Aff^{t,s}_R\textrm{ and }\Alg^{t,vs}_R\to \Alg^{t,s}_R$$
are equivalences. 
\end{proposition}
\begin{proof}
We only have to prove the essential surjectivity for affinoid algebras.
It follows from the fact that $|R|_R$ has arbitrarily large elements,
so that the coequalizer defining a strict affinoid algebra may be
chosen to be between two very strict rational algebras.
\end{proof}

\begin{proposition}
\label{strict-affinoid-algebraic}
Let $t\in \{an,\dagger,\dagger^1\}$ be a non-strict type of analytic spaces.
If $R=(R,|\cdot|_0)$ is an integral ring equipped with its trivial seminorm, then
the natural functors
$$\Aff^{t,vs}_R\to \Aff^{t,s}_R\textrm{ and }\Alg^{t,vs}_R\to \Alg^{t,s}_R$$
are equivalences that identify strict affinoid $t$-algebras to
finitely presented $R$-algebras and strict $R$-algebras to arbitrary $R$-algebras.
\end{proposition}
\begin{proof}
Every strict affinoid $t$-algebra $A$ may be written as a coequalizer
$$
\xymatrix{
C\ar@<0.7ex>[r]^f\ar@<-0.4ex>[r]_g & B\ar[r]	& A}
$$
with $C=R(\rho_1^{-1}Y_1,\dots,\rho_n^{-1})^t$, and, by Proposition \ref{strict-rational-algebraic},
we may also suppose that $B=R(X_1,\dots,X_n,h^{-1})^t$ (with underlying ring
the localized polynomial algebra).
The pair of morphisms $(f,g)$ thus correspond to a family $\{f_i,g_i\}$ of elements in the
localized polynomial algebra, that may also be seen as a pair of morphisms $(f,g)$ as
above with $C=R(Y_1,\dots,Y_n)^t$. This shows that $A$ is a very strict $t$-algebra.
The rest of the proposition follows from Proposition \ref{strict-rational-algebraic}.
\end{proof}

\begin{proposition}
Let $t\in \{an,\dagger,\dagger^1\}$ be a non-strict type of analytic spaces.
The category (opposite to) $\Aff^t_R$ (resp. $\Aff^{t,vs}_R$) is the $0$-truncated
geometric envelope of the pre-geometry defined by $\RatAlg^t_R$ (resp. $\RatAlg^{t,s}_R$)
in the sense of Lurie, \cite{Lurie-DAG-V}, Definition 3.4.9.
The natural functor
$$
\ind\Aff^t_R\to \Alg^t_R\;
\textrm{ (resp. }\ind\Aff^{t,vs}_R\to \Alg^{t,vs}_R\textrm{)}
$$
is an equivalence of categories and the functor
$$
\Alg^{t,s}_R=\ind\Aff^{t,s}_R\to \Alg^t_R
$$
is a fully faithful embedding.
The category $\Aff^{t,s}_R$ is the (opposite) underlying category of a geometry in the sense of
Lurie, \cite{Lurie-DAG-V}, Definition 1.2.5.
\end{proposition}
\begin{proof}
The two first statements essentially follows from our definition of affinoid and very
strict affinoid algebras. It remains to show that strict affinoid $t$-algebras form a
geometry, with strict rational domain embeddings as admissible morphisms.
We know that affinoid $t$-algebras form the geometric envelope of the pre-geometry
given by rational $t$-algebras (this follows from the construction of the geometric envelope,
that uses the colimit completion process of \cite{Lurie-higher-topos-theory}, Proposition 5.3.6.2).
It remains to show that strict affinoid $t$-algebras
form a sub-geometry of this geometry. They are clearly stable by
pushouts, retractions and have the two out of three property
(following the proof of Proposition \ref{rational-pre-geometries}). Strict rational domain
algebras over strict affinoid $t$-algebras are also stable by pushouts.
This shows that they indeed form a sub-geometry of the geometry given by affinoid
$t$-algebras, so that $\Aff^{t,s}_R\subset \Aff^{t}_R$ is a fully faithful embedding.
The last statement follows by passing to categories of ind-objects.
\end{proof}

\subsection{Perfectoid dagger algebras}
We discuss shortly the perfectoid analogs (see \cite{Scholze-perfectoid-spaces} for
an introduction to perfectoid spaces) of the constructions of the previous
subsections, that may be useful for non-archimedean applications.
\begin{definition}
A \emph{perfectoid field} is a non-archimedean Banach field $K$ of
residue characteristic $p$ whose seminorm is non-discrete and whose
Frobenius map is surjective on the quotient ring $K^\circ/p$. 
\end{definition}

\begin{definition}
Let $t=\{an,\dagger,\dagger^1\}$ be a non-strict type of analytic algebras.
Let $K$ be a perfectoid field with uniformizer $\omega\in K^\circ$.
Let $\rho\in \R_{>0}^n$ be a multiradius.
The non-strict $t$-perfectoid $K$-algebra of power series of multiradius
$\rho$ is the ind-Banach algebra
$$
K((\rho^{-1}T)^{1/p^\infty})^t:=\left(\colim_{n\to\infty}
K^\circ({\rho}^{-1}T,{{\rho}^{-1/p^n}}S)^t/(S^{p^n}-T)\right)[\omega^{-1}].
$$
\end{definition}

Scholze \cite{Scholze-perfectoid-spaces} defines a tilt operation on convergent
power series that sends $K\{T^{1/p^\infty}\}$ to $K^\flat\{T^{1/p^{\infty}}\}$.
Diekert's master thesis \cite{Diekert-these-master} describes the Tilt operation on
overconvergent power series of radius $1$. This can actually be extended to non-strict
$t$-convergent power series, by using the formula
$$K((\rho^{-1}T)^{1/p^\infty})^t\mapsto K^\flat((\rho^{-1}T)^{1/p^\infty})^t.$$

We may define perfectoid rational domain $t$-algebras and
$t$-perfectoid algebras in a way similar to the one used in the previous
section, by replacing everywhere the rings of $t$-convergent power series
by those of $t$-convergent perfectoid $K$-algebra of power series.
This will give rise to categories $\RatAlg^{t,perf}_K$ and $\Alg^{t,perf}_K$
of $t$-perfectoid rational algebras and $t$-perfectoid algebras over $K$
that may be used to define a perfectoid version of $t$-geometry, that gives
back Scholze's perfectoid spaces \cite{Scholze-perfectoid-spaces}
in the strict setting for $t=an$. As pointed out by Scholze, this category only contains
reduced rings, so that it is not adapted to the infinitesimal definition of differential notions.
The tilt operation extends directly to rational domain algebras giving
a natural tilting functor
$$\RatAlg^{t,perf}_K\to \RatAlg^{t,perf}_{K^\flat}$$
that extends, since it is compatible with coproducts and pushforwards along
rational domain immersions, to a tilting functor
$$\Alg^{t,perf}_K\to \Alg^{t,perf}_{K^\flat}.$$
This will also give an equivalence between the corresponding \'etale and pro-\'etale topoi,
to be define in \ref{topologies}.

Since a $t$-rational domain algebra over a perfectoid algebra is still perfectoid
(see \cite{Scholze-perfectoid-spaces}, Theorem 6.3 for the strictly analytic situation),
we may define $t$-perfectoid rational domain algebras by using rational domain
algebras in $t$-perfectoid algebras of power series. This gives a natural fully faithful
functor
$$\RatAlg^{t,perf}_K\to \Alg^t_K.$$
This functor sends pushouts along rational domain immersions to pushouts,
so that it extends to a functor
$$\Alg^{t,perf}_K\to \Alg^t_K.$$
This allows us to do infinitesimal calculus on perfectoid dagger algebras
by using infinitesimal extensions in the category of all dagger algebras.

\begin{remark}
One may use the above viewpoint to propose a global notion of perfectoid algebras.
Let $\Zbf_0:=(\Z,|\cdot|_0)$ be the non-archimedean Banach ring of integers
and $\Qbf$ its fraction field.
One may easily define as above the global perfectoid $t$-convergent
power series of multiradius $\rho$ by setting
$$
\Qbf\{(\rho^{-1}T)^{1/\N^\infty}\}^t:=
\left(\colim_{n\to\infty}
\Zbf_0({\nu}^{-1}T,{{\nu}^{-1/n}}S)^t/(S^{n}-T)\right)\otimes_{\Zbf_0}\Qbf.
$$
The fields $\Kbf=(\Q(\mu^\infty),|\cdot|_0)$ or $\Kbf=(\bar{\Q},|\cdot|_0)$
may play the role of the perfectoid base field.
One may then easily define the category $\RatAlg^{t,perf}_\Kbf$ and
$\Alg^{t,perf}_\Kbf$ as before. Remark that there is no clear
analog of the tilt operation in this situation, but Witt vector constructions with
rings of integers in the spirit of those done
by Davis and Kedlaya in \cite{Davis-Kedlaya-almost-purity}  may be enough to define interesting period rings and period sheaves.
It is also possible to use $\hat{\Z}$-completed derived de Rham
cohomology of $\bar{\Zbf}/\Zbf$ as proposed by Bhatt in \cite{Bhatt-derived-de-Rham},
Remark 10.22.
\end{remark}

\section{Overconvergent geometry}
In all this section, $R$ will denote an ind-Banach ring.
We will now define dagger analytic spaces over $R$, that are essentially given
by pasting dagger algebras along their natural coverings by rational
domains. This pasting operation can be done using $\Sets$-valued
sheaves on the category of dagger algebras. To ease comparisons
with other kinds of analytic spaces, we will also define convergent and $\ell^1$-dagger
analogs of dagger spaces.

\subsection{Dagger analytic spaces}
We first extend to affinoid $t$-algebras the notion of rational domain,
and also define rational coverings. Let $u\in \{an,\dagger,\dagger^1\}$ be a non-strict
type of analytic spaces and $t$ be one of the types $t=u$ or $t=\{u,s\}$ or $t=\{u,vs\}$.
\begin{definition}
Let $A$ be an affinoid $t$-algebra over $R$. Let $(f_0,\dots,f_n)\in \Ringrm(A)^{n+1}$ be
a finite family of elements of $A$ that generate $\Algrm(A)$ as an ideal, and
$\rho=(\rho_0,\dots,\rho_n)\in \R_{>0}^{n+1}$.
\begin{enumerate}
\item The \emph{rational domain $t$-algebra}
$A(\rho_0,f_0|\rho_1,f_1,\dots,\rho_n,f_n)^t$ is the affinoid $t$-algebra over $A$
given by the quotient
$$
A((\rho_1/\rho_0)^{-1}T_1,\dots,(\rho_n/\rho_0)^{-1}T_n)^t/(f_0T_1-f_1,\dots,f_0T_n-f_n).
$$
\item The \emph{associated rational domain} is the morphism
$$
D^t(\rho_0,f_0|\rho_1,f_1,\dots,\rho_n,f_n):=
\Mc(A(\rho_0,f_0|\rho_1,f_1,\dots,\rho_n,f_n)^t)
\longrightarrow
\Mc(A).
$$
\item The family
$$
\left\{D^t(f_i,\rho_i|f_0,\rho_0\dots,\widehat{f_i,\rho_i},\dots,f_n,\rho_n)\longrightarrow
\Mc(A)\right\}_i
$$
is called the \emph{standard covering} associated to the family $\{(f_i,\rho_i)\}_{i=0,\dots,n}$.
\item A rational domain algebra (resp. rational domain) is called \emph{strict} (resp. very strict) if
$A$ is strict (resp. very strict) and the polyradius $\rho=(\rho_0,\dots,\rho_n)$ has all
components equal to $1$.
\end{enumerate}
\end{definition}

By definition of the category $\Aff^t_R$ of affinoid $t$-algebras, the natural functor
$$i:\RatAlg^t_R\to \Aff^t_R$$
sends rational domain algebras on a rational domain algebra $R$ to rational domain
algebras over the associated affinoid $t$-algebra $R^t$, so that
the notation of the above definition for rational domain algebras remains consistent with
the previous one.

Let $(\Alg^t_R,\tau_{\Rat})$ be the category
of $t$-algebras over $R$ equipped with the topology
given by standard coverings by rational domains (naturally
induced from the topology on $\Aff^t_R$ and the equivalence
$\Alg^t_R\cong \ind\Aff^t_R$, as in \cite{Lurie-DAG-V}, Definition 2.4.3).
\begin{definition}
We now work in the category $\Shv_{\widehat{\Sets}}(\Alg^t_{R},\tau_{\Rat})$ of sheaves
on $t$-algebras over $R$ with values in a category $\widehat{\Sets}$ of
\emph{big} sets.
We denote $\Mb(A)$ the sheaf associated to the representable presheaf
given by $A$.
\begin{enumerate}
\item A morphism $D\to X$ of sheaves is called a \emph{rational domain}
if its pullback along any morphism $\Mb(A)\to Y$ is a rational domain.
\item A finite family of rational domains $\{D_i\to X\}_{i\in I}$ is called a standard rational covering
if its pullback along any morphism $x:\Mb(A)\to X$ is a standard rational covering.
\item A morphism $D\to X$ is called a \emph{quasi-compact domain}
if it is a union (colimit) of finitely many rational domains.
\item A morphism $U\to X$ is called a \emph{pre-domain} if
it is the colimit of an arbitrary family of rational domains.
\item A pre-domain $U\to X$ is called a \emph{domain (or an admissible domain)} if for
all points $x:\Mb(A)\to X$ that
factorize through $U$, there exists a quasi-compact
domain $D\to U$  with a factorization $x:\Mb(A)\to D$.
\item If $U\to X$ is a domain, a family of domains $\{U_i\to U\}$ is called
an \emph{admissible covering} if for all rational domain $D\to U$, the pullback of
the family $\{U_i\to U\}$ to $D$ can be refined by a standard rational covering.
\end{enumerate}
We then define a \emph{$t$-analytic space} to be a sheaf
$X$ in $\Sh(\Alg^t_{R},\tau_{\Rat})$
that can be written as the colimit
$$X=\colim_{\Mb(A)\to X} \Mb(A)$$
along the system of all representable domains $\Mb(A)\to X$.
The category of $t$-analytic spaces is denoted
$\An^t_R$.
\end{definition}

\begin{remark}
As explained by Temkin in \cite{Temkin-intro-Berkovich-spaces},
it is difficult to describe a general open affine subscheme but one can easily
characterize it by a universal property. For example, the Zariski open $\A^2-\{0\}$ of
the affine plane $\A^2$ is not described by a localization of the polynomial ring, but as
the union of the two subspaces given by taking out the axis, with coordinate rings
$\Z[x,y,1/x]$ and $\Z[x,y,1/y]$. The analytic analog of general Zariski open
subsets are given by (admissible) domains.
\end{remark}

We may now define the notion of finitely presented analytic space,
which is necessary to explain the relation of $t$-analytic spaces
with complex analytic or $p$-adic analytic geometries.
\begin{definition}
A morphism $f:X\to Y$ of $t$-analytic spaces is called:
\begin{enumerate}
\item affine and finitely presented if for every affinoid $t$-algebra point $\Mb(A)\to Y$, the
pullback is an affinoid $t$-algebra point $\Mb(B)\to X$ such that $B$ is
an affinoid $t$-algebra over $A$;
\item locally finitely presented if it is locally affine and finitely presented for
the $G$-topology.
\end{enumerate}
An analytic space is locally finitely presented if the morphism $X\to \Mb(R)$ is
locally finitely presented.
\end{definition}

For $u\in \{an,\dagger\}$,
the natural functors $\Aff^{u,vs}_R\to \Aff^{u,s}_R\to \Aff^{u}_R$ are compatible
with the rational domain topologies, so that they induces functors
between categories of sheaves, which gives ``destrictification'' functors
$$\An^{u,vs}_R\to \An^{u,s}_R\to \An^{u}_R.$$

\begin{remark}
To illustrate the difference between strict and very strict analytic spaces, we may look
at them over the base Banach ring $\Zbf=(\Z,|\cdot|_\infty)$. A very strict analytic space
over $\Zbf$ is something that one may think of as an analytic space over the field with
one element $\F_{\{\pm 1\}}$ (in analogy with what Durov does in the algebraic
setting in \cite{Durov-2007}). It is thus a very rigid object, and the category of
very strict analytic spaces over $\Zbf$ looks like being quite poor.
A strict analytic space over $\Zbf$ is something more similar in nature to what one
usually uses to call an Archimedean compactification in Arakelov geometry. We will discuss
this further in Section \ref{Dagger-analytic-Arakelov}.
\end{remark}

\begin{remark}
\label{strict-fully-faithful}
It is likely that the above destrictification functors are fully faithful at least in the non-archimedean
case, following arguments similar to the ones used by Temkin \cite{Temkin-local-properties-II}.
In the global case over $\Zbf=(\Z,|\cdot|_\infty)$, Temkin's graded approach to the description of
the points of the $G$-topology can't be directly adapted because of the lack of archimedean
rings of integers (see however Remark \ref{Durov-ring-of-integers} for a possible replacement),
but one may try (following a suggestion of Temkin) to start
from loc. cit., Proposition 2.5 to show that the fibers of the corresponding
morphism of $G$-topoi are connected.
\end{remark}

\begin{remark}
The destrification functor sometimes plays the role of the analytification functor:
if we work on a trivially seminormed integral ring $(R,|\cdot|_0)$,
we have shown in Proposition \ref{strict-affinoid-algebraic} that strict dagger algebras are
simply usual $R$-algebras, and rational domains are given by localizations.
Then domains correspond to Zariski open subsets, and strict analytic
spaces to usual schemes. The associated non-strict analytic spaces
give their non-archimedean analytification over $(R,|\cdot|_0)$.
One may probably identify the points of the $G$-topology on the non-strict analytic
spectrum of $(R,|\cdot|_0)$ with the valuation spectrum of $R$. This gives a
relation between non-strict analytic geometry in the trivially valued case
and Krull and Zariski's valuative approach to algebraic geometry. This relation
of course extend to Huber's approach to non-archimedean geometry,
since the Huber space of an affinoid ring over a $p$-adic field may
be identified with the space of points of the $G$-topology on the corresponding
strict analytic (i.e., rigid) space.
\end{remark}

\begin{remark}
Let $u\in \{an,\dagger,\dagger^1\}$ be a non-strict type of analytic spaces.
We must warn the reader here about the interpretation of algebraic geometry
as strict analytic geometry over a trivially valued ring:
looking at a scheme $X$ over a ring $R$ as a strict $u$-analytic space $X^{u,s}$ over
$(R,|\cdot|_0)$ with $u\in \{an,\dagger\}$, is not a completely harmless operation:
for example, the base extension of the strict $u$-unit disc
over $\Zbf_0:=(\Z,|\cdot|_0)$, with function algebra $(\Z[X],|\cdot|_0)$ along
$\Zbf_0\to \Qbf_p:=(\Q_p,|\cdot|_p)$, will give the $p$-adic unit
disc. If we want to really get back the (non-strict) analytic affine line $\A^{1,u}_{\Qbf_p}$,
we need to start from the affine line $\A^{1,u}_{\Zbf_0}$, that is a the non-strict analytic
space given by the union
$$\A^1_{\Zbf_0}=\colim_{\rho>0}D^{u,s}(0,\rho)$$
of all discs with arbitrary radius and center $0$. There are thus \emph{two very different
ways} of looking at a scheme over $\Z$: as a non-strict analytic space, which gives
its global analytification over $\Zbf_0$ and as a strict analytic space, which really gives
the corresponding algebraic analytic space (that may be used in a GAGA theorem,
for example).
We will discuss the relation of this problem with Arakelov geometry in Section
\ref{Dagger-analytic-Arakelov}. In particular, we will see that in the case of projective
varieties, the above two kinds of analytifications are identified.
\end{remark}

We may also apply the ``destrictification process'' to the situation over the initial ind-Banach
ring $\Zbf:=(\Z,|\cdot|_\infty)$:
the base extension functor
$$\An^{\dagger,s}_{\Zbf}\to \An^{\dagger,s}_{(\Z,|\cdot|_0)}$$
sends a strict analytic space over $\Zbf$ to the underlying scheme,
and the functor
$$\An_\Zbf^{\dagger,s}\to \An_\Zbf^\dagger$$
sends it to the associated global dagger space.

\begin{proposition}
\label{complex-strict-non-strict}
If $\Cbf=(\C,|\cdot|_\infty)$, the functor
$$\RatAlg^{\dagger,s}_{\Cbf}\to \RatAlg^{\dagger}_{\Cbf}$$
is an equivalence of pre-geometries, meaning that it induces a sequence of equivalences
$$
\An^{\dagger,vs}_\Cbf\overset{\sim}{\longrightarrow} \An^{\dagger,s}_\Cbf
\overset{\sim}{\longrightarrow} \An^{\dagger}_\Cbf.
$$
\end{proposition}
\begin{proof}
This comes from the fact that the value set of $|\cdot|_\infty$ is $\R_+$.
Indeed, we have for $\rho<1$, the isomorphism
$$\C\{\rho^{-1}T\}=\C\{T\}\{S\}/(\rho S-T)$$
and for $\rho>1$, one may cover admissibly the disc of radius $\rho$
by the strict rational domain $\{|T|\leq 1\}$ (which is the spectrum of
$\C\{T\}$) and the rational domain $\{1\leq |T|\leq \rho\}$ (which identifies
with the strict rational domain $\{\rho^{-1}\leq |S|\leq 1\}$ in the unit disc,
with rational domain algebra $\C\{S\}\{T\}/(TS-\rho^{-1})$). One shows
in a similar way that a rational domain over $\Cbf$ may be admissibly covered
by a strict rational covering.
\end{proof}

The base extension functor corresponding to the bounded map $\Zbf\to \Cbf$ thus
gives a kind of ``associated complex analytic space''
$$\An^{\dagger,s}_\Zbf\to \An^{\dagger,s}_\Cbf.$$
We will see in Proposition \ref{projective-Arakelov} that this operation works well in the case of
projective varieties.
However, in the $p$-adic case, we get two distinct functors on $\An^{\dagger,s}_\Zbf$
with values in strict (i.e., rigid analytic) and non-strict $p$-adic dagger spaces over $\Qbf_p$,
because discs with radius not in $\sqrt[\infty]{p^\Z}\subset \R_+^*$
may not be covered admissibly by strict rational domain algebras.

\begin{definition}
Let $t\in \{an,\dagger,\dagger^1\}$ be a non-strict type of analytic spaces and $X$ be a $t$-analytic
space.
The set of points $|X_{top}|$ of the underlying topological space $X_{top}$ is defined to be
the set of equivalence classes of morphisms
$$\Mb((K,|\cdot|_K)^t)\to X$$
for $R\to (K,|\cdot|_K)$ a morphism of uniform ind-Banach rings from $R$
to a complete multiplicatively seminormed field. A morphism of this form that
is minimal in a given equivalence class $x\in X_{top}$ is denoted $\Kc(x)$ and called
the \emph{residue field} at $x$.
If $f:X\to Y$ is a morphism of $t$-analytic spaces, then
there is a canonical map
$$|f_{top}|:|X_{top}|\to |Y_{top}|.$$
Let $X$ be a $t$-analytic space. A subset $U\subset |X_{top}|$ is called
a \emph{Berkovich open subset} if for every morphism $f:\Mb(B)\to X$, with
induced set map
$$|f|:\Mc(B)\cong |\Mb(B)_{top}|\to X_{top},$$
the inverse image set $|f|^{-1}(U)$ is open in $\Mc(B)$.
The \emph{Berkovich topology $\tau_B$} is defined to be the set of Berkovich
open subset of $|X_{top}|$.
We call $X_{top}:=(|X_{top}|,\tau_B)$ the \emph{underlying topological space}
of $X$.
\end{definition}

The underlying topological space is indeed a topological space (clear) and it is
functorial in morphisms of $t$-analytic spaces.
Indeed, if $f:X\to Y$ is a morphism of $t$-analytic spaces, $U\subset |X_{top}|$ is
open, and $g:\Mb(B)\to X$ is a morphism, then $|g|^{-1}(|f|^{-1}(U))=|f\circ g|^{-1}(U)$
is also open, so that $|f|:|X_{top}|\to |Y_{top}|$ is continuous.

We have thus showed that every $t$-analytic space gives rise to
two natural topologies: the $G$-topology, and the Berkovich topology.
There is natural condition to impose on them to get nicely behaved
underlying topological spaces.

If $U=U_{top}\subset X_{top}$ is an open subset, we will still denote $U\subset X$ the sub-functor
defined by the condition that every point $x:\Mb(A)\to X$ such that $x_{top}:\Mc(A)\to X_{top}$
factorizes through $U_{top}$ is in $U(A)$.

\begin{definition}
A $t$-analytic space $X$ is called a \emph{$t$-Berkovich space} if it is locally
finitely presented and every open subset $U\subset X_{top}$ may be covered by admissible domains.
We will denote $\Ber^t_R$ the category of $t$-Berkovich spaces.
\end{definition}

Remark that in a $t$-Berkovich space, every point $x\in |X_{top}|$ has a neighborhood
given by a finite union $\cup_i\Mc(A_i)$ for $\Mb(A_i)\to X$ some rational domains.

The interest of imposing the above conditions is given by the following:
there is a natural continuous morphism of sites
$$\pi:(X,\tau_G)\to (X_{top},\tau_B).$$

Remark that if $\Kbf=(K,|\cdot|_K)\to (L,|\cdot|_L)=\Lbf$ is a Banach field extension, and
$X$ is a $t$-Berkovich space over $\Kbf$, then $X_\Lbf$ is a $t$-Berkovich
space over $\Lbf$, but we may also take the base extension
$$\tilde{X}_\Lbf:=X\times_{\Mb(\Kbf)}\Mb(\Lbf)$$
of $X$ inside the category of $t$-analytic spaces over $\Kbf$,
which will be a $\Kbf$-dagger analytic space but not a $\Kbf$-dagger Berkovich space
anymore, since it is not locally finitely presented in general (i.e., not locally modeled on
$t$-affinoid algebras).

\begin{definition}
Let $u\in \{an,\dagger,\dagger^1\}$ be the underlying non-strict type of a type
$t\in \{an,\{an,s\},\dagger,\{\dagger,s\},\dagger^1,\{\dagger^1,s\}\}$ of analytic spaces.
The $u$-affine line over $R$ is the (non-strict) $u$-space over $R$ defined
by the colimit of discs in one variable over $R$:
$$\A^{1,u}_R:=\colim_\rho D^u_R(0,\rho).$$
The structural sheaf $\Oc^t$ of $t$-algebras on a
$t$-analytic space
$(X,\tau_G)$ is given on a domain $D\to X$ by
$$\Oc^t_{X_G}(D):=\Hom_{\An^u_R}(D,\A^{1,u}_R).$$
If $X$ is a $t$-Berkovich space, the structural sheaf on $X_{top}$ is
given by
$$\Oc^t_{X_{top}}(U)=\pi_*\Oc^t_{X_G}:=\colim_{D\subset U}(\Oc^t_{X_G}(D)),$$
where the colimit is taken in the category $\Alg^{t}_R$ and
over all domains (or even rational domains) contained in $U$.
\end{definition}

Remark that one needs to pass to the non-strict category to define the affine line,
even if we work in the strict case, because when we work with strict analytic spaces
on a trivially valued integral ring $(R,|\cdot|_0)$, i.e., with schemes, there are no
elements in $R$ of arbitrary big seminorm, so that one would need to give another
definition for the affine line in this case (for example, as the unit disc,
that is the usual algebraic affine line in this strict trivially valued situation).

\begin{example}
We now explain how the notion of weak formal scheme from Meredith
\cite{Meredith} is naturally related to dagger analytic spaces over a $p$-adic ring
of integers.
Let $K$ be a complete non trivially valued non-archimedean field, with ring
of integers $\O_K$ and residue field $k$. Let $\pi\in \O_K$ be a uniformizer.
A \emph{weak formal scheme} is a strict dagger analytic space over
$(\O_K,|\cdot|_K)$. The generic fiber of a weak formal scheme $X$ is
the extension of scalars of $X$ along the bounded morphism
$(\O_K,|\cdot|_K)\to (K,|\cdot|_K)$. Its special fiber is the scheme given by
the extension of scalars of $X$ along the bounded morphism
$(\O_K,|\cdot|_K)\to (k,|\cdot|_0)$. The case of formal schemes can be
treated similarly using strict analytic spaces over $(\O_K,|\cdot|_K)$.
It is easy to find weak formal schemes that gives dagger models over $\O_K$ of
finitely presented affine schemes over $k$. Indeed, a natural dagger model
for the affine space $\A^n_{k}$ is of course given by the unit polydisc
$\Db^{n,\dagger}(0,1)$ over $\O_K$, and given finitely many equations and inequations
$f_i=0$ and $g_j\neq 0$ in $k[X_1,\dots,X_n]$, one may extend them
to $\Db^{n,\dagger}(0,1)$ by taking pre-images in $\O_K[X_1,\dots,X_n]$
of the corresponding polynomials, and using the strict domain $D$ defined
by $|f_i|\leq |\pi|$ and $|g_j|\geq 1$, for example. A similar argument applies
to the projective space: the strict projective dagger space over $\O_K$ may
be obtained by pasting various polydiscs (instead of various affine spaces).
There is also a general theorem due to Arabia \cite{Arabia}: every smooth scheme $\bar{X}$
over $k$ can be extended to a smooth weak formal scheme $X$ over $\O_K$.
In the general case, one can always extend a quasi-projective scheme over $k$
to a weak formal scheme over $\O_K$, that is not smooth anymore in general.
One may however see it as a derived dagger space over $\O_K$ in the sense
of Section \ref{derived-dagger-analytic-geometry}, and the (Hodge-completed)
de Rham cohomology of its generic fiber, that is a derived dagger space over
$K$, will give us a nice cohomology theory. It is also quite sure that the (Hodge completed)
derived de Rham cohomology of the dagger model $\bar{X}$ over $\O_K$ is an
interesting invariant, because it also contains an important $p$-torsion information,
similar to the one used in $p$-adic cohomology theories.
\end{example}

\begin{remark}
\label{global-analytic-dagger}
Let $X$ be an analytic space over a Banach ring $\Rbf=(R,|\cdot|_R)$ in the sense
of Berkovich (see \cite{Berkovich1} and \cite{Poineau2}). To every dagger rational
domain algebra $A$ over $R$ we associate the corresponding germ of global analytic subspace
$D_A$ in the corresponding analytic affine space $\A^n_{(R,|\cdot|_R)}$. This
correspondence is actually an equivalence, since we can find back that algebra
using global sections on the germ.
The presheaf $X^\dagger$ on rational domain dagger algebras
defined by setting $X^\dagger(A)$ to be given by the set of morphisms of germs
(pro-global analytic spaces) $D_A\to X$ is a sheaf on the category of rational
domain algebras for the standard rational coverings, and this extends naturally to
a sheaf on the category $\Alg^\dagger_\Rbf$ by writing a dagger
algebra as a colimit of rational dagger algebras. This gives a natural functor
$$(-)^\dagger:\Ber^{glob}_\Rbf\to \An^\dagger_\Rbf$$
from global analytic spaces over $\Rbf$ in the sense of Berkovich
\cite{Berkovich1} to dagger analytic spaces over $\Rbf$.
This will allow us to define the de Rham cohomology of a global analytic space,
and to relate dagger analytic motives over $\C$ to Ayoub's analytic motives.
\end{remark}

\begin{remark}
\label{p-adic-analytic-dagger}
The functor that sends a dagger algebra to the associated analytic algebra,
being compatible with rational domains and standard rational coverings,
induces a natural functor between the corresponding category of sheaves,
which in turn induces a natural functor
$$\An^\dagger_R\to \An_R^{an}$$
from dagger analytic spaces to the category of
convergent analytic spaces. This reduces naturally to a functor
$$\Ber^\dagger_R\to \Ber_R^{an}.$$
If $(R,|\cdot|_R)=(K,|\cdot|_K)$ is a non-archimedean valued field,
this functor is close to being an equivalence, as already shown in
a similar (strict) situation by Gro{\ss}e-Kl\"onne in \cite{Grosse-Kloenne}.
The category $\Ber_K^{an}$ is very close to the category of Berkovich's non-archimedean
analytic spaces from \cite{Berkovich-etale-cohomology}.
If we denote $\Ber_K$ Berkovich's original category, there is a natural
functor
$$\Ber_K\to \Ber^{an}_K$$
that sends $X$ to the sheaf $\Hom_{\Ber_K}(-,X)$. We defined our notion
of dagger Berkovich spaces so that this functor is likely to be an equivalence
of categories, with the aim of making the comparison with Berkovich's theory
easier. Remark that, as shown by Bambozzi in his thesis \cite{Bambozzi}, the theory
of dagger spaces seems better adapted to the archimedean situation.
\end{remark}

\subsection{Tate's acyclicity theorem}
We now follow closely the approach of Tate \cite{Tate2} (see also \cite{BGR}, Chap. 7),
by adapting its archimedean version given by Bambozzi in \cite{Bambozzi}, Chapter 4.
The argument are almost the same, but we write them for the reader's
convenience.
In this section, we denote $(R,|\cdot|_R)$ a given base Banach ring,
and $t=\{an,\{an,s\},\dagger,\{\dagger,s\}\}$ be a uniform type of analytic spaces.
Let $X=\Mb(A)$ be an affinoid $t$-analytic space over $R$.
Recall that we have defined the presheaf $\Oc_X$ on the rational domain
topology by setting
$$\Oc_X(U):=\Hom_{\An^u_R}(U,\A^{1,u}_R),$$
where $u\in \{an,\dagger,\dagger^1\}$ is the non-strict type analytic space associated
to the type $t$.
One may compute this explicitly on rational domain algebras, and this gives
simply the functor
$$B\mapsto \Alg(B)$$
that sends a $t$-rational domain algebra $B=A(f_0,\rho_0|f_1,\rho_1,\dots,f_n,\rho_n)^t$
to the underlying algebra.
We will now check that this presheaf is actually a sheaf.

\begin{lemma}
Let $X=\Mb(A)$ be an affinoid $t$-analytic space over $R$
and $X=\cup_i U_i$ be an affinoid $t$-covering. Then
$$\Oc_X(X)\to \prod_i\Oc_X(U_i)$$
is injective.
\end{lemma}
\begin{proof}
This comes from the definition of $\Oc_X$ as morphisms with values in the (non-strict)
space given by affine line, and that analytic functions (overconvergent or not) are
determined by their germs at every point of the Berkovich space.
\end{proof}

\begin{lemma}
\label{maximum-sup}
Let $A$ be a $t$-affinoid $R$-algebra. For any $f\in \Alg(A)$, there exists
a point $x_0\in \Mc(A)$ such that
$$|f(x_0)|=\sup_{x\in \Mc(A)}|f(x)|.$$
\end{lemma}
\begin{proof}
The dagger analytic situation restricts to the strict situation by using
the fact that the Berkovich spectrum of a dagger affinoid algebra may be
identified with the Berkovich spectrum of the associated analytic affinoid algebra.
One may then use the fact that we work with uniform Banach rings, so
that the given norm is equal to the spectral norm
$$\|f\|_\infty:=\sup_{x\in \Mc(A)}|f(x)|.$$
Since $x\mapsto |f(x)|$ is continuous for the Berkovich topology (by definition)
on the compact topological space $\Mc(A)$, there exists $x_0\in \Mc(A)$ such
that $\|f\|_\infty=|f(x_0)|$.
\end{proof}

\begin{lemma}
\label{minimum-maximum}
Let $X=\Mc(A)$ with $A$ $t$-affinoid $R$-algebra and $f_1,\dots,f_n\in \Alg(A)$,
$\rho_1,\dots,\rho_n\in \R_{>0}$, then the function
$$\alpha(x):=\max_{i=1,\dots,n}\rho_i^{-1}|f_i(x)|$$
assume it's minimum in $X$.
\end{lemma}
\begin{proof}
If the $f_i$'s have a common zero, there is nothing to show. Otherwise, they generate the
unit ideal in $A$, and one may consider the rational covering of $X$ given by
$$
X_i=
D(f_i,\rho_i|f_1,\rho_1,\dots,\widehat{f_i,\rho_i},\dots,f_n,\rho_n)=\{x\in X,\alpha(x)=\rho_i^{-1}|f_i(x)|\}.
$$
By Lemma \ref{maximum-sup}, $\rho_i^{-1}|f_i(x)|$ assume its minimum on $X_i$ (because
$\rho_i|f_i^{-1}(x)|$ has a maximum), and so $\alpha$ has assumes its minimum in $X$,
which is the least of the minimum of the $\rho_i^{-1}|f_i(x)|$ over $X_i$.
\end{proof}

\begin{definition}
Let $A$ be a $t$-affinoid algebra, $f_1,\dots,f_n\in \Alg(A)$, $\rho_1,\dots,\rho_n\in \R_{>0}$
and $X=\Mb(A)$.
Then each
$$\Uc_i=\{D^t(1,1|f_i,\rho_i),D^t(1,1|f_i^{-1},\rho_i^{-1})\}$$
is a $t$-rational covering of $X$. We denote by $\Uc_1\times\cdots\times \Uc_n$ the
covering consisting of all intersections of the form
$U_1\cap \dots\cap U_n$ where $U_i\in \Uc_i$ and call
this the \emph{$t$-Laurent covering} of $X$ generated by $\{(f_i,\rho_i)\}_{i=1,\dots,n}$.
\end{definition}

More explicitly, the elements of the Laurent covering generated by $\{(f_i,\rho_i)\}$ are rational
domains of the form
$$D(1,1|f_1^{\mu_1},\rho_1^{\mu_1},\dots,f_n^{\mu_n},\rho_n^{\mu_n})$$
with $\mu_i=\pm 1$.

\begin{lemma}
\label{rational-laurent-refinement}
Suppose that $t=\{an,\dagger\}$ is a non-strict uniform type of analytic spaces.
Let $\Uc$ be a $t$-rational covering of $X=\Mb(A)$. There exists a $t$-Laurent covering
$\Vc$ of $X$ such that for any $V\in \Vc$, the covering $\Uc_{|V}$ is a $t$-rational
covering of $V$, which is generated by units in $\Oc_X(V)$.
\end{lemma}
\begin{proof}
Let $f_1,\dots,f_n\in \Alg(A)$, $\rho_1,\dots,\rho_n\in \R_{>0}$ be the datum of definition
of the $t$-rational covering $\Uc$. We may chose a constant $c\in \R_{>0}$ such that
$$c^{-1}<\inf_{x\in X}(\max_{i=1,\dots,n}\rho_i^{-1}|f_i(x)|),$$
which is well defined by Lemma \ref{minimum-maximum}.
Let $\Vc$ be the $t$-Laurent covering generated by $\{(f_i,c^{-1}\rho_i)\}_{i=1,\dots,n}$.
Consider the set
$$V=D(1,1|f_1^{\mu_1},(c^{-1}\rho_1)^{\mu_1},\dots,f_n^{\mu_n},(c^{-1}\rho_n)^{\mu_n})\in \Vc,$$
where $\mu_i=\pm 1$. We can assume that there exists an $s\in \{0,\dots,n\}$ such
that $\alpha_1=\cdots=\alpha_s=1$, and $\alpha_{s+1}=\cdots=\alpha_n=-1$.
Then
$$D(f_i,\rho_i|f_1,\rho_1,\dots\widehat{f_i,\rho_i},\dots,f_n,\rho_n)\cap V=\emptyset$$
for $i=1,\dots,s$, since
$$\max_{i=1,\dots,s}\rho_i^{-1}|f_i(x)|\leq c^{-1}<\max_{i=1,\dots,n}\rho_i^{-1}|f_i(x)|$$
for all $x\in V$. In particular, for all $x\in V$, we have
$$\max_{i=1,\dots,n}\rho_i^{-1}|f_i(x)|=\max_{i=s+1,\dots,n}\rho_i^{-1}|f_i(x)|,$$
hence $\Uc_{|V}$ is the rational covering of $V$ gee rated by $\{({f_i}_{|V},\rho_i)\}$ for
$s+1\leq i\leq n$, which are units in $V$ if $t=an$ since their spectral norm is positive.
They are also units in $V$ if $t=\dagger$.
\end{proof}

\begin{lemma}
\label{rational-laurent-refinement-units}
Let $\Uc$ be a $t$-rational covering of $X$ which is generated by units
$f_1,\dots,f_n\in \Oc_X(X)$, then there exists a $t$-Laurent covering $\Vc$ of
$X$ which is a refinement of $\Uc$.
\end{lemma}
\begin{proof}
Let $\Vc$ be the Laurent covering generated by the family $\{(f_if_j^{-1},\rho_i\rho_j^{-1}\}$,
with $1\leq i<j\leq n$. Consider $V\in \Vc$. Defining $I=\{(i,j)\in \N^2,1\leq i<j\leq n\}$, we can
find a partition $I=I_1\coprod I_2$ such that
$$
V=\cap_{(i,j)\in I_1}\left(D(1,1|f_if_j^{-1},\rho_i\rho_j^{-1})\right)\cap
\cap_{(i,j)\in I_2}\left(D(1,1|f_jf_i^{-1},\rho_j\rho_i^{-1})\right).
$$
One defines a partial order on $\{1,\dots,n\}$ requiring that if $(i,j)\in I_1$, then
$i\sim{<}j$ of if $(i,j)\in I_2$ then $j\sim{<} i$. For each $i,j\in \{1,\dots,n\}$ with $i\neq j$,
then $i\sim{<}j$ or $j\sim{<} i$. Consider a maximal chain (which always exists)
$i_1\sim{<}\dots\sim{<} i_r$ of elements of $\{1,\dots,n\}$. Since $i_r$ is maximal,
we have that for any $i\in \{1,\dots,n\}$, $i\sim{<}i_r$, which implies $\rho_i|f_i(x)|\leq \rho_j|f_j(x)|$
for all $x\in V$, i.e.,
$$V\subset D(f_{i_r},\rho_{i_r}|f_1,\rho_1,\dots,\widehat{f_{i_r},\rho_{i_r}},\dots,f_n,\rho_n).$$
\end{proof}

\begin{theorem}[Tate's acyclicity theorem]
Suppose either that the base Banach ring is non-archimedean, or that $t\in \{\dagger,\{\dagger,s\}\}$
is an overconvergent type of analytic spaces.
The presheaf $\Oc_X$ is acyclic for the $G$-topology on $\Mb(A)$.
\end{theorem}
\begin{proof}
Since the $G$-topology is generated by rational coverings, we may reduce
to them. Using Lemma \ref{rational-laurent-refinement} and \ref{rational-laurent-refinement-units},
we can refine a given rational covering by a (non-strict) Laurent covering, and then by induction
to the case of the covering $\,\Uc=\{D(1,1|f,\rho),D(1,1|f^{-1},\rho^{-1})\}$.
The proof now follows closely the Tate's one, adapted by Bambozzi to the archimedean setting
in \cite{Bambozzi}, Theorem 4.0.18: one has to check that the sequence
$$0\to A\to A(\rho^{-1} X\}^t/(X-f)\times A\{\rho Y)^t/(Yf-1)\to A(\rho^{-1} X,\rho Y)^t/(X-f,Xf-1)\to 0$$
is exact, which is done by Tate's tricks of the trade, as explained in loc. cit.
\end{proof}

\subsection{The \'etale, Nisnevich and pro-\'etale topology}
\label{topologies}
Let $R$ be a uniform ind-Banach ring and $t\in \{an,\{an,s\},\dagger,\{\dagger,s\}\}$.

We will now define general notions of differential calculus in a way
that is similar to the method used in scheme theory. We carefully
inform the reader that these notions will be interesting mostly for
dagger analytic spaces, but the general definitions will be useful
for comparison purposes.
\begin{definition}
Let $f:A\to B$ be a morphism of $t$-algebras over $R$.
\begin{enumerate}
\item The morphism $f$ is called
\emph{formally \'etale (resp. formally unramified, resp. formally smooth)} if for
every commutative square
$$
\xymatrix{
A\ar[r]^f\ar[d]	& B\ar[d]\ar@{..>}[dl]\\
C\ar[r]	& C/I}
$$
of $t$-algebras, with $I$ a nilpotent ideal, the dotted arrow
exists and is unique (resp. is unique if it exists, resp. exists).
\item It is called \emph{flat} (resp. \emph{finite}) if the underlying
morphism $\Algrm(f):\Algrm(A)\to \Algrm(B)$ is flat (resp. finite).
\item A morphism $f:X\to Y$ of $t$-analytic spaces is called
\emph{quasi-\'etale (resp. quasi-unramified, resp. quasi-smooth)}
if it is locally finitely presented and it is locally formally \'etale (resp. formally
unramified, resp. formally smooth) for the $G$-topology.
\item A quasi-\'etale morphism $f:X\to Y$ of non-strict Berkovich spaces
is said to \emph{have the Nisnevich property} if
every point $y\in Y_{top}$ has a domain neighborhood $U=\Mb(A)$ such that
$f$ has a section on $U$, meaning that there is a commutative diagram
$$
\xymatrix{& X\ar[d]^f\\
U\ar@{..>}@/^/[ur]\ar[r] & Y}
$$
\end{enumerate}
\end{definition}

\begin{definition}
Let $X$ be a $t$-analytic space over $R$. A family $\Rc=\{Y_i\to X\}_{i\in I}$
of \'etale morphisms is called an \emph{\'etale covering} if, for all
morphism $\Mb(A)\to X$ with $A$ a $t$-algebra over $R$, the family
$\Rc\times_X U:=\{Y_i\times_X \Mb(A)\to \Mb(A)\}_{i\in I}$
can be refined by a finite family $\{\Mb(B_j)\to \Mb(A)\}_{j\in J}$ of \'etale morphisms.
If we further suppose that $X$ and $Y$ are Berkovich spaces, we will say that it is a
 \emph{Nisnevich covering} if the morphism $\coprod_{j\in J}\Mb(B_j)\to \Mb(A)$
has the Nisnevich property.
\end{definition}

We now define the pro-\'etale topology following closely Scholze's approach
from \cite{Scholze-p-adic-Hodge-theory}, Section 3. This topology gives a better take
at completed \'etale cohomology.
Let $X$ be a $t$-analytic space and $X_{fet}$ be the category of finite \'etale
morphisms $f:Y\to X$.
\begin{definition}
The pro-finite \'etate site is defined to be the category $\pro X_{et}$ of pro-objects
of the category $X_{fet}$.
A morphism $U\to V$ of objects of $\pro X_{et}$ is called \'etale (resp. finite \'etale)
if it is induced by an \'etale (resp. finite \'etale) morphism $U_0\to V_0$ of objects
in $X_{et}$ by base extension along a morphism $V\to V_0$. A morphism
$U\to V$ of objects of $\pro X_{et}$ is called pro-\'etale if it can be written as a cofiltered
inverse limit $U=\lim U_i$ of objects $U_i\to V$ \'etale over $V$, such that $U_i\to U_j$
is finite \'etale and surjective for large $i>j$. Such a presentation $U=\lim U_i\to V$ is
called a pro-\'etale presentation.
The pro-\'etale site $X_{proet}$ has as underlying category the full subcategory
of $\pro X_{et}$ of objects that are pro-\'etale over $X$. Finally, a covering in
$X_{proet}$ is given by a family of pro-\'etale morphisms $\{f_i:U_i\to U\}$
such that $\coprod f_i:\coprod U_i\to U$ is an \'epimorphism.
\end{definition}

There is a natural projection
$$\eta:X_{proet}\to X_{et}.$$

We will denote $\Zbf_\ell$ the limit of the constant pro-\'etale sheaves $\Z/\ell^n\Z$,
and $\hat{\Zbf}$ the limit of the constant pro-\'etale sheaves $\Z/n\Z$.
We will also denote $\Qbf_\ell:=\Zbf_\ell\otimes_\Z\Q$ and $\Abf_f:=\hat{\Zbf}\otimes_\Z\Q$.
\begin{definition}
Let $X$ be a $t$-analytic space. The $\ell$-adic (resp. complete integral)
\'etale cohomology of $X$ is defined by
$$
\begin{array}{c}
H^*_{et}(X,\Z_\ell):=H^*(X_{proet},\Zbf_\ell)\\
\textrm{(resp. }H^*_{et}(X,\hat{\Z}):=H^*(X_{proet},\hat{\Zbf})\textrm{).}
\end{array}
$$
\end{definition}

\begin{example}
Let $X$ be a proper non-strict analytic space over $\Zbf_0=(\Z,|\cdot|_0)$ and
denote $\Qbf=(\Q,|\cdot|_0)$. It may be interesting to try to extend Scholze's
result on $p$-adic Hodge theory in
\cite{Scholze-p-adic-Hodge-theory} to give a comparison theorem between
the pro-\'etale cohomology of $X_{\bar{\Qbf}}$ with coefficients in $\hat{\Zbf}$ and the derived
de Rham cohomology of $X/\Zbf_0$. Such an extension has already been discussed in
the semistable algebraic setting (which is a strict analytic situation over $\Zbf_0$) by
Bhatt in \cite{Bhatt-derived-de-Rham}, Remark 10.22, but the use of analytic methods
may allow for a treatment of the question based only on Faltings' almost mathematical
methods (like in Scholze's approach), avoiding the semistable reduction hypothesis.
This will be further discussed in Subsection \ref{motivation-global-periods}.
\end{example}

A notion of overconvergent sub-analytic subset will be necessary for the
study of direct images of analytically constructible sheaves in the pro-\'etale topology. 
We refer to Martin's paper \cite{Martin-overconvergent-subanalytic} for the notion
of overconvergent sub-analytic subsets in strict $p$-adic geometry. We give
here an adaptation of his definition to our general setting.
\begin{definition}
Let $R$ be a uniform ind-Banach algebra and $X_{top}$ be the underlying topological
space of a dagger (resp. a strict dagger) space $X$ over $R$.
A \emph{semi-analytic subset of $X$} is a subspace in the boolean algebra generated
by its rational (resp. strictly rational) subsets (by finite union, finite intersections
and complements).
A \emph{sub-analytic subset in $X$} is a subset given by the projection
of a semi-analytic subset in $X\times_R D^\dagger_n(0,\rho)$ (resp. in
$X\times_R D^\dagger_n(0,1)$) for some $n\geq 0$ along the natural projection map to $X$.
\end{definition}

It is not clear that subanalytic subsets defined as above form a boolean algebra,
as it is the case in the strict $p$-adic setting thanks to the result of Martin,
loc. cit., Proposition 1.39, but this may be an interesting question to ask to model theoretists.

\section{Dagger analytic geometry and Archimedean compactifications}
\label{Dagger-analytic-Arakelov}
Let $\Zbf:=(\Z,|\cdot|_\infty)$ be the global analytic basis and
$\Zbf_0:=(\Z,|\cdot|_0)$ be its non-archimedean counterpart.
It is quite clear that any scheme locally of finite type over $\Z$ may
be equipped with a structure of non-strict dagger space over $\Zbf$. Indeed,
closed affine subschemes of $\A^n_\Z$ may be described as closed analytic
subspaces of the overconvergent analytic affine space $\A^{n,\dagger}_\Zbf$,
and since algebraic maps are overconvergent, they can be used to
define dagger spaces over $\Zbf$ from schemes locally of finite type over $\Z$.
We thus get a (not very natural) ``base restriction functor''
$$\An^{\dagger,s}_{\Zbf_0} \cong \Sch_{\Z} \longrightarrow \An^\dagger_\Zbf$$
from the category of schemes over $\Z$
(i.e., strict dagger spaces over $\Zbf_0$) to the category of dagger spaces over $\Zbf$.
All this shows that the category of non-strict global analytic spaces is a natural
recipient both for algebraic geometry and analytic geometry over various bases
like $\R$, $\Q_p$ or $\Z_p$.
\begin{definition}
The above defined functor
$$\Anrm^\dagger:\Sch_\Z\longrightarrow \An^{\dagger}_\Zbf$$
will be called the non-strict (dagger) analytification functor.
\end{definition}

We now ask the natural question of describing which schemes over $\Z$ may
be seen as \emph{strict} dagger spaces over $\Zbf$, i.e., are isomorphic
to schemes in the image of the functor
$$\An^{\dagger,s}_\Zbf\longrightarrow \An^{\dagger}_\Zbf.$$
This will lead us to the idea that an extension of a scheme structure over $\Z$ to
a strict dagger analytic space structure over $\Zbf$ may be naturally thought of
as a kind of Archimedean compactification, in the sense that is usually meant
in Arakelov geometry.

\begin{definition}
Let $X$ be a scheme over $\Z$.
An \emph{Archimedean compactification} of $X$ is a strict dagger analytic space
$\Anrm^{\dagger,s}(X)$ over $\Zbf$ together with an isomorphism
$$\Anrm^{\dagger,s}(X)\overset{\sim}{\longrightarrow} \Anrm^\dagger(X)$$
of non-strict analytic spaces over $\Zbf$.
\end{definition}

\begin{example}
\label{example-disc-line}
The unit disc $D^\dagger(0,1)_\Zbf$ is a (very) strict dagger space with the affine line as
base extension to $\Zbf_0$, but it is not an Archimedean compactification of the algebraic
affine line $\A^1_\Z$ because there is no isomorphism
$$D^\dagger(0,1)_\Zbf\overset{\sim}{\longrightarrow} \A^{1,\dagger}_{\Zbf}.$$
Indeed, such an isomorphism would identify the associated topological spaces,
but one of them is compact and the other is non-compact.
\end{example}

\subsection{Archimedean compactifications of projective schemes}
\label{archimedean-compactifications-projective}
To illustrate the notion of Archimedean compactification, we will describe it for the projective
space $\Pb^1_\Z$.
Recall from Example \ref{examples-global} that the natural isomorphism
$$\Z[X_0,1/X_0]\longrightarrow \Z[X_1,1/X_1]$$
given by $X_0\mapsto 1/X_1$ is the underlying ring map of an isomorphism
$$\Zbf\{X_0,1/X_0\}^\dagger\longrightarrow \Zbf\{X_1,1/X_1\}^\dagger$$
of overconvergent rational domain algebras over $\Zbf$.
One may paste the overconvergent global analytic discs
$$
\Db^{1,\dagger}_{\Zbf}:=\Mb(\Zbf\{X_0\}^\dagger)\textrm{ and }
\Db^{1,\dagger}_\Zbf:=\Mb(\Zbf\{X_1\}^\dagger)
$$
(with functions given by the polynomial algebras over $\Z$ with their
sup norms on all discs containing the global unit disc)
to get a global analytic version $\Pb^{1,\dagger}_\Zbf$ of the projective line,
that will be the Archimedean compactification that we were looking for.
It is even a very strict Archimedean compactification, that one may think
of as a kind of analytic projective space $\Pb^{n,\dagger}_{\F_{\{\pm 1\}}}$.
Adding the polydisc seminorm structures on the polynomial rings
gives an important additional information that may be thought as
some kind of ``Arakelov compactification''. Remark that one may see
the (overconvergent analytic) complex projective space $\Pb^{1,\dagger}_\C$
either as the pasting of two copies of $\A^{1,\dagger}_\C$ along $\G^{\dagger}_{m,\C}$
(which gives also an algebraic model for it over $\C$), or as the pasting of two
(overconvergent) discs $\Db^{1,\dagger}_\C$ along $\Db^{1,\dagger}_\C-\{0\}$.
Using the disc viewpoint ``breaks the $\G_m$-symmetry'' of the algebraic situation.
It is quite stricking that such a ``breaking of the $\G_m$'' symmetry can
be also done in the global analytic setting.

\begin{lemma}
\label{Archimedean-closed-immersion}
If $X$ is a scheme over $\Z$ that admits an archimedean compactification
$$\Anrm^{\dagger,s}(X)\overset{\sim}{\longrightarrow} \Anrm^\dagger(X),$$
then any closed subscheme $Z$ of $X$ also has an archimedean compactification
and the inclusion $Z\to X$ may be extended to $\Anrm^{\dagger,s}(Z)\subset\Anrm^{\dagger,s}(X)$.
\end{lemma}
\begin{proof}
Let $Z\subset X$ be a closed subscheme, and let $\Anrm^{\dagger,s}(X)=\cup_i \Mb(A_i)$ be a
covering of the archimedean compactification of $X$ by strict representable domains over $\Zbf$.
Then one may write the equations of $Z$ in these charts and they are compatible by
construction with the pasting maps, so that $Z$ also has an archimedean compactification.
\end{proof}

Not every affine variety over $\Z$, written in explicit coordinates, can be easily
seen as an affine overconvergent analytic variety over $\Zbf$, because the solutions
of an equation in the affine line are not always contained
in the unit disc. However, the case of projective varieties is different, as we
will see from the following proposition.
\begin{proposition}
\label{projective-Arakelov}
If $X$ is a projective variety over $\Z$, then it has a natural Archimedean compactification
$$\Anrm^{\dagger,s}(X)\overset{\sim}{\longrightarrow} \Anrm^\dagger(X).$$
\end{proposition}
\begin{proof}
The case of the projective space $\Pb^n_\Z$ is obtained by generalizing directly
the above example: the natural morphism
$$
\begin{array}{ccc}
\Z[t_0,\dots,t_n,1/t_i]	& \to 	& \Z[t_0,\dots,t_n,1/t_j]\\
t_k					&\mapsto	& t_k/t_j\textrm{ si }k\neq i\\
t_i					&\mapsto 	& 1/t_j
\end{array}
$$
is bounded enough, so that it induces a morphism of overconvergent rational
domain algebras
$$
\Zbf\{t_0,\dots,t_n,1/t_i\}^\dagger\longrightarrow \Zbf\{t_0,\dots,t_n,1/t_j\}^\dagger
$$
over the base Banach ring $\Zbf=(\Z,|\cdot|_\infty)$ (these are just localizations of polynomial
algebras, but equipped with a family of norms induced by the over-seminorms of the sup norm
on the global unit polydisc). In this way, we get a very strict model for the projective space
$\Pb^n$. If $Z\subset \Pb^n_\Z$ is a closed sub-scheme, then
we may apply Lemma \ref{Archimedean-closed-immersion} to get an Archimedean
compactification of $Z$.
\end{proof}

\begin{remark}
It is clear from what we explained in Example \ref{example-disc-line} that general
schemes usually don't have Arithmetic compactifications. Indeed, the affine
line $\A^1_\Zbf$ is not representable in strict dagger spaces.
\end{remark}

\begin{remark}
Another approach to finding strict dagger models over $\Zbf$ of projective schemes
over $\Z$ may be given by the isomorphism
$$\Pb^1_\Z\cong \A^2_\Z-\{(0,0)\}/\G_{m,\Z}.$$
One may define a strict dagger model of $\Pb^1_\Z$ on $\Zbf$ by using the quotient analytic
space $\Db^{2,\dagger}(0,1)_\Zbf-\{(0,0)\}/U(1)_\Zbf$,
where $U(1)_\Zbf:=\Mb(\Zbf\{X,1/X\}^\dagger)$.
The analytic space
$$\Db^{2,\dagger}(0,1)_\Zbf-\{(0,0)\}$$
may be defined by pasting $\Db^\dagger(0,1)_\Zbf\times U(1)_\Zbf$ and
$U(1)_\Zbf\times \Db^\dagger(0,1)_\Zbf$ along
their common rational domain $U(1)_\Zbf\times U(1)_\Zbf$. This gives another way of
presenting the strict dagger projective space, as the quotient analytic sheaf
(i.e., set-valued sheaf with values in sets on the rational domain topology on strict dagger
algebras)
$$\Pb^{1,\dagger}_\Zbf:=\Db^{2,\dagger}(0,1)_\Zbf-\{(0,0)\}/U(1)_\Zbf.$$
\end{remark}

\begin{proposition}
\label{projective-Archimedean-compactification}
The following diagram of functors
$$
\xymatrix{
\ProjAn_\Zbf^{\dagger,s}\ar[d]\ar[r]^(0.38)\sim	& \ProjAn^{\dagger,s}_{\Zbf_0}\cong \Proj_\Z\ar[d]\\
\An^{\dagger,s}_\Zbf\ar[r]				& \An^\dagger_\Zbf}
$$
is (2)-commutative, with vertical arrows fully faithful and the upper horizontal
arrow an equivalence.
\end{proposition}
\begin{proof}
The vertical arrows are fully faithful by definition.
The fact that the upper horizontal arrow is essentially surjective follows from proposition
\ref{projective-Arakelov}. The fact that it is faithful is clear. The fact that it is full
is less clear. If $f:X\to Y$ is a morphism of projective varieties, then its graph
$\Gamma_f$, defined as the pullback
$$
\xymatrix{
\Gamma_f\ar[r]\ar[d]			& X\times Y\ar[d]^{f\times \id}\\
Y\ar[r]^{\Delta_Y}			& Y\times Y}
$$
is projective. By proposition \ref{projective-Arakelov}, this graph
has an Arithmetic compactification. Since $f:X\to Y$ may be written as the
pullback of the projection $\Gamma_f\to Y$ along the identity map, it
also has an Archimedean compactification, so that the upper horizontal arrow of the
diagram in the statement of the proposition is an equivalence.
\end{proof}

\subsection{Logarithmic Archimedean compactifications of quasi-projective schemes}
\label{logarithmic-Arakelov}
We would like to have a way to associate to a quasi-projective variety
over $\Z$ some kind of strict global dagger space over $\Zbf$
with generic fiber the given variety, that will give an Arakelov model
of the given variety. For example, if we start from $X=\A^1_\Z$, this wish
can't be fulfilled stricto sensu.
A way to overcome this problem with non-projective schemes, at least at the
cohomological level, was paved by Deligne in \cite{De12}, and then formalized
geometrically by Fontaine-Illusie and Hyodo-Kato, by the use of
logarithmic analytic spaces. In the above example, one replaces the affine
line over $\Zbf$ by the logarithmic analytic space over $\Zbf$ given by
the projective line over $\Zbf$, together with the sheaf of
monoids $\Mc:=j_*\Oc_{\A^1}^*\cap \Oc_{\Pb^1}\subset \Oc_{\Pb^1}$,
where $j:\A^1\to \Pb^1$ is the natural embedding. The definition
of this monoid of course uses some non-strict dagger geometry, since even
$\A^1$ and the structural sheaf can't be defined in the strict setting in general,
but the analytic space in play (here $\Pb^1$) is really a strict analytic space.
So strict logarithmic geometry gives an intermediary setting between strict analytic
geometry and non-strict analytic geometry.

\begin{definition}
Let $t\in \{an,\{an,s\},\dagger,\{\dagger,s\}\}$ be a type of analytic spaces.
Let $X$ be a $t$-analytic space over an ind-Banach ring $R$. A pre-logarithmic structure
on $X$ is given by a sheaf of monoids $\Mc$ on $X_{et}$, together with a morphism
of multiplicative monoids
$$\alpha:\Mc\to \Oc_X.$$
The pre-log structure is called a log structure if $\alpha$ induces an isomorphism
$$\alpha^{-1}(\Oc_X^*)\overset{\sim}{\longrightarrow} \Oc_X^*.$$
\end{definition}

It is easy to generalize the notion of Archimedean compactification to the
logarithmic setting.
\begin{definition}
Let $(X,\Mc)$ be a log scheme over $\Z$, and $\An^\dagger(X,\Mc)$ be the associated
non-strict logarithmic dagger space over $\Zbf$. An \emph{Archimedean compactification}
of $(X,\Mc)$ is a strict logarithmic dagger space $\Anrm^{\dagger,s}(X,\Mc)$ over $\Zbf$ together
with an isomorphism
$$\Anrm^{\dagger,s}(X,\Mc)\overset{\sim}{\longrightarrow} \Anrm^\dagger(X,\Mc)$$
of non-strict dagger logarithmic spaces over $\Zbf$.
\end{definition}

We now may now extend Proposition \ref{projective-Arakelov} to the case of semistably
compactifiable schemes.
\begin{proposition}
Let $X$ be a smooth scheme over $\Z$ that admits a projective compactification $\bar{X}$
over $\Z$ such that $D:=\bar{X}\backslash D$ is a divisor with normal crossings.
Then the associated log-scheme $(\bar{X},\Mc_D)$, where $\Mc_D:=j_*\Oc_X^*\cap \Oc_{\bar{X}}$,
and where $j:X\to \bar{X}$ is the natural embedding, has a canonical
Archimedean compactification $\An^{\dagger,s}(\bar{X},\Mc)$.
\end{proposition}
\begin{proof}
The closed inclusion $D\to \bar{X}$ is a morphism of projective schemes over $\Z$ that
has an Archimedean compactification $\Anrm^{\dagger,s}(D)\to \Anrm^{\dagger,s}(\bar{X})$
by Proposition \ref{projective-Archimedean-compactification}. The associated logarithmic
strict dagger space $\Anrm^{\dagger,s}(X,\Mc_D)$ will do the job.
\end{proof}

\begin{remark}
\label{h-local-archimedean-compactifications}
We can still say something in the non-semistable case,
using de Jong's resolution of singularities, as explained
by Beilinson in \cite{Beilinson-derived-de-Rham}.
Let $X$ be a smooth quasi-projective scheme over $\bar{\Q}$.
De Jong's theorem implies that a basis for the $h$-topology on $X$ is
given by arithmetic semistable pairs $(U,\bar{U})/\bar{\Z}$
(a smooth compactification of a smooth variety with boundary a normal crossing divisor).
To each such pair, one may associate a logarithmic scheme
over $\bar{\Z}$ that has an Archimedean compactification
over $\bar{\Zbf}^\dagger$. So we may say that in some sense,
every scheme over $\bar{\Q}$ may be $h$-locally logarithmically Arithmetically
compactified. This construction may give a natural setting to explain geometrically
Arakelov-motivic cohomology \cite{Arakelov-motivic-cohomology-I}
in a way that avoids the direct use of Deligne cohomology.
This point will be further discussed in Remark \ref{Arakelov-motivic}.
This may also give a natural setting to define global period rings
by derived periods \`a la Beilinson-Bhatt. This point will be further discussed in
Subsection \ref{motivation-global-periods}.
\end{remark}

\subsection{A dagger arithmetic Riemann-Roch problem}
\label{dagger-arithmetic-Riemann-Roch}
It is quite tempting to generalize Riou's homotopy theoretic approach
to the Riemann-Roch theorem from \cite{Riou-K-theory} by looking at
it as written in the setting of homotopy theory of strict analytic spaces
over the base $\Zbf_0=(\Z,|\cdot|_0)$
(in a sense to be explained in Section \ref{global-analytic-motives}),
and trying to extend it to strict analytic spaces over $\Zbf=(\Z,|\cdot|_\infty)$. We will call the
question of this extension the \emph{dagger arithmetic Riemann-Roch problem}.

One may define the dagger general linear group as the sheaf on dagger algebras over
$\Zbf$ given by
$$
\GL_n:A\mapsto \GL_{n}(\Alg(A)).
$$

Since $\A^1:A\mapsto \Alg(A)$ is not representable in the category of strict dagger spaces over
$\Zbf_0$, it is quite reasonable to imagine that the same applies to the general linear group
for $n>1$.
It is quite easy, however, to define the strict global dagger analog of the classifying
space $\BGLbf$ used by Riou: in $\A^1$-homotopy theory, this space is described as the infinite
Grassmannian $\Gr_\infty$ given by the colimit of the systems $(\Gr_{d,n})_{(d,n)\in \N^2}$
where the transition morphisms are of the form $\Gr_{d,r}\to \Gr_{1+d,r}$ and
$\Gr_{d,r}\to \Gr_{d,r+1}$. We thus only have to show that these varieties
and maps have a dagger Archimedean compactification, i.e., a strict model over $\Zbf$,
which is true since they are projective,
so that we can apply Proposition \ref{projective-Arakelov}.

Another approach to this problem, that is followed by Karoubi and
Villamayor in \cite{Karoubi-Villamayor}, and more recently by Tamme
\cite{Tamme}, is to replace the group $\GL_n$ by the simplicial
group $\GL_n^\bullet$ given by
$$
\GL_n^\bullet:A\mapsto \GL_n(\Alg(A\{\Delta^\bullet\}^\dagger)),
$$
where the simplicial ring $A\{\Delta^\bullet\}^\dagger$ is defined by
$$
A\{\Delta^n\}^\dagger:=A\{T_0,\dots,T_n\}^\dagger/(\sum T_i-1).
$$
The classifying space may then be defined as the total $\infty$-stack
associated to the functor
$$\BGLbf:A\mapsto \Z\times B_\bullet\GL(A\{\Delta^\bullet\}^\dagger)$$
with values in bisimplicial sets.
One then defines the (overconvergent) Karoubi-Villamayor $K$-theory of $A$ as
$$\KV_i(A)=\pi_i(\BGLbf).$$
Following \cite{Tamme}, 2.4, this gives back algebraic $K$-theory for $i\geq 1$
in the case of a trivially normed integral ring $R$ that is supposed to be regular.
Remark that both approaches may be related by working in the setting of
global analytic homotopy theory described in Section \ref{global-analytic-motives}.

Once given the correct Archimedean dagger compactification of $\BGLbf$, and using
the notion of rational motivic cohomology proposed in Section \ref{global-analytic-motives},
one may ask the following question: if $f:X\to S$ is a projective smooth morphism
of strict dagger analytic spaces over $\Zbf$, does the following diagram
$$
\xymatrix{
\R f_*\BGLbf_\Q\ar[rrr]^{\R f_*(\ch.\Td(T_f))}\ar[d]^{f_*}	&&&
\prod_{i\in \Z}\R f_* \Hbf_{\Q(i)}[2i]\ar[d]^{f_*}\\
\BGLbf_\Q\ar[rrr]^{\ch}							&&&
\prod_{i\in \Z}\Hbf_{\Q(i)}[2i]}
$$
commute in the strict rational stable homotopy theory $\SH^{\dagger,s}(S)$?
The same question may apply in the quasi-projective case by replacing
$f_*$ by the proper direct image $f_!$ and motivic cohomology $\Hbf$ by
its version with proper support $\Hbf_c$.

\begin{remark}
As a corollary of this homotopy theoretic Riemann-Roch theorem, one would
get a Riemann-Roch theorem relating the direct image of higher Artin-Verdier
$K$-theory classes to the direct image of their Chern classes in higher
Artin-Verdier motivic cohomology (to be defined as motivic cohomology
of strict analytic spaces over $\Zbf$ in the sense of
Section \ref{global-analytic-motives}).
This would give a kind of higher arithmetic Riemann-Roch for Artin-Verdier
motivic \'etale cohomology, that is quite different in nature from
the Riemann-Roch statements proved on Arakelov-motivic cohomology
in \cite{Arakelov-motivic-cohomology-II}, since the Hodge filtration is not included in our approach.
\end{remark}

\begin{remark}
To get a global analytic interpretation of Arakelov-motivic cohomology,
one really needs to combine the global analytic information given by strict
global motivic cohomology with the differential information given by
Hodge-filtered de Rham cohomology. This problem may be approached
by trying to globalize the period isomorphism of $p$-adic Hodge theory
(see Subsection \ref{motivation-global-periods}).
\end{remark}

\begin{remark}
A global analytic interpretation of Arakelov-motivic cohomology may also be
attained by taking inspiration in the work of Karoubi on multiplicative
$K$-theory \cite{Karoubi-K-theorie-multiplicative}. As explained to the author
by Gregory Ginot, the use of the cyclic Chern character has the great advantage
on the usual geometric
approach to avoid the introduction of denominators of the form $\frac{1}{n!}$ in the
definition of the Chern character map
$$\ch:K_*(X)\longrightarrow HC_{-}(X).$$
In any case (i.e., even in the Karoubi approach), to get the right $p$-torsion information,
one needs to work with a version of (maybe Hodge-completed) derived Hodge-filtered de
Rham cohomology (or, in the cyclic setting, simply cyclic homology) relative to the global
analytic base $\Zbf$, to be defined in Subsection \ref{derived-de-Rham}.
A global analytic version of the Chern character will be discussed in Section \ref{Chern-character}.
\end{remark}

\section{Global analytic motives}
\label{global-analytic-motives}
The aim of this section is to give a formalism for analytic motives \`a la
Morel-Voevodsky \cite{Morel-Voevodsky}, that
gives back usual motives in the strict case over a trivially valued integral ring,
and that also gives back good categories of rigid and complex analytic motives,
similar to those defined by Ayoub in \cite{Ayoub-Betti} and \cite{Ayoub-analytic-motives}.
We will define \'etale, Nisnevich and pro-\'etale motives, with a preference to
\'etale motives, since they seem to have better properties
than Nisnevich motives with respect to the integral Hodge and Tate conjectures
(see \cite{Rosenschon-Srinivas}), and they also allow a direct definition of the
\'etale realization functor with finite coefficients (remark however that
over a characteristic $p$-basis, they are with coefficients in $\Z[1/p]$,
so that they don't give a good information on the $p$-part of the cohomology).
We will also use the pro-\'etale topology
that gives a better take at the completed \'etale realization functor.

We refer to Ayoub's thesis \cite{Ayoub-six-operations-I} and \cite{Ayoub-six-operations-II}
for a systematic treatment of the homotopy theory of schemes and to
Cisinski and Deglise for a refined treatment of the theory of motives
with rational coefficients \cite{Cisinski-Deglise-mixed-motives}.
We will use the language of $\infty$-categories (for which we refer to
Lurie's books \cite{Lurie-higher-topos-theory} and \cite{Lurie-higher-algebra}) to get
a shorter presentation, but the language of model
categories and symmetric spectra in presheaves, developed by Ayoub
in \cite{Ayoub-six-operations-II} has the advantage of allowing more
explicit computations. We will give a presentation of our theory
that is a neat combination of the viewpoint used by Roballo in his
thesis \cite{Robalo} and of Ayoub in his works on motives and analytic motives.

\subsection{Stable homotopy theory of sheaves}
The analog in global analytic geometry of the affine line used in algebraic homotopy theory
(and of the unit interval used in classical homotopy theory)
will be the unit disc. It indeed gives back the algebraic affine line in
the strict situation over a trivially valued integral ring.
The stable homotopy theory of analytic spaces will be
constructed by using $\infty$-sheaves (aka $\infty$-stacks) on the \'etale
(resp. Nisnevich, resp. pro-\'etale)
site of analytic spaces with values in a stable presentable symmetric monoidal $\infty$-category
$(\Mfk,\otimes)$, that will be the stable $\infty$-category $(\Sp,\wedge)$ of spectra,
the stable $\infty$-category $(\Sp(\Mod_s(\Ac)),\otimes)$ of simplicial module spectra
on a sheaf $\Ac$ of rings (that will often be $\Z$, $\Z/n\Z$ of $\Q$),
or (in the caracteristic zero situation) the stable $\infty$-category $(\Mod_{dg}(\Ac),\otimes)$ of
differential graded modules over $\Ac$. The main difference between the simplicial and
the differential graded setting is that commutative differential graded algebras give correct
strictifications of homotopy commutative algebras only over $\Q$.
If we work with modules,
we will get categories of motives, and if we work with spectra, we will
get stable homotopy categories.

We refer to Robalo \cite{Robalo} (see also \cite{Robalo-these})
for a short introduction to the $\infty$-categorical
tools used in this subsection, and to Lurie's book \cite{Lurie-higher-algebra} for a
complete reference on homotopical algebraic tools.
We start by recalling from Robalo's \cite{Robalo} important facts about the stabilization
of symmetric monoidal $\infty$-categories.
\begin{theorem}
Let $(\Cc,\otimes)$ be a presentable symmetric monoidal $\infty$-category and
$T\in (\Cc,\otimes)$ be an object. There exists a natural monoidal
functor $(\Cc,\otimes)\to (\Cc[T^{\otimes -1}],\otimes)$
from $\Cc$ to a presentable symmetric monoidal $\infty$-category such
that for every symmetric monoidal category $\Dc$, the natural morphism
$$
\Map((\Cc[T^{\otimes -1}],\otimes),(\Dc,\otimes))\longrightarrow
\Map_{T^{\otimes -1}}((\Cc,\otimes),(\Dc,\otimes)),
$$
from symmetric monoidal functors to symmetric monoidal functors that make
$T$ invertible, is an equivalence.
If the object $T$ is further symmetric, meaning that there is a natural
$2$-equivalence in $\Cc$ between the cyclic permutation $\sigma_{(123)}$
on $T\otimes T\otimes T$ and the identity map, given by a $2$-morphism:
$$
$$
then the underlying $\infty$-category of $\Cc[T^{\otimes -1}]$ is identified
with the stabilization
$$
\Stab_{T}(\Cc):=
\colim(\cdots\overset{T\otimes -}{\longrightarrow}\Cc\overset{T\otimes -}{\longrightarrow}
\Cc\overset{T\otimes -}{\longrightarrow}\cdots).
$$
\end{theorem}

We denote $(\Sc,\times)$ the symmetric monoidal $\infty$-category of spaces,
obtained as the $\infty$-localization of the monoidal category $(\SSets,\times)$ of
simplicial sets by weak equivalences. The symmetric monoidal $\infty$-category of pointed spaces
with the wedge product is denoted $(\Sc_*,\wedge)$. The symmetric monoidal $\infty$-category
of spectra is obtained by
$$(\Sp,\wedge):=((\Sc_*,\wedge)[(S^1)^{\otimes -1}],\otimes).$$ 
Since $S^1$ is symmetric in $\Sc_*$, the underlying $\infty$-category of $\Sp$ may
be described as the stabilization of $\Sc_*$ with respect to the wedge product by $S^1$.

\begin{definition}
An object $X$ of a stable $\infty$-category $\Mfk$ is called \emph{homotopically compact}
if for all $n$, the functor $\Hom_{h(\Mfk)}(X,-[n])$ commutes to small filtered colimits.
\end{definition}

We essentially give here an $\infty$-categorical analog of Ayoub's notion of coefficient
category from \cite{Ayoub-six-operations-II}, Definition 4.4.23.
\begin{definition}
Let $(\Mfk,\otimes)$ be a symmetric monoidal $\infty$-category.
We say that $(\Mfk,\otimes)$ is a \emph{category of coefficients} if
\begin{enumerate}
\item $\Mfk$ is stable and presentable,
\item there exists a set $\Ec$ of homotopically compact objects of $\Mfk$ that
generate the triangulated category with infinite sums $h(\Mfk)$.
\end{enumerate}
\end{definition}
By definition, a symmetric monoidal model category $(\Mfk,\otimes)$ that is
a category of coefficients in the sense of Ayoub loc. cit. will give a symmetric
monoidal $\infty$-category
$(\Mfk,\otimes)$ that is a coefficient category in the above sense. The model
category setting gives a better take at explicit computations, but we chose the
$\infty$-category setting because it sometimes allows easier universal constructions.

\begin{proposition}
Let $(X,\tau)$ be a small $\infty$-site and $(\Mfk,\otimes)$ be an $\infty$-category
of coefficients. Then the categories
$$\PreShv(X,\tau,\Mfk)\textrm{ and }\Shv(X,\tau,\Mfk)$$
of presheaves and sheaves on $(X,\tau)$ with values in $\Mfk$
are stable presentable symmetric monoidal $\infty$-categories. 
\end{proposition}
\begin{proof}
See \cite{Robalo} for a closely related result. This follows from the fact
(explained to us by Brad Drew) that one may write
$$\PreShv(X,\tau,\Mfk)=\PreShv(X,\tau,\Sp)\otimes_\Sp\Mfk$$
and similarly for sheaves. The fact that $\PreShv(X,\tau,\Sp)$ is
stable presentable and symmetric monoidal is already known, because
it may be obtained by stabilizing presheaves with values in $\SSets$,
that are presentable.
\end{proof}

We refer to Robalo \cite{Robalo}, Section 5 for the following.
\begin{definition}
Let $(X,\tau)$ be a small $\infty$-site, $(\Mfk,\otimes)$ be a coefficient $\infty$-category,
and $I=\{I_s\}_{s\in S}\in \Shv(X,\tau,\Mfk)$ be a family of objects.
\begin{enumerate}
\item The associated \emph{unstable homotopy category}
is the $\infty$-localization
$$\H(X,\tau,I,\Mfk)=L_I(\Shv(X,\tau,\Mfk))$$
of the $\infty$-category of sheaves with respect to the class
of morphisms $X\times I_s\to X$.
\item The \emph{pointed unstable homotopy category}
is the associated pointed  symmetric monoidal $\infty$-category $\H(X,\tau,I,\Mfk)_*$.
\item If $T$ is a symmetric object in $\H(X,\tau,I,\Mfk)_*$, we define the associated
\emph{stable homotopy category} as the universal presentable symmetric monoidal
$\infty$-category in which $T$ becomes $\otimes$-invertible:
$$\SH(X,\tau,I,T,\Mfk):=\H(X,\tau,I,\Mfk)_*[T^{\otimes -1}].$$
\end{enumerate}
\end{definition}

The underlying $\infty$-category of $\SH(X,\tau,I,\Mfk)$ is equivalent to the $T$-stabilization
of $\H(X,I,\Mfk)_*$, which is given by the $\infty$-categorical colimit
of the sequence
$$
\cdots\overset{T\otimes -}{\longrightarrow}\H(X,I,\Mfk)_*
\overset{T\otimes -}{\longrightarrow}\H(X,I,\Mfk)_*
\overset{T\otimes -}{\longrightarrow}\cdots
$$

\subsection{Analytic motives and spectra}
Let $t\in \{an,\{an,vs\},\{an,s\},\dagger,\{\dagger,vs\},\{\dagger,s\}\}$ be a type of analytic spaces.
We now apply the constructions of the previous section to
the category of smooth $t$-analytic spaces with its
\'etale, Nisnevich and pro-\'etale topologies. We follow quite closely the
approach of Ayoub in the complex \cite{Ayoub-Betti} and $p$-adic
analytic \cite{Ayoub-analytic-motives} situations. We introduce two
types of non-strict motives: those obtained by contracting the unit disc $\Db^1$,
and those obtained by contracting the family of all discs $\{\Db(0,\rho)\}_{\rho \in \R_{>0}}$.

Let $R$ be an ind-Banach ring and $X$ be a $t$-analytic space over $R$.
The category $\AnSm^t_X$ of smooth $t$-spaces over $X$ is small.
It will be equipped with a topology $\tau$ that
is either the \'etale topology $\tau_{et}$, the Nisnevich topology $\tau_{Nis}$ or
the pro-\'etale topology $\tau_{proet}$.
We fix an $\infty$-category $(\Mfk,\otimes)$ of coefficients.
We will denote $T$ the object in $\Shv(\AnSm^t_X,\tau,\Mfk)$ given by
$$T=\cof(\G_{m,X}\otimes \1\to \A^1_{X}\otimes \1).$$
The proof of Ayoub that $T$ is symmetric in the algebraic
setting in \cite{Ayoub-six-operations-II}, Lemme 4.5.65, being based on
elementary integer valued matrix computations, extends directly to the strict
and non-strict overconvergent setting.

\begin{definition}
The $\tau$-stable homotopy category $\SH_{\Mfk}^t(X,\tau)$ with coefficients in $\Mfk$ is
defined by
$$\SH_{\Mfk}^t(X,\tau):=\SH(\AnSm^t_X,\tau,D^t_X(0,1)\otimes \1,T,\Mfk).$$
If $t\in \{an,\dagger\}$ is a non-strict type of analytic spaces,
we also define the $\tau$-stable non-strict homotopy
category $\SH_{\Mfk}^{t,ns}(X,\tau)$ with coefficients in $\Mfk$ by
$$\SH_{\Mfk}^{t,ns}(X,\tau):=\SH(\AnSm^t_X,\tau,\{D^t(0,\rho)\otimes \1\}_{\rho\in \R{>0}},T,\Mfk).$$
If $(\Mfk,\otimes)=(\Sp,\wedge)$ is the symmetric monoidal $\infty$-category of spectra,
we will denote
$\SH_\Mfk^t(X,\tau)$ (resp. $\SH_\Mfk^{t,ns}(X,\tau)$) simply by $\SH^t(X,\tau)$
(resp. $\SH^{t,ns}(X,\tau)$).
If $\Lambda$ is a commutative ring and $(\Mfk,\otimes)=(\Mod_{dg}(\Lambda),\otimes)$,
the $\infty$-category
$$
\begin{array}{c}
\DA_\tau^t(X,\Lambda)=\SH_\Mfk^t(X,\tau)\\
\textrm{(resp. }\DA_\tau^{t,ns}(X,\Lambda)=\SH_\Mfk^{t,ns}(X,\tau)\textrm{)}
\end{array}
$$
is called the category of $\tau$-motivic (resp. non-strict $\tau$-motivic) sheaves
with coefficients in $\Lambda$.
More generally, if $\Lambda$ is a sheaf of rings for the given topology $\tau$,
we will still denote
$$
\DA_\tau^t(X,\Lambda):=
\DA_\tau^t(X,\Z)\otimes_{\Shv(X,\tau,\Mod_{dg}(\Z_X))}\Shv(X,\tau,\Mod_{dg}(\Lambda))
$$
the associated category of $\tau$-motivic sheaves with coefficients in $\Lambda$
(and similarly for $\DA_\tau^{t,ns}(X,\Lambda)$).
\end{definition}

The notation of the above definition are consistent, because if $\Lambda_X$ is a
constant sheaf of rings with values $\Lambda$, we will have a canonical equivalence
$$
\SH_{\Mod_{dg}(\Lambda)}^t(X,\tau)\cong
\SH_{\Mod_{dg}(\Z)}^t(X,\tau)\otimes_{\Shv(X,\tau,\Mod_{dg}(\Z_X))}
\Shv(X,\tau,\Mod_{dg}(\Lambda_X)).
$$

It is natural, following what we said in Remark \ref{strict-fully-faithful}, to ask if the natural functor
$$\SH_\Mfk^{\dagger,s}(X,\tau)\to \SH_\Mfk^\dagger(X,\tau)$$
is fully faithful.
This question seems to have a positive answer over $\Cbf$ (where it is
even an equivalence), and may
also have a positive answer on a non-archimedean field $K$, if one
can adapt the work of Temkin \cite{Temkin-local-properties-II}.
If we work over $(\Z,|\cdot|_\infty)$, this adaptation does not seem
to be so easy, but the question remains interesting: this would
relate usual algebraic motives to global analytic motives, which
are still quite rigid objects.

\begin{remark}
Let $R$ be a Banach ring, $\Mfk$ be a category of coefficients
and $t\in \{an,\dagger\}$ be a non-strict type of analytic spaces
and $X$ be a $t$-analytic space.
The multiplication map
$$D(0,1)\times D(0,\rho)\to D(0,\rho)$$
shows that $D(0,\rho)$ is $D(0,1)$-contractible, so that there is a natural functor
$$\SH^{t,ns}_\Mfk(X,\tau)\to \SH^{t}_\Mfk(X,\tau).$$
There is also of course a natural localization functor
$$\SH^{t}_\Mfk(X,\tau)\to \SH^{t,ns}_\Mfk(X,\tau).$$
\end{remark}

\begin{remark}
It is quite natural to try to extend Ayoub's formulation of the six operation
formalism from \cite{Ayoub-six-operations-II}
(partially extended to the $\infty$-categorical setting by Robalo \cite{Robalo-these})
to our more general setting. Ayoub's papers \cite{Ayoub-analytic-motives} and
\cite{Ayoub-Betti} show us that there is no essential obstructions to this possibility.
We will use this idea in some of our discussions.
\end{remark}

\begin{example}
Let $X$ be a scheme, seen as a strict analytic space over $\Zbf_0:=(\Z,|\cdot|_0)$.
Then the stable homotopy categories $\SH_\Mfk(X,\tau)$ for $\tau=\tau_{et}$ and
$\tau=\tau_{Nis}$ give back
the usual \'etale and Nisnevich stable homotopy categories. This will be useful to
get various strict non-archimedean analytifications over $\Qbf_p$ and $\Zbf_p$
for (say) projective schemes by a mere base change.
One must not forget however, that the base extension of
$\A^1_\Z$, seen as the strict analytic space over $\Zbf_0$ given by the unit disc,
only give the unit disc on $\Q_p$ and $\Z_p$, and not the affine line.
\end{example}

\begin{example}
If the base Banach ring $\Qbf_p=(\Q_p,|\cdot|_p)$ is seen as a strict analytic
algebra over itself, the category $\SH_\Mfk^s(\Qbf_p,\tau_{Nis})$ gives
back Ayoub's category $\RigSH_\Mfk(\Q_p)$ of rigid analytic motives over $\Q_p$ with
coefficients in $\Mfk$. We will use the overconvergent analog, because it carries
a natural de Rham realization functor (because of homotopy invariance of de Rham
cohomology, that is only valid in the overconvergent setting).
We may also work with smooth perfectoid spaces over the completion of $\Q_p(p^{1/p^{\infty}})$,
seen as analytic spaces over this field. Using Nisnevich coverings, we find a
perfectoid version of Ayoub's rigid analytic motives (see
Vezzani's article \cite{Vezzani-perfectoid} for
a description of the tilt operation in the setting of rigid analytic motives).
If we work over the Banach ring
$\Zbf_p=(\Z_p,|\cdot|_p)$ and $X$ is a strict dagger analytic space over
$\Zbf_p$, we find a homotopy category
$\SH_\Mfk^{\dagger,s}(X,\tau_{Nis})$ that is an overconvergent
(sometimes called ``weakly convergent'' in the litterature) analog
of Ayoub's homotopy category $\RigSH_\Mfk(Y)$ over a $\Zbf_p$-rigid scheme $Y$
(see \cite{Ayoub-analytic-motives}). It may (or may not) be possible to represent syntomic
cohomology in this new category following closely the approach of Deglise and
Mazzari \cite{Deglise-Mazzari}.
\end{example}

\begin{example}
\label{complex-analytic-dagger}
Let $X$ be a complex analytic space. This is also a global analytic space
over $\Cbf$ in the sense of Berkovich (see \cite{Poineau2}), which has,
by Remark \ref{global-analytic-dagger}, a naturally associated
dagger analytic space $X^\dagger$ over $\Cbf$.
There is a natural functor
$$\SH_\Mfk^{Ayoub}(X,\tau_{usu})\to \SH_{\Mfk}^\dagger(X^\dagger,\tau_{et})$$
from the complex analytic stable homotopy category over $X$ with the usual topology
and with coefficients in $\Mfk$ (in the sense of Ayoub \cite{Ayoub-Betti}) to the stable
homotopy category over $X^\dagger$ with coefficients in $\Mfk$.
It is likely a fully faithful functor. It may even be an equivalence.
In any case, the same methods as those of Ayoub in loc. cit. should allow
to prove that a convenient version of $\SH_\Mfk^\dagger(X^\dagger,\tau_{et})$
is equivalent to the $\infty$-category $\Shv(X,\tau_{usu},\Mfk)$ of sheaves
on $X$ for the usual topology with values in $\Mfk$. In particular, if $X$ is
a point, we should get back the usual homotopy theory given by $\Mfk$.
If the base for $X$ is $\Rbf:=(\R,|\cdot|_\infty)$, then we should get an identification
of $\SH_\Mfk^\dagger(X^\dagger,\tau_{et})$ with the $\infty$-category
$\Shv([X(\C)/\sigma],\tau_{usu},\Mfk)$ where $\sigma$ is complex conjugation.
Indeed, the analytic etale $\infty$-topos of $X$ is identified with the quotient topos
$[X(\C)/\sigma]$.
\end{example}

\begin{example}
Suppose that a given scheme over $\Z$ may be seen as a the extension
of a strict dagger analytic space over $\Zbf=(\Z,|\cdot|_\infty)$. The associated stable
homotopy categories $\SH_\Mfk^{\dagger,s}(X,\tau_{Nis})$ give a category of (strict) motives
over $\Zbf$ that has a natural analytification by base change to $\Cbf=(\C,|\cdot|_\C)$
that is very close to usual homotopy theory of $\Mfk$-valued sheaves (by
Example~\ref{complex-analytic-dagger}),
and also natural non-archimedean analytifications over $\Qbf_p$ and $\Zbf_p$
that are close to Ayoub's rigid analytic motives. It is likely that the Artin-Verdier
\'etale cohomology theory can be represented as the \'etale cohomology spectrum
in the stable homotopy category $\SH^{s,\dagger}(\Zbf,\tau_{et})$ or
$\SH^\dagger(\Zbf,\tau_{et})$.
Remark that the base extension along $\Zbf\to \Zbf_0:=(\Z,|\cdot|_0)$ gives a
pullback functor
$$\DA_{et}^{\dagger,s}(\Zbf,\Z)\to \DA_{et}^{\dagger,s}(\Zbf_0,\Z)=\DA_{et}^{alg}(\Z,\Z)$$
that sends a strict global analytic motive to the associated algebraic motive.
The base extension along $\Zbf\to \Cbf=(\C,|\cdot|_\infty)$ gives a functor
$$\DA_{et}^{\dagger,s}(\Zbf,\Z)\to \DA_{et}^{\dagger,s}(\Cbf,\Z)$$
that is a strict version of the Betti realization functor. It seems clear
from the previous examples that the complex base extension looses
a lot of information, since motivic cohomology in this setting is
essentially Betti cohomology. However, this is less clear in the global
analytic setting, because the global motive contains information about
all places, and even the underlying algebraic motive, obtained by
base extension to $(\Z,|\cdot|_0)$. This global analytic strict \'etale motivic cohomology
may thus be an interesting new invariant, that we will call the \emph{Artin-Verdier
motivic cohomology}. We denote $\Hbf_{et,\Z}^{av}$ the associated spectrum,
constructed by using Ayoub's theory of six operations in our setting.
\end{example}

\begin{remark}
\label{Arakelov-motivic}
It is quite tempting to define (an \'etale Artin-Verdier version of) Arakelov motivic
cohomology (defined by Holmstrom and Scholbach in \cite{Arakelov-motivic-cohomology-I})
with coefficients in $\R$ by using the various Grothendieck operations that may
be available on the motivic categories. By definition, real Beilinson-Deligne cohomology
is representable by a spectrum $\Hbf_{BD}$ in
$$\SH_{et}^{\dagger,s}(\Zbf_0,\Z)=\SH_{et}^{alg}(\Z,\Z),$$
where $\Zbf_0=(\Z,|\cdot|_0)$. The same is true for the \'etale motivic cohomology spectrum
$\Hbf_{mot,et,\Z}$. Remark that one can't hope to represent Beilinson-Deligne cohomology in
the non-strict analytic category because its construction is based on smooth
compactifications with boundary given by a normal crossing divisor, which are not
available in general in the analytic setting. The Beilinson-Deligne component $\Hbf_{BD}$ should
however be more naturally explained by a nice spectrum in the category of strict global
analytic motives $\SH_{et}^{\dagger,s}(\Zbf,\Z)$ with $\Zbf=(\Z,|\cdot|_\infty)$.
Indeed, it is related to the Betti
realization which is naturally available only over the global base $\Zbf$.
We may work with the global Artin-Verdier analog $\Hbf_\Z^{av}$ over $\Zbf$ of $\Hbf_{mot,et,\Z}$.
There is a natural morphism
$$
\xymatrix{
\Hbf_{mot,\R}\ar[rr]^(0.4){\id\wedge 1_{\Hbf_{BD}}} && \Hbf_{mot,\R}\wedge \Hbf_{BD}}
$$
and Holmstrom and Scholbach define the Arakelov-motivic cohomology spectrum $\hat{\Hbf}$
as the homotopy fiber of this morphism (recall the motivic \'etale cohomology with coefficients
in $\Q$ identifies with usual motivic cohomology).
One may try to extend, The interpretation of special values of $L$-functions
as proposed by Scholbach in his thesis \cite{Scholbach-special-L-values}
(with probably some additional truncational cares) to the study
of special values up to a factor in $\Z^\times$:
the determinant of the pairing
$$H_{*}^{mot,et}(X,\Z)\times \hat{H}^*(X,\R)\to \R$$
between real motivic Arakelov cohomology and integral motivic homology with
values in $\R$ (in the case $X/\Z$ smooth projective)
may give the special value up to a factor in $\{\pm 1\}$ (an argument against
this idea, explained to the author by Baptiste Morin, is that the use of \'etale motives
destroys the $p$-torsion information in characteristic $p$).
This comes from the fact, explained to the author by Jakob Scholbach, that there are
natural isomorphisms:
$$
\begin{array}{ccc}
\det (\hat{H}^*(X,\R))	& =	& 	\det(H^*(X, \R))\otimes_\R \det(H^*_{BD}(X,\R))\\
					& =	&  	\det(H^*_{mot,et}(X,\Z))\otimes_\Z\R\\
					&	&	\otimes_\R\\
					&	&	\det(H^*_{B}(X_\R,\Z))\otimes_\Z\R\\
					&	&	\otimes_\R\\
					&	&	\det^{-1}(H^*_{dR, fil}(X/\Z))\otimes_\Z\R
\end{array}
$$
where $H^*_B(X,\Z)$ is the usual Betti cohomology.
Working with a smooth log-scheme over $\Z$ would treat the semistable case.
Remark that the integral structure on the Deligne cohomology determinant is
---not--- given by integral Deligne cohomology (that is a locally compact group;
this may have relations to Morin's global arithmetic cohomology, however),
but by a combination of Betti cohomology with filtered absolute de Rham cohomology.
To give a global analytic interpretation of Scholbach's constructions, here is how
we proceed: We interpret the factor
$$\det(H^*_{mot,et}(X,\Z))\otimes_\Z\det(H^*_{B}(X_\R,\Z))$$
in his determinant as the determinant of the Artin-Verdier \'etale motivic
cohomology $H^*_{et,av}(X,\Z)$ (which contains both integral \'etale motivic information
and Betti information; maybe a $p$-part information should be added following
Milne and Ramachandran \cite{Milne-Ramachandran}; the Weil-\'etale motivic
cohomology would even be better) and the factor
$$
\textstyle\det^{-1}(H^*_{dR, fil}(X/\Z))
$$
as the determinant of filtered de Rham cohomology over $(\Z,|\cdot|_0)$. We conjecture that
the base extension
$$(\Z,|\cdot|_\infty)\to (\Z,|\cdot|_0)$$
of an algebraic motive does not change its filtered de Rham cohomology,
so that we may interpret $H^*_{dR, fil}(X/\Z)$ as the de Rham cohomology
of $X$ over $(\Z,|\cdot|_\infty)$. We refer to subsection \ref{logarithmic-Arakelov}
for a discussion of the problem of finding a model over $(\Z,|\cdot|_\infty)$
of a scheme over $\Z$, seen as a strict dagger space over $(\Z,|\cdot|_0)$.
This allows us to seek for the definition
of a regulator from Artin-Verdier motivic cohomology to filtered de Rham
cohomology over $(\Z,|\cdot|_\infty)$, given by a filtered de Rham realization functor over this
global analytic base.
The fiber of this map of spectra gives back the integral structure on Arakelov motivic cohomology,
and the pairing between motivic homology over
$$U=\{|2|\leq |1|\}\subset X=\Mc(\Z,|\cdot|_\infty)$$
and Arakelov motivic cohomology may be defined in a natural way.
See Remark \ref{Arakelov-motivic-cyclic} for a possible construction of
a global version of Arakelov-motivic cohomology.
\end{remark}

\subsection{Realizations}
Let $t\in \{an,\{an,s\},\dagger,\{\dagger,s\}\}$ be a type of analytic spaces.
Let $\Lambda$ be a sheaf of rings for the pro-\'etale topology on
$\AnSm^t_X$. As before, we denote
$$\eta:\AnSm^\dagger_{X,proet}\to \AnSm^\dagger_{X,et}.$$
\begin{definition}
The \'etale realization of integral motives with coefficients in $\Lambda$ is given
by the composition
$$
\DA_{et}^\dagger(X,\Z)\overset{\eta^*}{\longrightarrow} \DA_{proet}^\dagger(X,\Z)\to
\DA_{proet}^\dagger(X,\Lambda).
$$
\end{definition}
In some particular torsion cases like for example $\Lambda=\Z/n\Z$,
it is possible to show that $\DA_{et}(X,\Lambda)$
is equivalent to the $\infty$-category of sheaves of $\Lambda$-modules
$\Shv(X,\tau_{et},\Mod_{dg}(\Lambda))$.
This is the approach used by Ayoub in \cite{Ayoub-realisation-etale} to
define the \'etale realization.
This may extend nicely to the pro-\'etale situation with coefficients
in pro-\'etale sheaves like $\Zbf_\ell$ or $\hat{\Zbf}$.

Recall that any scheme over $\Z$ may be seen as a non-strict dagger space
over $(\Z,|\cdot|_\infty)$. We will now define a Betti realization for these
objects.
\begin{definition}
Let $X$ be a non-strict dagger analytic space over $\Zbf=(\Z,|\cdot|_\infty)$. The
Betti realization with coefficients in a coefficient category $\Mfk$ is given by
the base extension
$$\SH_{\Mfk}^\dagger(X,\tau)\to \SH_{\Mfk}^\dagger(X_\Rbf,\tau).$$
\end{definition}
The fact that the above definition is reasonable follows from what we said
in Remark \ref{complex-analytic-dagger}: Ayoub's methods
allow us to show that there is a natural equivalence
$$
\Shv(X_\Rbf^{ber},\tau_{usu},\Mfk)\overset{\sim}{\longrightarrow} \SH_{\Mfk}^\dagger(X_\Rbf,\tau),
$$
where $X_\Rbf^{ber}$ is the Berkovich space associated to $X_\Rbf$.

\begin{remark}
If $X$ is a strict dagger analytic space over $\Zbf$, then the diagram
$$
\xymatrix{
\SH_{\Mfk}^{\dagger,s}(X,\tau)\ar[r]\ar[d]	& \SH_{\Mfk}^{\dagger}(X,\tau)\ar[d]\\
\SH_{\Mfk}^{\dagger,s}(X_\Rbf,\tau)\ar[r]^\sim	& \SH_{\Mfk}^{\dagger}(X_\Rbf,\tau)}
$$
is commutative and the down horizontal arrow is an equivalence (this
last fact follows from Proposition \ref{complex-strict-non-strict}).
This means that we may see the Betti realization of a strict dagger motive
over $X$ as a strict dagger motive over $X_\Rbf$.
\end{remark}

The main interest of the theory of overconvergent analytic spaces
is that they have a nice de Rham cohomology theory, as was
already shown by Gro{\ss}e-Kl\"onne in \cite{Grosse-Kloenne2}.
The main point here is that the Poincar\'e Lemma is valid with overconvergent functions
if we work over a base that contains $\Q$.

We will now define the de Rham realization of dagger motives over
a given base by using the associated sheaves on the de Rham
space.

Let $R$ be a base ind-Banach ring and $X$ be a dagger analytic space
over $R$. 
\begin{definition}
The de Rham space of $X$ is the presheaf on $\Alg^\dagger_R$
given by
$$X(A):=X(A/I)$$
where $I\subset \Alg(A)$ is the nilradical. 
A sheaf of quasi-coherent modules on $X_{dR}$ is called a cristal on $X$.
\end{definition}

The de Rham space is functorial in presheaves, so that it is in
particular functorial in morphisms $f:Y\to X$ of analytic spaces.

Let $X$ be a dagger analytic space that is flat over $\Mb(\Z,|\cdot|_\infty)$.
The category of de Rham coefficients on $X$ is the category
$\Shv(X_{dR},\Mod_{dg}(\Oc_{X_{dR}}\otimes_\Z \Q))$.
\begin{theorem}
There are natural realization functors
$$\SH^\dagger(X,\tau)\to \Shv(X_{dR},\Mod_{dg}(\Oc_{X_{dR}}\otimes_\Z \Q))$$
and for $\Lambda\subset \Q$,
$$\DA^\dagger_\tau(X,\Lambda)\to \Shv(X_{dR},\Mod_{dg}(\Oc_{X_{dR}}\otimes_\Z \Q)).$$
\end{theorem}
\begin{proof}
The realization functor will extend the natural relative de Rham cohomology functor
$$
\begin{array}{ccc}
\AnSm^\dagger_{X,\tau}	& \to 	& \Shv(X_{dR},\Mod_{dg}(\Oc_{X_{dR}}\otimes_\Z \Q))\\
\left[Y\to X\right]				& \mapsto	& (f_*\Oc_{Y_{dR}})\otimes \Q
\end{array}
$$
We have to show that if $Y\to X$ is a smooth morphism, then there is a natural
homotopy equivalence
$$(f_*\Oc^{dR}_{Y\times D^\dagger(0,1)_X})\otimes \Q\cong (f_*\Oc^{dR}_Y)\otimes \Q.$$
Using the compatibility of the Kunneth formula with direct image, this will
follow from the computation of the de Rham cohomology of the disc (which works
only in the overconvergent setting), that may be done over the initial base
$(\Z,|\cdot|_\infty)$. Using the fact that we work with $\Q$-coefficients, we get
the Poincar\'e Lemma for overconvergent power series on the disc.
We also need to show that $(f_*\Oc^{dR}_T)\otimes \Q$ is $\otimes$-invertible
in $\Shv(X_{dR},\Mod_{dg}(\Oc_{X_{dR}}\otimes_\Z \Q))$. This follows from
the formula
$$
f_*\Oc^{dR}_T\cong
\cof(f_*\Oc^{dR}_{\G_{m,X}}\to f_*\Oc^{dR}_{\A^1_X}),
$$
that gives that $f_*\Oc^{dR}_T$ is locally free on $X$ of rank $1$, and thus
dualizable over $\Oc_X$. The dual will give the tensor inverse.
\end{proof}

\section{Derived dagger analytic geometry}
\label{derived-dagger-analytic-geometry}
We have defined dagger analytic spaces, by following
the usual method of synthetic geometry, explained in the introduction
of the book \cite{Fred-Towards-the-maths-of-QFT}: we started by improving
the category of rational domain dagger algebras by adding
natural solution spaces for ideals. This was done using
a ``functor of function'' viewpoint. We then used
``functors of points'' to define spaces.
These ideas are close to the ones used by Lawvere \cite{Lawvere-categorical-dynamics} in
synthetic differential geometry and Dubuc and
Taubin \cite{Dubuc-Taubin} in synthetic analytic geometry.
Our main motivation for using this ``synthetic'' approach, as opposed to the usual
approach to classical analytic geometry using locally ringed spaces, is that
it generalizes directly to the derived setting.
We are very much inspired by Lurie's approach to derived analytic geometry
from \cite{Lurie-DAG-V} and \cite{Lurie-DAG-IX} and by the related work in progress of
Mauro Porta on complex analytic derived geometry \cite{Porta-these}.
We refer to this last work for a complete
and neat description of the complex analytic derived theory, including a good theory of
modules. We will now extend the above categories of analytic spaces
to $\infty$-categories of derived analytic spaces. This can be done by
using homotopical functors of functions on categories $\RatAlg^t_R$ of
rational $t$-algebras, that will give derived analytic algebras,
and homotopical functors of points on them.

Before diving into the abstract theory, we will give some motivations for its development.

\subsection{Motivation: global periods}
\label{motivation-global-periods}
One of our main motivations for developing overconvergent global derived analytic geometry
comes from the work of Beilinson and Bhatt (see \cite{Beilinson-derived-de-Rham} and
\cite{Bhatt-derived-de-Rham}) on $p$-adic Hodge theory: they define a ring
of $p$-adic periods by the formula
$$A_{cris}:=\DR(\bar{\Z}_p/\Z_p)\hat{\otimes}\Z_p,$$
where $\DR$ denotes algebraic de Rham cohomology and the
completed tensor product means the homotopy colimit
$$A_{cris}:=\hocolim_n \DR(\bar{\Z}_p/\Z_p)\Lotimes_\Z\Z/p^n\Z.$$
Using Hodge-completed derived de Rham cohomology instead of derived de Rham
cohomology, one gets
$$A_{dR}:=\widehat{\DR}(\bar{\Z}_p/\Z_p)\hat{\otimes}\Z_p$$
and also
$$B^+_{dR}:=\widehat{\DR}(\bar{\Z}_p/\Z_p)\hat{\otimes}\Q_p.$$
Seeking for a geometric interpretation of these derived completed tensor product,
we may interpret $A_{cris}$ as the analytic derived de Rham cohomology of a natural
morphism of strict analytic spaces over $\Zbf_0:=(\Z,|\cdot|_0)$, given by
$$A_{cris}\cong \DR^{an}(\bar{\Z}_p\Lotimes_\Z\Z_p/\Z_p),$$
where $\Z_p$ denotes here the strict derived analytic ring over $\Zbf_0$ given
by $\holim_n \Z/p^n\Z$. It is natural to ask if this cohomology can be compared to the
derived overconvergent analytic de Rham cohomology
$$\DR^{an}(\bar{\Z}_p\Lotimes_{\Zbf_0}\Zbf_p/\Zbf_p)$$
where $\Zbf_p$ denotes the Banach ring $(\Z_p,|\cdot|_p)$ and
$\bar{\Z}_p\Lotimes_{\Zbf_0}\Zbf_p$ denotes the derived analytic ring
over $\Zbf_p$ obtained by extension of scalars
of the non-strict analytic algebra $\bar{\Zbf}_p$ over $(\Z,|\cdot|_0)$
along the bounded morphism $\Zbf_0=(\Z,|\cdot|_0)\to (\Z_p,|\cdot|_p)=\Zbf_p$ of Banach rings. 
Similarly, one would have
$$
B^+_{dR}\cong \widehat{\DR}^{an}(\bar{\Z}_p\Lotimes_{\Zbf_0}\Qbf_p/\Qbf_p).
$$
This interpretation may help to sheafify the construction in the spirit of
Scholze's work \cite{Scholze-p-adic-Hodge-theory} and to globalize it in the spirit
of Bhatt's paper loc. cit., Remark 11.10: one gets global analogs
$$A_{ccris}\cong \DR^{an}(\bar{\Z}\Lotimes_\Z\hat{\Z}/\hat{\Z})$$
and
$$B^+_{ddR}\cong \widehat{\DR}^{an}(\bar{\Z}\Lotimes_\Z\A_f/\A_f),$$
given by the extension of scalar of derived de Rham cohomology of $\bar{\Z}/\Z$
to the ring $\A_f$ of finite adeles, seen as an analytic ring over $(\Z,|\cdot|_0)$.
In our global analytic viewpoint, one may even define naturally, using the base
$\Zbf=(\Z,|\cdot|_\infty)$, a new period ring
$$B^+_{g,ddR}:=\widehat{\DR}^{an}(\bar{\Zbf}\Lotimes_\Zbf\A/\A)$$
that also takes care of the archimedean component
$$B_{\infty,ddR}^+:=\widehat{\DR}^{an}(\bar{\Zbf}\Lotimes_\Zbf\R/\R).$$
In the archimedean situation of Hodge theory, one usually only uses the Galois group
of $\C$ over $\R$, but knowing that a variety is defined over $\Z$ may be an important information to
be used in archimedean Hodge theory.
One may even study the groupoid derived stack
$$\R\Hom_{\A^1_\Z}(D_*,\A^1_{\bar{\Zbf}})\rightrightarrows \A^1_{\Zbf}$$
(where $D_*$ is the cosimplicial scheme that is given in degree $n$ by the union of
the $n+1$ coordinate axis in $\A^{n+1}$ and the face and degeneracies on $[n\mapsto \A^{n+1}]$
are given by addition of coordinates and insertion of zeroes)
that encodes (when derived pullbacked to $\A$ and completed along the unit section)
the derived Hodge filtration on $B^+_{g,ddR}$.
Its pullback at the archimedean place gives back the archimedean Hodge filtration of $\bar{\Zbf}$,
and its pullback on $\hat{\Z}$ gives back the Hodge filtration on the global period ring.

Now, if we want to adapt Beilinson's strategy from \cite{Beilinson-derived-de-Rham} in
this global case, we can proceed in the following way: define the sheaf $\Bc^+_{dR}$
of filtered dg-algebras on the $h$-topology on the category $\Var_{\bar{\Q}}$ of quasi-projective
varieties by sheafifying for the $h$-topology the presheaf that sends a semistable pair
$(U,\bar{U})$ over $\bar{\Z}$ (with $\bar{U}$ projective over $\bar{\Z}$) to the
Hodge completed analytic derived de Rham cohomology
$$B^+_{dR}(U,\bar{U}):=\widehat{\DR}^{an}((\bar{U},\Mc_U)\Lotimes_\Zbf\A/\A)$$
of the corresponding strict dagger logarithmic space over $\Zbf$
(defined in Subsection \ref{logarithmic-Arakelov}), extended to $\A$.
One then defines the global Arithmetic de Rham complex of $X$ as
$$\R\Gamma^+_{dR}(X):=\R\Gamma(X_h,\Bc^+_{dR}).$$
There is a natural diagram
$$
H^*_{dR}(X)\overset{\alpha}{\longrightarrow}
H^*(\R\Gamma^+_{dR}(X))\overset{\beta}{\longleftarrow}
H^*_{proet}(X,\A)\otimes_\A B^+_{g,ddR}.
$$
Now the global analog of the Poincar\'e Lemma would be that the natural morphism
of sheaves on the $h$-topology
$$
B^+_{g,ddR}\overset{\sim}{\longrightarrow} \Bc^+_{dR}
$$
is a filtered quasi-isomorphism.
This would imply that $\beta$ is an isomorphism, and the
base extension to $B_{g,ddR}$ of the corresponding morphism
$$
H^*_{dR}(X)\to H^*_{proet}(X,\A)\otimes_\A B^+_{g,ddR}
$$
would then be the global period isomorphism.

We may also try to follow Scholze's approach from \cite{Scholze-p-adic-Hodge-theory} to propose
a strategy to prove a global version of his $p$-adic comparison theorem.
Let $X/\Z$ be a flat scheme over an open subset of $\Spec(\Z)$,
and whose generic fiber $X_\Q/\Q$ is proper and smooth.
One may define a period sheaf $\B^+_{dR}$ on the pro-\'etale site of $X_\Q$ by
$$\B^+_{dR}(U):=\widehat{\DR}^{an}(U_{\A_f}/\A_f).$$
This induces a period sheaf $\B^+_{dR}$ on the pro-\'etale site of $X_{\bar{\Q}}$ by
$$\B^+_{dR}(U):=\widehat{\DR}^{an}(U_{\A_f}/\A_f).$$
One should have an isomorphism
$$
H^*_{proet}(X_{\bar{\Q}},\A_f)\otimes_{\A_f}B^+_{ddR}\overset{\sim}{\longrightarrow}
H^*_{proet}(X_{\bar{\Q}},\B^+_{dR}),
$$
and an isomorphism
$$
H^*_{dR}(X_\Q/\Q)\otimes_\Q B_{ddR}\overset{\sim}{\longrightarrow}
H^*_{proet}(X_{\bar{\Q}},\B^+_{dR})\otimes_{B^+_{ddR}} B_{ddR}
$$
given by a Poincar\'e Lemma similar to the one proved by Scholze in the local setting,
that identifies the global sheaf of constants
$\B^+_{dR}$ to the horizontal sections of the natural connection on $\Oc\B^+_{dR}$.
All this would imply that there exists a natural isomorphism
$$
H^*_{proet}(X_{\bar{\Q}},\A_f)\otimes_{\A_f}B_{ddR}\overset{\sim}{\longrightarrow}
H^*_{dR}(X_\Q,\Q)\otimes_\Q B_{ddR}
$$
compatible with the filtration and the Galois action on Both sides.
Adding the archimedean information to the above reasoning is a quite tempting
task, if one uses the global period sheaf
$$\B^+_{g,ddR}(U):=\widehat{\DR}^{an}(U_\A/\A)$$
and the global period ring
$$B^+_{g,ddR}:=\widehat{\DR}^{an}(\bar{\Zbf}\otimes_{\Zbf}\A/\A).$$
For this naive idea to work, one needs to think of $\Q$ not as a Banach ring
but as a non-strict analytic ring $\Qbf$ over $\Zbf:=(\Z,|\cdot|_\infty)$ given by germs
of functions around the trivial norm $|\cdot|_0\in \Mc(\Zbf)$. One then defines
$\bar{\Qbf}:=\bar{\Zbf}\otimes_\Zbf \Qbf$. This way, it is
meaningful to study the pro-\'etale cohomology of $X_{\bar{\Qbf}}$ with coefficients
in the full ring of ad\`eles, and one may still have a comparison isomorphism
$$
H^*_{proet}(X_{\bar{\Qbf}},\A)\otimes_\A B_{g,ddR}\overset{\sim}{\longrightarrow}
H^*_{dR}(X/\Q)\otimes_\Q B_{g,ddR}.
$$

\begin{remark}
Since Betti cohomology may be nicely computed as analytic motivic cohomology,
and the classical comparison isomorphism is between Betti cohomology and
de Rham cohomology, one may try to generalize this
isomorphism in the motivic direction, by trying to relate (\'etale) motivic cohomology
to a kind of global derived analog of Deligne cohomology (i.e., a motivically graded version
of global multiplicative $K$-theory, that combines global analytic $K$-theory with
the Hodge filtration). This question will be studied later.
\end{remark}

\subsection{Derived dagger algebras}
Let $R$ be a uniform ind-Banach ring. We denote $\Sc$ the $\infty$-category of spaces,
obtained by the $\infty$-localization of the category $\SSets$ of simplicial sets by
weak equivalences.
\begin{definition}
Let $t\in \{an,\dagger\}$ be a type of analytic space.
A \emph{derived $t$-algebra (resp. very strict derived $t$-algebra)} over $R$ is a functor
$$
\begin{array}{c}
A:(\RatAlg^t_R)^{op}\to \Sc\\
\textrm{(resp. }A:(\RatAlg^{t,s}_R)^{op}\to \Sc\textrm{))}
\end{array}
$$
that commutes to finite products and sends pullbacks along
rational domain immersions to pullbacks.
A derived (resp. very strict derived) $t$-algebra over $R$ is called
an \emph{affinoid $t$-algebra} (resp. a very strict affinoid $t$-algebra)
if it is finitely presented, i.e., the finite colimit of a diagram of rational domain $t$-algebras
(resp. strict rational domain $t$-algebras).
We will denote $\DAlg_{R}^t$ (resp. $\DAlg^{t,vs}_{R}$, resp. $\DAff_{R}^t$,
resp. $\DAff_{R}^{t,vs}$) the $\infty$-category of derived $t$-algebras
(resp. very strict derived $t$-algebras, resp. derived affinoid $t$-algebras,
resp. derived very strict affinoid $t$-algebras). The $\infty$-category
$\DAff_R^{t,s}$ of strict affinoid $t$-algebras is defined as the smallest
subcategory of $\DAff_R^t$ that contains (derived) coequalizers
$$
\xymatrix{
C\ar@<0.7ex>[r]^f\ar@<-0.4ex>[r]_g & B\ar[r]	& A}
$$
of two morphisms in $\RatAlg^t_R$, with $B$ strict, and that is stable
by pushouts and retractions. We denote $\DAlg_R^{t,s}:=\ind\DAff_R^{t,s}$ the associated
$\infty$-category of algebras.
\end{definition}

\begin{proposition}
Let $t\in \{an,\dagger\}$ be a non-strict type of analytic spaces.
The $\infty$-category opposite to $\DAff_{R}^t$ (resp. $\DAff_R^{t,vs}$, resp $\DAff_R^{t,s}$),
equipped with the Grothendieck topology generated by standard rational domain coverings
is a geometry in the sense of Lurie \cite{Lurie-DAG-V}, Definition 1.2.5.
For $u\in \{t,\{t,vs\},\{t,s\}\}$, one has
$$\DAlg^u_R=\ind\DAff_R^u.$$
The geometry given by $\DAff_{R}^t$ (resp. $\DAff_R^{t,vs}$) is the geometric
envelope of the pre-geometry given by $\RatAlg^t_R$ (resp. $\RatAlg_R^{t,s}$).
\end{proposition}
\begin{proof}
The construction of $\DAff_R^t$ from $\RatAlg^t_R$ and of $\DAff^{t,vs}_R$ from
$\RatAlg^{t,s}_R$ shows that they are given by geometric envelopes,
as described by Lurie in \cite{Lurie-DAG-V}, Lemma 3.4.3 (see also \cite{Lurie-higher-topos-theory},
5.3.6.2).
The statement about ind-objects follows from the fact that the category of small
derived $u$-algebras is generated under the combination of finite colimits and filtered colimits
by rational domain algebras. It remains to check that $\DAff^{t,s}_R$ is indeed a geometry.
By definition, it is stable by pushouts and retractions. Admissible morphisms are given
by strict rational domain algebras, and they indeed form an admissibility structure,
as shown in Proposition \ref{rational-pre-geometries}.
\end{proof}

\begin{definition}
A \emph{derived $t$-analytic scheme} is a scheme for the geometry $\DAff^t_R$ in the sense
of Lurie \cite{Lurie-DAG-V}, Definition 2.3.9. More precisely, it is a $\DAff^t_R$-structured
$\infty$-topos $(\Xc,\Oc_\Xc)$ that is covered by affine $\DAff^t_R$-schemes
(representable ones).
By \cite{Lurie-DAG-V}, Theorem 2.4.1, a scheme corresponds to an $\infty$-stack
$X\in \Shv(\DAlg^t_R,\tau_{\Rat},\Sc)$ that is locally isomorphic to a representable
stack $\Mb(A):=\Map_{\DAlg^t_R}(A,\_)$.
\end{definition}

By replacing the analytic topology by the etale topology on $\DAff^t_R$, one may also
define Deligne-Mumfor derived $t$-analytic stacks.

\subsection{The dagger cotangent complex and derived de Rham cohomology}
\label{derived-de-Rham}
Let $t\in \{\dagger,\{\dagger,s\}\}$ be an overconvergent type of analytic spaces.

One uses the tangent $\infty$-category approach (stabilization
of the overcategory) of
\cite{Lurie-DAG-IV} and \cite{Lurie-higher-algebra}, 7.3,
to define quasi-coherent modules on derived analytic algebras
and derived analytic spaces. This also gives a definition of
the cotangent complex and of the de Rham space $X_{dR}$
associated to a derived dagger analytic space.

\begin{definition}
If $A\in \DAlg^t_R$, we denote $\DAlg^t_{A\backslash}$ the pointed $\infty$-category
with finite colimits whose objects are morphisms $A\to B$. The $\infty$-category
$\Mod(A)$ of modules over $A$ is defined as the tangent $\infty$-category of
$\DAlg^t_R$ at $A$, given by the stabilization
$$\Mod(A)=T_A\DAlg_t:=\Stab(\DAlg^t_{A\backslash}).$$
\end{definition}

\begin{definition}
The right adjoint $\Lb$ to the natural forgetful functor
$$\Mod(A)\to \DAlg^t_{A\backslash}$$
is called the \emph{dagger cotangent complex}, and we denote $\Lb(B)$ by $\Lb_{B/A}$.
\end{definition}

\begin{proposition}
The $\infty$-category $\Mod(A)$ is equipped with a natural symmetric monoidal structure
$\otimes$ with unit object $\1_A$ that makes it a symmetric monoidal $\infty$-category
$(\Mod(A),\otimes,\1_A)$.
\end{proposition}
\begin{proof}
The tensor product of two modules $M$ and $N$ with corresponding spectral $A$-algebras
$B_M$ and $B_N$ is defined as the cofiber
$$B_{M\otimes N}:=\cof(B_M\oplus B_N\to B_M\otimes_A B_N).$$
If $A$ is a $t$-rational $R$-algebra,
we will denote $\1_A$ the module over $A$ given by the $t$-affinoid $R$-algebra
$$\1_A:=A\{\Alg(A)\}^t/((\{a,\;a\in \Alg(A)\})^2,(a([b]+[c])-[a(b+c)],a\in \Alg(A),\;b\in \Alg(A))),$$
where the set $\Alg(A)$ is equipped with the grading given by the uniform norm
on $\Mc(A)$, which is well defined since $\Mc(A)$ is compact.
If $A$ is a derived $t$-analytic $R$-algebra, we may write it as a colimit
$$A=\colim_i A_i$$
of $t$-rational $R$-algebras $A_i$, and we define $\1_A$ as the colimit of the
corresponding modules $\1_{A_i}$. The above binary tensor product operation
extends naturally to a symmetric monoidal $\infty$-category structure with
unit object $\1_{A}$ (model of the Lawvere theory of commutative monoids in
the $\infty$-category ${}^\infty\Cat^{pr}$ of $\infty$-categories).
\end{proof}

\begin{definition}
Let $A$ be a derived $t$-analytic algebra. The symmetric monoidal $\infty$-category
$(\Perf(A),\otimes,\1_\A)$ of \emph{perfect complexes} over $A$ is the symmetric
monoidal stable sub-$\infty$-category of $\Mod(A)$ generated by $\1_A$. 
\end{definition}

One must be careful in extending the definition of the cotangent complex of a morphism
of dagger algebras to the geometric situation of a morphism of derived dagger analytic
spaces $f:X\to Y$.
Actually, even for the spectrum $f:X=\Mb(A)\to \Mb(R)=Y$ of an $R$-affinoid $t$-analytic
algebra, one needs to define the category of quasi-coherent modules on $X$ in a local way
for the $G$-topology, as was pointed out to the author, on a  $p$-adic example due to Gabber
by Brian Conrad (see \cite{Conrad-relative-ampleness-rigid-analytic}, Remark 2.1.5
and Example 2.1.6).
\begin{definition}
Let $X$ be a $t$-analytic space over $R$. The category $\QCoh(X)$ of quasi-coherent
modules over $X$ is defined as the (opposite) tangent category
$$\QCoh(X)^{op}:=T_X(\An^t)^{op}$$
to the category of analytic spaces at $X$, defined as the category of abelian
co-group objects in the category of morphisms $f:Y\to X$.
Similarly, if $X$ is a derived $t$-analytic space over $R$, the derived category
$\DQCoh(X)$ of quasi-coherent modules over $X$ is defined as the (opposite) $\infty$-tangent
category
$$\DQCoh(X)^{op}:=T_X(\DAn^t)^{op}$$
to the category of derived analytic spaces at $X$, defined as the stabilization
of the category of morphisms $f:Y\to X$.
The cotangent complex functor $\Lb$ is given by the adjoint of the forgetful functor
$$\DQCoh(X)\to \DAn^t.$$
We denote the module $\Lb(Y)$ by $\Lb_{Y/X}$. 
\end{definition}

We may still define the symmetric monoidal $\infty$-category of perfect complexes
in this geometric situation.
\begin{proposition}
The $\infty$-category $\DQCoh(X)$ is equipped with a natural symmetric monoidal
structure $\otimes$ with unit object denoted $\1_X$. The symmetric monoidal
stable $\infty$-category generated by $\1_X$ is denoted $(\Perf(X),\otimes,\1_X)$.
Every object of $\Perf(X)$ is strongly dualizable for the monoidal structure.
\end{proposition}
\begin{proof}
The same constructions work as in the situation of analytic algebras.
The fact that every perfect complex is strongly dualizable follows from
the fact that $\1_X$ is the unit object for the stable monoidal structure.
\end{proof}

\begin{example}
If $(R,|\cdot|_0)$ is a trivially seminormed integral ring and if we work with strict
analytic spaces, then we get back the usual cotangent complex of algebraic
geometry defined originally by Illusie in his thesis \cite{Illusie-cotangent},
and the usual notion of perfect complexes on an $R$-scheme.
\end{example}

One may also define derived de Rham cohomology of dagger analytic spaces.
\begin{definition}
Let $f:X\to Y$ be a morphism of derived analytic spaces and $\Lb_{X/Y}$ be the
corresponding cotangent complex. The derived de Rham complex (resp.
completed derived de Rham complex) is defined by
$$
\begin{array}{c}
\Omega^*_{X/Y}			 :=	 \left|\wedge^*\,\Lb_{X/Y}\right|\\
\left(\textrm{resp. }\widehat{\Omega}^*_{X/Y}	 :=	 \left|\wedge^*\,\Lb_{X/Y}\right|_*\right),
\end{array}
$$
where the exterior product is meant as the derived exterior product in the derived category
of $\Oc_X$-modules and the sign $|\cdot|$ (resp. $|\cdot|_*$) means the totalization
(resp. the product totalization) of the bicomplex.
The derived de Rham (resp. Hodge completed derived de Rham) cohomology
of $X$ over $Y$ is given by
$$
\begin{array}{c}
\DR^*(X/Y):=\R\Gamma(X,\Omega^*_{X/Y})\\
\left(\textrm{resp. }\widehat{\DR}^*(X/Y):=\R\Gamma(X,\widehat{\Omega}^*_{X/Y})\right).
\end{array}
$$
\end{definition}

The derived and completed derived de Rham complexes are equipped with natural commutative
algebra structures.

\begin{remark}
In characteristic $0$, it is explained by Bhatt
in \cite{Bhatt-derived-de-Rham}, Remark 2.6, that derived de Rham cohomology is trivial
in the algebraic setting, and that one really needs to pass to the Hodge completed setting
to get back usual de Rham cohomology in the case of schemes
(this is another important result due to Bhatt \cite{Bhatt-completed-derived-de-Rham}).
However, the non-completed version plays a central role in the definition of
refined period rings, such as $A_{cris}$, so that one really needs to use it.
\end{remark}

\section{Chern characters and global regulators}
\label{Chern-character}
We will now use the formalism of Toen and Vezzosi \cite{Toen-Vezzosi-Chern}
for the cyclic Chern character, and adapt it to higher $K$-theory by getting
inspiration from Blanc's thesis \cite{Blanc-K-theorie-topologique}.
The fact that the Chern character in cyclic homology is \emph{integral}
(explained to the author by Gregory Ginot) is our main motivation for
working on its adaptation to our global analytic setting, with the aim
of defining various local and global ``regulator type'' maps.

\subsection{Analytic Waldhausen $K$-theory and cyclic homology}
For the following constructions to work with integral coefficients,
we need a notion of analytic $K$-theory that is not necessarily homotopy
invariant, so that neither the usual Karoubi-Villamayor approach
\cite{Karoubi-Villamayor}, neither the dagger $\Db^1$-homotopical
approach described in Section \ref{global-analytic-motives} will give us
what we need. We will thus use the Waldhausen approach explained
on the nlab contributive website (following \cite{Toen-Vezzosi-K-theory};
see also \cite{Lurie-higher-algebra}, Remark 11.4).

First recall that from a stable $\infty$-category, one may define
the associated Waldhausen $K$-theory by the following definition.

\begin{definition}
Let $\Cc$ be an $\infty$-category. The \emph{core} of $\Cc$ is the maximal
sub-$\infty$-groupoid $\Core(\Cc)$ of $\Cc$. The $n$-gap of $\Cc$ is the
full sub-$\infty$-category $\Gap(\Cc^{\Delta^n})$ of $\Func(\Arr(\Delta^n),\Cc)$
on those objects $F$ for which
\begin{itemize}
\item the diagonal $F(n,n)$ is inhabited by zero objects, for all $n$;
\item all diagrams of the form
$$
\xymatrix{
F(i,j)\ar[d]\ar[r]	& F(i,k)\ar[d]\\
F(j,j)\ar[r]		& F(j,k)}
$$
is an $\infty$-pushout.
\end{itemize}
The (connective) $K$-theory spectrum of an $\infty$-category $\Cc$ with pushouts
is defined by
$$\Kbf(\Cc):=\colim \Core(\Gap(\Cc^{\Delta^n})).$$
The universal completion of the functor
$$\Kbf:{}^\infty\Cat^{st}\to \Sp$$
from stable $\infty$-categories to spectra that sends homotopy cofibers of stable $\infty$-category
to homotopy cofibers of spectra is the corresponding unconnective $K$-functor
$$\Kbf^{nc}:{}^\infty\Cat^{st}\to \Sp.$$
\end{definition}

If we have a geometric derived analytic stack $X$, we will be interested by
the associated monoidal $\infty$-category $(\Perf(X),\otimes)$ of perfect complexes,
and the corresponding $K$-theory spectrum
$$\Kbf(X):=\Kbf(\Perf(X)).$$

\begin{remark}
\label{devissage-K-theory}
In our global analytic setting, it may be necessary to give a refined definition of $K$-theory,
e.g., using simplicial methods \`a la Karoubi-Villamayor \cite{Karoubi-Villamayor}, or
motivic analytic methods like in our paper, to get the following property:
if $X$ is a nice strict analytic space over $\Zbf=(\Z,|\cdot|_\infty)$
(e.g. the one associated to a projective smooth scheme), and we denote
$\Zbf_0=(\Z,|\cdot|_0)$, then the double inclusion
$$U=\{|2|\leq |1|\}\subset X\supset \{|2|>1\}=Z$$ induces
an exact triangle
$$\Kbf(X_\Rbf)\to \Kbf(X_\Zbf)\to \Kbf(X_{\Zbf_0}),$$
that relates global analytic $K$-theory
of $X$ to algebraic $K$-theory of the underlying scheme and to
topological $K$-theory of the associated analytic space over $\Rbf$.
It is also possible that working only with perfect modules will not be enough,
since the above decomposition is sub-analytic.
\end{remark}

Let ${}^\infty\Mon$ be the $\infty$-category of $\infty$-monoids, given
by models of the Lawvere theory $(\Tc_{Mon},\times)$ of commutative monoids (category
with finite products opposite to that of free finitely generated commutative monoids)
with values in the $\infty$-category ${}^\infty\Grpd$ of $\infty$-groupoids
(i.e., simplicial sets up to weak equivalences).
Let $t\in \{an,\{an,s\},\dagger,\{\dagger,s\}\}$ be a type of analytic spaces.
\begin{definition}
The \emph{Hochshild homology pre-stack}
$$\HH_{pr}:\DAn^t_R\to {}^\infty\Mon$$
is defined on the $\infty$-category of derived $t$-analytic spaces over $R$
by sending an analytic space $X$ to the $\infty$-monoid $\End_{\Perf(X^{S^1})}(\1)$
of endomorphisms of the unit object in the $\infty$-symmetric monoidal category of
perfect complexes on the derived loop space $X^{S^1}$ of $X$.
The associated stack is called the \emph{Hochshild homology stack} and denoted
$$\HH:\DAn^t_R\to {}^\infty\Mon.$$
\end{definition}

The Hochshild homology pre-stack
$$\HH_{pr}:\DAn^t_R\to {}^\infty\Mon$$
has actually a natural lifting
$$\widetilde{\HH}_{pr}:\DAn^t_R\to \R\uHom(BS^1,{}\infty\Mon)=:S^1-{}^\infty\Mon$$
to the $\infty$-category of $S^1$-equivariant $\infty$-monoids.
The associated stack also gives
$$\widetilde{\HH}:\DAn^t_R\to S^1-{}^\infty\Mon.$$
\begin{definition}
The functors obtained by composing the above lifting of the Hochshild homology
functors with the functor of homotopy fixed point
$$(\cdot)^{hS^1}:=\lim_{BS^1}:S^1-{}^\infty\Mon\to {}^\infty\Mon,$$
defines two $\infty$-functors
$$
\begin{array}{c}
\HC_{pr}^{neg}:\DAn^t_R\to {}^\infty\Mon\\
\HC^{neg}:\DAn^t_R\to {}^\infty\Mon
\end{array}
$$
called respectively the \emph{negative cyclic pre-stack and negative cyclic stack}.
Similarly, composing with the $\infty$-functor of homotopy coinvariants
$$(\cdot)_{hS^1}:=\colim_{BS^1}:S^1-{}^\infty\Mon\to {}^\infty\Mon,$$
defines two new $\infty$-functors
$$
\begin{array}{c}
\HC_{pr}:\DAn^t_R\to {}^\infty\Mon\\
\HC:\DAn^t_R\to {}^\infty\Mon
\end{array}
$$
called the \emph{cyclic homology pre-stack and the cyclic homology stack}.
\end{definition}

There is actually an equivalence (see \cite{Blanc-K-theorie-topologique}, Section 2.1)
of $\infty$-categories
$${}^\infty\Mon\overset{\sim}{\to} \Sp^{con}$$
between the $\infty$-category of $\infty$-monoids and the $\infty$-category $\Sp^{con}$ of
connective spectra, that we may use to define the cyclic homology spectra of an analytic space,
\begin{definition}
Let $X$ be a $t$-analytic space. The cyclic homology spectra of $X$ are defined as
the spectra $\HCbf^{neg}(X)$ and $\HCbf(X)$ associated to the $\infty$-monoids
$\HC^{neg}(X)$ and $\HC(X)$.
\end{definition}


\subsection{The cyclic Chern character}
We want to adapt Toen-Vezzosi's construction of the Chern character
from \cite{Toen-Vezzosi-Chern} to get morphisms of spectra
$$\ch:\Kbf(X)\to \HCbf^{neg}(X)$$
and
$$\ch:\Kbf(X)\to \HCbf(X).$$
that induce the usual Chern character (for the negative one)
if we work over a regular Banach ring of characteristic $0$.

This may be easily extended to logarithmic analytic spaces if one uses the
tangent category definition for their category of perfect complexes, as
we did for classical analytic spaces.

To show that Toen and Vezzosi's work extends to a Chern character
$$\ch:\Kbf(X)\to \HCbf^{neg}(X)$$
we need to show that
the trace morphism may be extended to every $\Gap(\Cc^{\Delta^n})$
in a compatible way with the simplicial maps. This should follow from the
derived additivity of traces alluded to in the introduction of loc. cit., Subsection 2.4:
one has to show that the cyclic trace $\Tr^{S^1}$ restricted to a stable symmetric
monoidal $\infty$-category is compatible with exact triangles.


We leave the details of these constructions for a later publication.
We will finish by discussing in a somewhat imprecise
way a possible application of our formalism to the definition of an integral version
of rational Arakelov motivic cohomology \cite{Arakelov-motivic-cohomology-I}.
\begin{remark}
\label{Arakelov-motivic-cyclic}
Recall that we denoted
$\Zbf_0:=(\Z,|\cdot|_0)$, $\Zbf:=(\Z,|\cdot|_\infty)$ and $\Rbf=(\R,|\cdot|_\infty)$.
In this remark, we will work with the analytic etale topology, and in particular, etale $K$-theory,
and we will only consider projective schemes (logarithmic methods are necessary for more general
ones). We suppose given a good notion of analytic $K$-theory, such as the one proposed
in Remark \ref{devissage-K-theory}.
We may now continue the discussion of Remark \ref{Arakelov-motivic} on a possible
global analytic version of Arakelov motivic cohomology. The existence of an integral
Chern character with values in cyclic homology allows us to define an analog of Deligne
cohomology for a proper smooth analytic space $X$ over any base Banach
ring $R=(R,|\cdot|_R)$, defined as the homotopy fiber
$$
\Kbf_\Dc(X)\longrightarrow
\Kbf(X)\overset{\ch}{\longrightarrow} \HCbf(X).
$$
One has natural $\lambda$-operations on $\Kbf_\Dc(X)$ and a natural $\Kbf(X)$-module
structure
$$
\Kbf(X)\times \Kbf_\Dc(X)\to \Kbf_\Dc(X)
$$
that comes from the fact that $\Kbf(X)\to \HCbf(X)$ is a ring morphism.
We may try to use constructions of this kind to define an \emph{integral Arakelov motivic $K$-theory}
$\widehat{\Kbf}(X_\Zbf)$ fulfilling the following conditions:
\begin{enumerate}
\item there is a module structure
$$\Kbf(X_{\Zbf_0})\times \widehat{\Kbf}(X_\Zbf)\to \widehat{\Kbf}(X_\Zbf)$$
under algebraic $K$-theory
\item that induces by composition with the natural projection
$\widehat{\Kbf}(X_\Zbf)\to \widehat{\Kbf}(\Zbf)$ an integrally defined
$\Lambda$-filtered pairing
$$\Kbf(X_{\Zbf_0})\times \widehat{\Kbf}(X_\Zbf)\to \widehat{\Kbf}(\Zbf).$$
\item The determinant of the associated $\Lambda$-graded group
$\gr_\Lambda \widehat{\Kbf}(X_\Zbf)$ tensored with $\Q$ would be isomorphic to
$$
\det_\Q(H^*_{mot}(X,\Q))\otimes \det_\Q(H^*(X_\Rbf,\Q)) \otimes {\det_\Q}^{-1}(H^*_{dR,fil}(X_\Q/\Q)).
$$
\end{enumerate}
The main interest of our new methods is that they would allow us to avoid the introduction
of denominators in the definition of the fundamental pairing used by Scholbach in his approach
\cite{Scholbach-special-L-values} to special $L$-values, so that one may hope
the above pairing or the associated graded integral pairing
$$
\gr_\Lambda \Kbf(X_{\Zbf_0})\times \gr_\Lambda\widehat{\Kbf}(X_\Zbf)\to
\gr_\Lambda\widehat{\Kbf}(X_\Zbf)
$$
to be related to the special values of the $L$-function of $X$ up to $\Z^\times=\{\pm 1\}$.
This picture may be too optimistic, but the idea of using a derived global analytic and
cyclic version of the Chern character to avoid the introduction of denominators in the theory of
special $L$-values certainly deserves further attention.
\end{remark}

\bibliographystyle{alpha}
\bibliography{$HOME/Documents/travail/fred}

\def\cprime{$'$} \def\cprime{$'$} \def\cprime{$'$} \def\cprime{$'$}
  \def\cprime{$'$} \def\cprime{$'$}
\begin{thebibliography}{{Bha}12b}

\bibitem[Ara01]{Arabia}
Alberto Arabia.
\newblock Rel\`evements des alg\`ebres lisses et de leurs morphismes.
\newblock {\em Comment. Math. Helv.}, 76(4):607--639, 2001.

\bibitem[Art67]{E-Artin1}
Emil Artin.
\newblock {\em Algebraic numbers and algebraic functions}.
\newblock Gordon and Breach Science Publishers, New York, 1967.

\bibitem[Ayo07a]{Ayoub-six-operations-I}
Joseph Ayoub.
\newblock Les six op\'erations de {G}rothendieck et le formalisme des cycles
  \'evanescents dans le monde motivique. {I}.
\newblock {\em Ast\'erisque}, 314:x+466 pp. (2008), 2007.

\bibitem[Ayo07b]{Ayoub-six-operations-II}
Joseph Ayoub.
\newblock Les six op\'erations de {G}rothendieck et le formalisme des cycles
  \'evanescents dans le monde motivique. {II}.
\newblock {\em Ast\'erisque}, 315:vi+364 pp. (2008), 2007.

\bibitem[Ayo10]{Ayoub-Betti}
Joseph Ayoub.
\newblock Note sur les op\'erations de {G}rothendieck et la r\'ealisation de
  {B}etti.
\newblock {\em J. Inst. Math. Jussieu}, 9(2):225--263, 2010.

\bibitem[Ayo11]{Ayoub-analytic-motives}
Joseph Ayoub.
\newblock Motifs des vari\'et\'es analytiques rigides.
\newblock {\em preprint}, 2011.

\bibitem[Ayo14]{Ayoub-realisation-etale}
Joseph Ayoub.
\newblock La r\'ealisation \'etale et les op\'erations de grothendieck.
\newblock {\em Preprint}, 2014.

\bibitem[{Bam}14]{Bambozzi}
F.~{Bambozzi}.
\newblock {On a generalization of affinoid varieties}.
\newblock {\em ArXiv e-prints}, January 2014.

\bibitem[BB15]{Bambozzi-Ben-Bassat}
F.~{Bambozzi} and O.~{Ben-Bassat}.
\newblock {Dagger Geometry As Banach Algebraic Geometry}.
\newblock {\em ArXiv e-prints}, February 2015.

\bibitem[{Bei}11]{Beilinson-derived-de-Rham}
A.~{Beilinson}.
\newblock {p-adic periods and derived de Rham cohomology}.
\newblock {\em ArXiv e-prints}, February 2011.

\bibitem[Ber90]{Berkovich1}
Vladimir~G. Berkovich.
\newblock {\em Spectral theory and analytic geometry over non-{A}rchimedean
  fields}, volume~33 of {\em Mathematical Surveys and Monographs}.
\newblock American Mathematical Society, Providence, RI, 1990.

\bibitem[Ber93]{Berkovich-etale-cohomology}
Vladimir~G. Berkovich.
\newblock \'{E}tale cohomology for non-{A}rchimedean analytic spaces.
\newblock {\em Inst. Hautes \'Etudes Sci. Publ. Math.}, 78:5--161 (1994), 1993.

\bibitem[BGR84]{BGR}
S.~Bosch, U.~G{\"u}ntzer, and R.~Remmert.
\newblock {\em Non-{A}rchimedean analysis}, volume 261 of {\em Grundlehren der
  Mathematischen Wissenschaften [Fundamental Principles of Mathematical
  Sciences]}.
\newblock Springer-Verlag, Berlin, 1984.
\newblock A systematic approach to rigid analytic geometry.

\bibitem[{Bha}12a]{Bhatt-completed-derived-de-Rham}
B.~{Bhatt}.
\newblock {Completions and derived de Rham cohomology}.
\newblock {\em ArXiv e-prints}, July 2012.

\bibitem[{Bha}12b]{Bhatt-derived-de-Rham}
B.~{Bhatt}.
\newblock {p-adic derived de Rham cohomology}.
\newblock {\em ArXiv e-prints}, April 2012.

\bibitem[BK13]{Ben-Bassat-Kremnitzer}
Oren {Ben-Bassat} and Kobi {Kremnitzer}.
\newblock {Non-Archimedean analytic geometry as relative algebraic geometry}.
\newblock {\em ArXiv e-prints}, December 2013.

\bibitem[{Bla}13]{Blanc-K-theorie-topologique}
A.~{Blanc}.
\newblock {Invariants topologiques des Espaces non commutatifs}.
\newblock {\em ArXiv e-prints}, July 2013.

\bibitem[CD09]{Cisinski-Deglise-mixed-motives}
D.-C. {Cisinski} and F.~{D{\'e}glise}.
\newblock {Triangulated categories of mixed motives}.
\newblock {\em ArXiv e-prints}, December 2009.

\bibitem[Con06]{Conrad-relative-ampleness-rigid-analytic}
Brian Conrad.
\newblock Relative ampleness in rigid geometry.
\newblock {\em Ann. Inst. Fourier (Grenoble)}, 56(4):1049--1126, 2006.

\bibitem[Del71]{De12}
Pierre Deligne.
\newblock Th\'eorie de {H}odge. {II}.
\newblock {\em Inst. Hautes \'Etudes Sci. Publ. Math.}, 40:5--57, 1971.

\bibitem[Die12]{Diekert-these-master}
Nikolai Diekert.
\newblock {\em Der Tilt von \"uberkonvergenten Potenzreihen}.
\newblock Diploma. Freiburg Universit\"at, 2012.

\bibitem[DK14]{Davis-Kedlaya-almost-purity}
C.~{Davis} and K.~S. {Kedlaya}.
\newblock {Almost purity for overconvergent Witt vectors}.
\newblock {\em ArXiv e-prints}, March 2014.

\bibitem[DM12]{Deglise-Mazzari}
F.~{D{\'e}glise} and N.~{Mazzari}.
\newblock {The rigid syntomic ring spectrum}.
\newblock {\em ArXiv e-prints}, November 2012.

\bibitem[DT83]{Dubuc-Taubin}
Eduardo Dubuc and Gabriel Taubin.
\newblock Analytic rings.
\newblock {\em Cahiers Topologie G\'eom. Diff\'erentielle}, 24(3):225--265,
  1983.

\bibitem[Dur07]{Durov-2007}
Nikolai Durov.
\newblock {\em New Approach to {A}rakelov {G}eometry}.
\newblock arXiv.org:0704.2030, 2007.

\bibitem[DZ94]{Dubuc-Zilber}
Eduardo~J. Dubuc and Jorge~G. Zilber.
\newblock On analytic models of synthetic differential geometry.
\newblock {\em Cahiers Topologie G\'eom. Diff\'erentielle Cat\'eg.},
  35(1):49--73, 1994.

\bibitem[GK00]{Grosse-Kloenne}
Elmar Grosse-Kl{\"o}nne.
\newblock Rigid analytic spaces with overconvergent structure sheaf.
\newblock {\em J. Reine Angew. Math.}, 519:73--95, 2000.

\bibitem[GK02]{Grosse-Kloenne2}
Elmar Grosse-Kl{\"o}nne.
\newblock Finiteness of de {R}ham cohomology in rigid analysis.
\newblock {\em Duke Math. J.}, 113(1):57--91, 2002.

\bibitem[GK04]{Grosse-Kloenne3}
Elmar Gro{\ss}e-Kl{\"o}nne.
\newblock De {R}ham cohomology of rigid spaces.
\newblock {\em Math. Z.}, 247(2):223--240, 2004.

\bibitem[HS10]{Arakelov-motivic-cohomology-I}
A.~{Holmstrom} and J.~{Scholbach}.
\newblock {Arakelov motivic cohomology I}.
\newblock {\em ArXiv e-prints}, December 2010.

\bibitem[Ill71]{Illusie-cotangent}
Luc Illusie.
\newblock {\em Complexe cotangent et d\'eformations. {I}}.
\newblock Lecture Notes in Mathematics, Vol. 239. Springer-Verlag, Berlin,
  1971.

\bibitem[Kar86]{Karoubi-K-theorie-multiplicative}
Max Karoubi.
\newblock {$K$}-th\'eorie multiplicative.
\newblock {\em C. R. Acad. Sci. Paris S\'er. I Math.}, 302(8):321--324, 1986.

\bibitem[KS06]{Kashiwara-Schapira-categories-and-sheaves}
Masaki Kashiwara and Pierre Schapira.
\newblock {\em Categories and sheaves}, volume 332 of {\em Grundlehren der
  Mathematischen Wissenschaften [Fundamental Principles of Mathematical
  Sciences]}.
\newblock Springer-Verlag, Berlin, 2006.

\bibitem[KV73]{Karoubi-Villamayor}
Max Karoubi and Orlando Villamayor.
\newblock {$K$}-th\'eorie alg\'ebrique et {$K$}-th\'eorie topologique. {II}.
\newblock {\em Math. Scand.}, 32:57--86, 1973.

\bibitem[Law79]{Lawvere-categorical-dynamics}
F.~William Lawvere.
\newblock Categorical dynamics.
\newblock In {\em Topos theoretic methods in geometry}, volume~30 of {\em
  Various Publ. Ser.}, pages 1--28. Aarhus Univ., Aarhus, 1979.

\bibitem[Law04]{Lawvere-functorial-semantics}
F.~William Lawvere.
\newblock Functorial semantics of algebraic theories and some algebraic
  problems in the context of functorial semantics of algebraic theories.
\newblock {\em Repr. Theory Appl. Categ.}, 5:1--121, 2004.
\newblock Reprinted from Proc. Nat. Acad. Sci. U.S.A. {{\bf{5}}0} (1963),
  869--872 [MR0158921] and {\it Reports of the Midwest Category Seminar. II},
  41--61, Springer, Berlin, 1968 [MR0231882].

\bibitem[{Lur}07]{Lurie-DAG-IV}
J.~{Lurie}.
\newblock {Derived Algebraic Geometry IV: Deformation Theory}.
\newblock {\em ArXiv e-prints}, September 2007.

\bibitem[Lur09a]{Lurie-DAG-V}
Jacob Lurie.
\newblock Derived {A}lgebraic {G}eometry {V}: {S}tructured {S}paces.
\newblock {\em Preprint}, 2009.

\bibitem[Lur09b]{Lurie-higher-algebra}
Jacob Lurie.
\newblock Higher {A}lgebra.
\newblock {\em Preprint}, 2009.

\bibitem[Lur09c]{Lurie-higher-topos-theory}
Jacob Lurie.
\newblock {\em Higher topos theory}, volume 170 of {\em Annals of Mathematics
  Studies}.
\newblock Princeton University Press, Princeton, NJ, 2009.

\bibitem[Lur11]{Lurie-DAG-IX}
Jacob Lurie.
\newblock Derived {A}lgebraic {G}eometry {IX}: {C}losed {I}mmersions.
\newblock {\em Preprint}, 2011.

\bibitem[{Mar}12]{Martin-overconvergent-subanalytic}
F.~{Martin}.
\newblock {Overconvergent subanalytic subsets in the framework of Berkovich
  spaces}.
\newblock {\em ArXiv e-prints}, November 2012.

\bibitem[Mer72]{Meredith}
David Meredith.
\newblock Weak formal schemes.
\newblock {\em Nagoya Math. J.}, 45:1--38, 1972.

\bibitem[MR13]{Milne-Ramachandran}
J.~{Milne} and N.~{Ramachandran}.
\newblock {Motivic complexes and special values of zeta functions}.
\newblock {\em ArXiv e-prints}, November 2013.

\bibitem[MV99]{Morel-Voevodsky}
Fabien Morel and Vladimir Voevodsky.
\newblock {${\bf A}^1$}-homotopy theory of schemes.
\newblock {\em Inst. Hautes \'Etudes Sci. Publ. Math.}, 90:45--143 (2001),
  1999.

\bibitem[MW68]{Monsky-Washnitzer}
P.~Monsky and G.~Washnitzer.
\newblock Formal cohomology. {I}.
\newblock {\em Ann. of Math. (2)}, 88:181--217, 1968.

\bibitem[Pau14]{Fred-Towards-the-maths-of-QFT}
Fr\'ed\'eric Paugam.
\newblock {\em Towards the mathematics of quantum field theory}, volume~59 of
  {\em Ergebnisse der Mathematik und ihrer Grenzgebiete}.
\newblock Springer Verlag, 2014.

\bibitem[Poi10]{Poineau1}
J{\'e}r{\^o}me Poineau.
\newblock La droite de {B}erkovich sur {$\bold Z$}.
\newblock {\em Ast\'erisque}, 334:viii+xii+284, 2010.

\bibitem[Poi13]{Poineau2}
J{\'e}r{\^o}me Poineau.
\newblock Espaces de {B}erkovich sur {$\bold Z$} : \'etude locale.
\newblock {\em Invent. Math.}, 194(3):535--590, 2013.

\bibitem[Por14]{Porta-these}
Mauro Porta.
\newblock {\em G\'eom\'etrie analytique d\'eriv\'ee}.
\newblock Th\`ese. universit\'e de Paris 7, 2014.

\bibitem[Rio10]{Riou-K-theory}
Jo{\"e}l Riou.
\newblock Algebraic {$K$}-theory, {${\bf A}^1$}-homotopy and {R}iemann-{R}och
  theorems.
\newblock {\em J. Topol.}, 3(2):229--264, 2010.

\bibitem[{Rob}12]{Robalo}
M.~{Robalo}.
\newblock {Noncommutative Motives I: A Universal Characterization of the
  Motivic Stable Homotopy Theory of Schemes}.
\newblock {\em ArXiv e-prints}, June 2012.

\bibitem[Rob14]{Robalo-these}
Marco Robalo.
\newblock {\em Th\'eorie homotopique motivique des espaces non commutatifs}.
\newblock Universit\'e de Montpellier 2s, Montpellier, 2014.
\newblock Th\`ese, universit\'e de Montpellier 2.

\bibitem[RS14]{Rosenschon-Srinivas}
Andreas Rosenschon and Vasudevan Srinivas.
\newblock {\em Etale motivic cohomology and algebraic cycles}.
\newblock Preprint, 2014.

\bibitem[{Sch}10]{Scholbach-special-L-values}
J.~{Scholbach}.
\newblock {Special L-values of geometric motives}.
\newblock {\em ArXiv e-prints}, March 2010.

\bibitem[{Sch}12a]{Arakelov-motivic-cohomology-II}
J.~{Scholbach}.
\newblock {Arakelov motivic cohomology II}.
\newblock {\em ArXiv e-prints}, May 2012.

\bibitem[Sch12b]{Scholze-perfectoid-spaces}
Peter Scholze.
\newblock Perfectoid spaces.
\newblock {\em Publ. Math. Inst. Hautes \'Etudes Sci.}, 116:245--313, 2012.

\bibitem[Sch13]{Scholze-p-adic-Hodge-theory}
Peter Scholze.
\newblock {$p$}-adic {H}odge theory for rigid-analytic varieties.
\newblock {\em Forum Math. Pi}, 1:e1, 77, 2013.

\bibitem[{Tam}11]{Tamme}
G.~{Tamme}.
\newblock {Karoubi's relative Chern character, the rigid syntomic regulator,
  and the Bloch-Kato exponential map}.
\newblock {\em ArXiv e-prints}, November 2011.

\bibitem[Tat71]{Tate2}
John Tate.
\newblock Rigid analytic spaces.
\newblock {\em Invent. Math.}, 12:257--289, 1971.

\bibitem[Tem04]{Temkin-local-properties-II}
M.~Temkin.
\newblock On local properties of non-archimedean spaces ii.
\newblock {\em Isr. J. of Math.}, 140(1):1--27, 2004.

\bibitem[{Tem}10]{Temkin-intro-Berkovich-spaces}
M.~{Temkin}.
\newblock {Introduction to Berkovich analytic spaces}.
\newblock {\em ArXiv e-prints}, October 2010.

\bibitem[TV02]{Toen-Vezzosi-K-theory}
B.~{Toen} and G.~{Vezzosi}.
\newblock {A remark on K-theory and S-categories}.
\newblock {\em ArXiv Mathematics e-prints}, October 2002.

\bibitem[TV09]{Toen-Vezzosi-Chern}
B.~{Toen} and G.~{Vezzosi}.
\newblock {Caract\`eres de Chern, traces \'equivariantes et g\'eom\'etrie
  alg\'ebrique d\'eriv\'ee}.
\newblock {\em ArXiv e-prints}, March 2009.

\bibitem[{Vez}14]{Vezzani-perfectoid}
A.~{Vezzani}.
\newblock {A motivic version of the theorem of Fontaine and Wintenberger}.
\newblock {\em ArXiv e-prints}, May 2014.

\end{thebibliography}

\end{document}